\documentclass[10 pt, reqno]{amsart}

\usepackage{amsthm}
\usepackage{amsmath}
\usepackage{amssymb}
\usepackage{amsfonts}
\usepackage{latexsym}
\usepackage{hyperref}
\usepackage{bbm}
\usepackage{bm}
\usepackage{comment} 
\usepackage{enumitem}
\usepackage{braket}
\usepackage{cleveref}
\usepackage[noadjust]{cite}
\usepackage{chngcntr}
\counterwithin*{equation}{section}
\usepackage{mathtools}

\usepackage[T1]{fontenc}
\usepackage[letterpaper, hmarginratio=1:1]{geometry}
\usepackage{color}
\usepackage{fancyhdr}
\usepackage{latexsym}

\newtheorem{ccounter}{ccounter}[section]
\newtheorem{thm}[ccounter]{Theorem}
\newtheorem{lem}[ccounter]{Lemma}
\newtheorem{cor}[ccounter]{Corollary}
\newtheorem{prop}[ccounter]{Proposition}
\newtheorem{ass}[ccounter]{Assumption}
\newtheorem{ex}[ccounter]{Example}

\theoremstyle{definition}
\newtheorem{defn}[ccounter]{Definition}
\newtheorem{rmk}[ccounter]{Remark}

\def\bet{\begin{thm}}
\def\eet{\end{thm}}
\def\bel{\begin{lem}}
\def\eel{\end{lem}}
\def\bas{\begin{ass}}
\def\eas{\end{ass}}
\def\bec{\begin{cor}}
\def\eec{\end{cor}}
\def\bed{\begin{defn}}
\def\eed{\end{defn}}
\def\bep{\begin{prop}}
\def\eep{\end{prop}}
\def\beq{\begin{equation}}
\def\eeq{\end{equation}}
\def\bea{\begin{equation*}}
\def\eea{\end{equation*}}
\def\brmk{\begin{rmk}}
\def\ermk{\end{rmk}}
\def\tr{\mathrm{tr}}
\def\bex{\begin{ex}}
\def\eex{\end{ex}}

\def\eer{\normalcolor}
\def\gstar{{ \gamma_z^\star }}
\def\rstar{{R_\star}}

\def\vstar{{{\mathcal R}_\star}}

\def\benr{\begin{enumerate}[label=(\roman*)]}
\def\eenr{\end{enumerate}}

\def\N{\mathbb{N}}
\def\bH{\mathbf{H}}
\def\C{\mathbb{C}}

\def\P{\mathbb{P}}
\def\E{\mathbb{E}}
\def\S{\mathbb{S}}
\def\B{\mathcal{B}}
\def\mfc{m_{\mathrm{fc} } }
\def\fc{m_{\mathrm{fc}}}

\def\one{{\mathbbm 1}}
\def\eps{\varepsilon}
\def\dsc{\varrho_{\mathrm{fc} }}
\def\Im{\operatorname{Im}}
\def\Re{\operatorname{Re}}

\newcommand{\bma}{\begin{bmatrix}}
\newcommand{\ema}{\end{bmatrix}}

\def\tr{\operatorname{Tr}}

\def\iu{\mathrm{i}}
\def\supp{\operatorname{supp}}
\def\sgn{\operatorname{sgn}}
\def\av{\operatorname{Av}}
\def\flat{\operatorname{Flat}}
\def\S{{\mathcal S}}
\def\L{{\mathcal L}}
\def\U{{\mathcal U}}
\def\B{{\mathcal B}}

\def\fa{{\widetilde d}}

\def\fc{{\mathfrak c}}
\def\fd{{\mathfrak d}}

\def\supp{\operatorname{supp}}
\def\B{{\mathcal B}}
\def\X{{\mathbf X}}
\def\R{{\mathbf {R}}}
\def\W{{\mathbf {W}}}
\def\q{{\mathbf {q}}}
\def\u{{\mathbf {u}}}
\def\e{{\mathbf {e}}}
\def\bA{{\mathbf {A}}}

\def\S{{\mathbf {S}}}

\def\eps{\varepsilon}
\def\epsilon{\varepsilon}
\def\cplus{{\mathbb H} }
\def\rplus{{\mathbb R_+} }
\newcommand{\dom}{\lesssim}

\newcommand{\threepartdef}[6]
{
	\left\{
		\begin{array}{lll}
			#1 & \mbox{if } #2 \\
			#3 & \mbox{if } #4 \\
			#5 & \mbox{otherwise.} #6
		\end{array}
	\right.
}
\DeclareMathOperator{\sign}{sign}

\begin{document}
\title{Eigenvector statistics of L\'evy matrices}
\author{Amol Aggarwal}
\thanks{A.A.\ is partially supported by NSF grant NSF DMS-1664619, the NSF Graduate Research Fellowship under grant DGE-1144152, and a Harvard Merit/Graduate Society Term-time Research Fellowship.}
\author{Patrick Lopatto}
\thanks{P.L.\ is partially supported by NSF grant DMS-1606305, NSF grant DMS-1855509, and the NSF Graduate Research Fellowship Program under grant DGE-1144152.}
\author{Jake Marcinek}

\begin{abstract}
		
		We analyze statistics for eigenvector entries of heavy-tailed random symmetric matrices (also called L\'{e}vy matrices) whose associated eigenvalues are sufficiently small. We show that the limiting law of any such entry is non-Gaussian, given by the product of a normal distribution with another random variable that depends on the location of the corresponding eigenvalue. Although the latter random variable is typically non-explicit, for the median eigenvector it is given by the inverse of a one-sided stable law. Moreover, we show that different entries of the same eigenvector are asymptotically independent, but that there are nontrivial correlations between eigenvectors with nearby eigenvalues. Our findings contrast sharply with the known eigenvector behavior for Wigner matrices and sparse random graphs.
	
\end{abstract}

\maketitle

{\setcounter{tocdepth}{1}
	\tableofcontents
}

\section{Introduction} \label{s:int}

A central tenet of random matrix theory is that the spectral properties of large random matrices should be independent of the distributions of their entries.
This phenomenon is known as \emph{universality}, and it originates from investigations of Wigner in the 1950s \cite{wigner1955characteristic}. One realization of this phenomenon is the \emph{Wigner--Dyson--Mehta conjecture}, which asserts the universality of local eigenvalue statistics of $N\times N$ real symmetric (or complex Hermitian) Wigner matrices in the bulk of the spectrum. Over the past decade, this conjecture has been resolved and generalized to an extensive collection of random matrix models \cite{bourgade2016fixed,erdos2010bulk,erdos2011universality,erdos2012local,erdos2017dynamical,erdos2012bulk,johansson2012universality,johansson2001universality,lee2014necessary,tao2012random,tao2011random, ajanki2017singularities,ajanki2017universality,ajanki2016local,che2017universality, bauerschmidt2017bulk,erdos2013spectral,erdos2012spectral,huang2015spectral,huang2015bulk,bourgade2018survey, bourgade2017universality, bourgade2018random, bourgade2019random, yang2018random, aggarwal2019bulk}.

Universality for the limiting distributions of eigenvector entries of Wigner matrices was proven more recently in \cite{bourgade2013eigenvector} by Bourgade and Yau. There, they introduced the eigenvector moment flow, a system of differential equations with random coefficients that govern the evolution of the moments of eigenvector entries of a matrix under the addition of Gaussian noise. 
Through a careful analysis of these dynamics, they prove asymptotic normality for the eigenvector entries of Wigner matrices. Extensions of this method later enabled the analysis of eigenvector statistics for sparse and deformed Wigner matrices in \cite{benigni2017eigenvectors, bourgade2017eigenvector}, and for other eigenvector observables in \cite{bourgade2018random,benigni2019fermionic}.

The matrices considered in these works all have entries of finite variance. The question of whether universality persists for matrices with entries of infinite variance was first raised in 1994 by Cizeau and Bouchaud \cite{cizeau1994theory}. Such matrices are now believed to be better suited for modeling heavy-tailed phenomena in physics \cite{bouchaud1995more, sornette2006critical,cizeau1993mean}, finance \cite{heiny2019eigenstructure,galluccio1998rational, bouchaud1997option, bun2017cleaning, bouchaud1998taming, bouchaud2009financial,laloux1999noise,laloux2000random}, and machine learning \cite{martin2018implicit,martin2019heavy,mahoney2019traditional}. They may therefore be regarded as exemplars of a new universality class for highly correlated systems. 

Concretely, \cite{cizeau1994theory} introduced a class of symmetric matrices called \emph{L\'evy matrices}, whose entries are random variables in the domain of attraction of an $\alpha$-stable law, and made a series of predictions regarding their spectral and eigenvector behavior. One such prediction concerns the global eigenvalue distribution. Given a L\'evy matrix $\mathbf{H}$, \cite{cizeau1994theory} conjectured that as $N$ tends to $\infty$, its empirical spectral distribution $\mu_\mathbf{H}$ should converge to a deterministic, heavy-tailed measure $\mu_\alpha$. This was later proven by Ben Arous and Guionnet in \cite{arous2008spectrum}. The measure $\mu_{\alpha}$ contrasts with the empirical spectral densities of the other finite variance models studied previously, which are all compactly supported. 
 
 The main predictions in \cite{cizeau1994theory} concerned eigenvector (de)localization and local eigenvalue statistics for L\'evy matrices. The more recent predictions of Tarquini, Biroli, and Tarzia are slightly different and can be stated as follows \cite{tarquini2016level}. First, for $\alpha \in [1,2)$, all eigenvectors of $\textbf{H}$ corresponding to finite eigenvalues are completely delocalized, and these eigenvalues exhibit local statistics matching those of the Gaussian Orthogonal Ensemble (GOE) as $N$ tends to infinity.  Second, for $\alpha \in (0,1)$, there exists a \emph{mobility edge} $E_\alpha > 0$ separating a regime with delocalized eigenvectors and GOE local statistics around small energies $E \in (-E_{\alpha} , E_{\alpha})$ from one with localized eigenvectors and Poisson local statistics around large energies $E \notin [-E_{\alpha}, E_{\alpha}]$.

 Partial forms of eigenvector localization and delocalization in the regimes listed above were proven by Bordenave and Guionnet in \cite{bordenave2013localization,bordenave2017delocalization}. Later, in \cite{aggarwal2018goe}, complete eigenvector delocalization and convergence to GOE local statistics were established in the two delocalized phases mentioned above. For $\alpha \in (0, 1)$, these delocalization results were only proven for eigenvalues in a non-explicit neighborhood of $0$, and extending them to the conjectured mobility edge remains open.

The transition identified by \cite{tarquini2016level,cizeau1994theory} is also known as an \emph{Anderson transition}, a concept fundamental to condensed matter physics that describes phenomena such as metal--insulator transitions \cite{abrahams201050,anderson1958absence,anderson1978local, mott1949basis,mott1987mobility}. It is widely believed to exist in the context of random Schr\"odinger operators \cite{abou1973selfconsistent, abou1974self, aizenman2011absence, aizenman2013resonant, bapst2014large}, but rigorously establishing this statement has
remained an important open problem in mathematical physics for decades. L\'evy matrices provide one of the few examples of a random matrix ensemble for which such a transition
 is also believed to appear. Partly for this reason, they have been a topic of intense study for both mathematicians and physicists over the past 25 years \cite{auffinger2009poisson,arous2008spectrum,benaych2014central,benaych2014central,belinschi2009spectral,biroli2007top,bordenave2011spectrum,bordenave2017delocalization,bordenave2013localization,burda2006random,soshnikov2004poisson,tarquini2016level,biroli2007extreme, auffinger2016extreme, monthus2016localization}.

While the above results and predictions address (de)localization of L\'evy matrix eigenvectors, little is known about refined properties of their entry fluctuations. In \cite{benaych2014central}, Benaych-Georges and Guionnet showed that averages of $O (N^2)$ of these eigenvector entries converge to Gaussian processes, after scaling by $N^{-1 / 2}$. However, until now, we have not been aware of any results or predictions concerning fluctuations for individual entries of L\'evy matrix eigenvectors.

In this paper we establish several such results, which in many respects contrast with their known counterparts for Wigner matrices (and all other random matrix models for which the eigenvector entry distributions have previously been identified). We establish, for almost all $\alpha <2$, the following statements concerning the unit (in $L^2$) L\'evy eigenvectors $\mathbf{u}_k = \big( u_k (1), u_k (2), \ldots , u_k (N) \big)$ whose associated eigenvalues $\lambda_k \approx E$ are sufficiently small.
\begin{enumerate}
\item \label{novelty} An eigenvector entry $\sqrt{N} u_j(i)$ is not asymptotically normal: its square converges to $\mathcal{N}^2 \cdot \mathcal{U}_{\star} (E)$ as $N$ tends to $\infty$, where $\mathcal N$ is a standard normal random variable and $\mathcal{U}_{\star} (E)$ is an independent (non-constant and typically non-explicit) random variable that depends on $E$.
\item Different entries of the same eigenvector are asymptotically independent.

\item Entries of different eigenvectors with the same index are not asymptotically independent: if $k_1, k_2, \ldots , k_n \in [1, N]$ and $i \in [1, N]$ are indices such that $\lambda_{k_n} \approx E$, then the vector $\big( N \u_{k_1}(i )^2, N \u_{k_2}(i )^2, \dots , N \u_{k_n}(i )^2 \big)$ converges to $\big( {\mathcal N}^2_1 \cdot \mathcal{U}_{\star} (E), {\mathcal N}^2_2\cdot \mathcal{U}_{\star} (E), \dots,{\mathcal N}^2_n \cdot \mathcal{U}_{\star} (E) \big)$, where the $\mathcal N_j$ are i.i.d.\ standard Gaussians that are independent from $\mathcal{U}_{\star} (E)$.

\item  The law of $\mathcal{U}_{\star} (0)$ is given explicitly as the inverse of a one-sided $\frac{\alpha}{2}$-stable law. In particular, all asymptotic moments of the median eigenvector are also explicit.
\end{enumerate}

To contextualize our results, we recall the asymptotic normality statements for Wigner eigenvectors proved in \cite{bourgade2013eigenvector}. First, an individual eigenvector entry converges to a standard normal random variable. Second, different entries of the same eigenvector are asymptotically independent. Third, the same is true for entries of different eigenvectors with the same index. Our results show that, although the second of these phenomena persists in the L\'{e}vy case, the first and third do not. We further note that while \cite{aggarwal2018goe} showed that L\'evy matrices exhibit GOE local statistics at small energy, which only differ through a rescaling factor corresponding to the eigenvalue density, the eigenvector statistics follow a family of distinct random variables that vary in the energy parameter $E \in \mathbb R$. This is again unlike the Wigner eigenvectors, which display the same Gaussian statistics throughout the spectrum.

Our work confronts one of the major differences between L\'evy matrices and Wigner matrices: the non-concentration of the resolvent entries. In the Wigner case, these diagonal resolvent entries converge to a deterministic quantity (the Stieltjes transform of the semicircle law) as $N$ tends to $\infty$. However, for L\'evy matrices $\mathbf{H}$, the diagonal entries $G_{jj}(z)$ of the resolvent $\mathbf{G}(z) = ( \mathbf{H} - z)^{-1}$ converge for fixed $z \in \mathbb{H}$ to nontrivial (complex) random variables $R_\star(z)$ as $N$ tends to $\infty$ \cite{bordenave2011spectrum}. This is the mechanism that generates the non-Gaussian scaling limit for the limiting eigenvector entries. Indeed, the random variable $\mathcal{U}_{\star} (E)$ mentioned above is defined as a (multiple of a) weak limit of $R_\star (E + \iu \eta)$, as $\eta$ tends to $0$.

Our proof strategy is dynamical. First, we define a matrix $\mathbf{X}$, which is coupled to the L\'evy matrix $\mathbf{H}$ and obtained by setting all sufficiently small (in absolute value) entries of $\mathbf{H}$ to zero. We also introduce the Gaussian perturbation $\mathbf{X}_s = \mathbf{X} + \sqrt{s} \mathbf{W}$, where $\mathbf{W}$ is a GOE and $s \ll 1$. Under a certain choice of $s = t$, we are able to show that the eigenvector statistics of $\mathbf{H}$ are asymptotically the same as those of $\mathbf{X}_t$. Second, we identify the moments of the eigenvector entries of $\mathbf{X}_t$ in terms of entries of the resolvent matrix $\mathbf{R}(t, z) = ( \mathbf{X}_t - z )^{-1}$, where $z = E + \iu \eta$ and $\eta$ tends to $0$ as $N$ tends to $\infty$. Third, we compute the large $N$ limit of these resolvent entries and deduce the above claims from the behavior of the resulting scaling limits. We now describe the steps of our argument, and their associated challenges, in more detail.

1. The first step is a comparison of eigenvector statistics, which has been achieved before for Wigner matrices with entries matching in the first four moments \cite{knowles2013eigenvector,tao2012random}. 
However, these results do not apply to L\'evy matrices, since the second moments of their entries are infinite. Instead, we use the comparison scheme introduced in \cite{aggarwal2018goe} that conditions on the locations of the large entries of $\mathbf{H}$ and matches moments between the small entries of $\mathbf{H}$ and the Gaussian perturbation $\sqrt{t}\mathbf{W}$ in $\mathbf{X}_t = \mathbf{X} + \sqrt{t} \mathbf{W}$; these have all moments finite by construction. While \cite{aggarwal2018goe} considered the comparison for resolvent elements, we apply it to general smooth functions of the matrix entries and write the eigenvector statistics as such functions using ideas from \cite{bourgade2017eigenvector,huang2015bulk}.

2. The second step uses the eigenvector moment flow to show that moments of the eigenvector entries of $\textbf{X}_t$ approximate moments of the $\Im \textbf{R} (t, z)$ entries. As in \cite{bourgade2013eigenvector, bourgade2017eigenvector}, a primary idea here is to apply the maximum principle to show under these dynamics that eigenvector moment observables equilibriate after a short period of time to a polynomial in the entries of $\Im \textbf{R} (t, z)$. However, unlike in those previous works, the entries of $\Im \textbf{R} (s, z)$ do not concentrate and therefore might be unstable under the dynamics. To address this, we condition on the initial data $\mathbf{X}_0$, which one might hope renders the $R_{jj} (s, z)$ essentially constant under the flow $\mathbf{X}_s$ for $s \ll 1$. Unfortunately, we cannot show this directly, and in fact it appears that these resolvent entries can be unstable in the beginning of the dynamics even after this conditioning. Therefore, we run the flow for a short time $\tau$ before beginning the main analysis. This has a regularizing effect and ensures the stability of the resolvent entries of $\mathbf{X}_s$ for $s \in [\tau, t]$. Our analysis then proceeds by running the dynamics for a further amount of time $t - \tau \gg N^{-1 / 2}$ to prove convergence to equilibrium, given this initial regularity. 

3. The third step asymptotically equates moments of $\Im R_{ii}(t, z)$ with those of $\Im R_\star (z)$, as $\eta = \Im z$ tends to $0$ and $N$ tends to $\infty$ simultaneously. To analyze the former, we first use the Schur complement formula and a certain integral identity to express arbitrary moments of $\Im R_{ii} (t, z)$ through the $\frac{\alpha}{2}$-th moments of (real and imaginary parts of) $R_{ii} (t, z)$, as in \cite{aggarwal2018goe,bordenave2017delocalization,bordenave2013localization}. Next, using a local law from \cite{aggarwal2018goe}, we approximate these $\frac{\alpha}{2}$-th moments by corresponding ones for $R_\star(z)$. To analyze the moments of $\Im R_\star(z)$, we use the same integral identity and a recursive distributional equation for $R_\star(z)$ from \cite{bordenave2011spectrum} to express them through $\frac{\alpha}{2}$-moments of (real and imaginary parts of) $R_\star(z)$. We then observe the two expressions are equal as $N$ tends to $\infty$. 

The remainder of this paper is organized as follows. In \Cref{s:results} we state our results in detail. In \Cref{s:proofs} we give a full proof outline and establish our main results, assuming several preliminary claims which are shown in the remainder of the paper. \Cref{s:preliminary} recalls results on L\'evy matrices from previous works that are required for the argument. \Cref{s:comparison} details the comparison part of the argument. \Cref{s:dynamics} analyzes the eigenvector moment flow. \Cref{s:scaling} computes the scaling limits of the resolvent entries mentioned above.
\Cref{s:appendixA} provides some preliminary results needed in the previous sections, and \Cref{s:appendixdist} addresses convergence in distribution. In \Cref{matrixeigenvectors}, we discuss quantum unique ergodicity (QUE) for eigenvectors of L\'{e}vy matrices, whose analogue for Wigner matrices was established in \cite{bourgade2013eigenvector}. \\

\noindent {\bf Acknowledgments.} The authors thank Giulio Biroli, Jiaoyang Huang, and Horng-Tzer Yau for helpful conversations. We also thank an anonymous referee for their comments.

\section{Results} \label{s:results}

\subsection{Definitions}

Denote the upper half plane by $\mathbb{H} = \left\{ z \in \mathbb{C}: \Im z > 0\right\}$. Set $\rplus = [ 0,\infty)$, set $\mathbb{K} = \{ z \in \mathbb{C}: \Re z > 0 \}$, and set $\mathbb{K}^+ = \overline{\mathbb{K} \cap \mathbb{H}}$ to be the closure of the positive quadrant of the complex plane. We also let $\mathbb{S}^1 = \{ z\in \C : |z|  =1 \}$ be the unit circle and define $\mathbb{S}_+^1 = \overline{\mathbb{K}^+ \cap \mathbb{S}}$.

Fix a parameter $\alpha \in (0, 2)$, and let $\sigma > 0$ and $\beta \in [-1, 1]$ be real numbers. A random variable $Z$ is a \emph{$(\beta, \sigma)$ $\alpha$-stable law} if it has the characteristic function 
\begin{align}
\label{betasigmaalphalaw}
\mathbb{E} \left[ e^{ \iu t Z} \right] = \exp \Big( - \sigma^{\alpha} |t|^{\alpha} \big( 1 - \iu \beta  \sgn (t) u \big) \Big), \quad \text{for all $t \in \mathbb{R}$},
\end{align}

\noindent where $u = u_{\alpha} = \tan \big( \frac{\pi \alpha}{2} \big)$ if $\alpha \ne 1$ and $u = u_1 = -\frac{2}{\pi} \log |t|$ if $\alpha = 1$. Note $\beta = 0$ ensures that $Z$ is symmetric. The case $\beta =1$ is known as a \emph{one-sided $\alpha$-stable law} and is always positive.

We now define the entry distributions we consider in this paper. Our proofs and results should also apply to wider classes of distributions, but we will not pursue this here (see the similar remark in \cite[Section 2]{aggarwal2018goe} for more on this point).

\bed
	
	\label{momentassumption}
	
	Let $Z$ be a $(0, \sigma)$ $\alpha$-stable law with
	\begin{flalign}
	\label{stable}
	\sigma = \left( \displaystyle\frac{\pi}{2 \sin \big( \frac{\pi \alpha}{2} \big) \Gamma (\alpha)} \right)^{1 / \alpha} > 0.
	\end{flalign} 
	Let $J$ be a symmetric\footnote{By \emph{symmetric}, we mean that $J$ has the same law  as $-J$.} random variable (not necessarily independent from $Z$) such that $\E [J^2] < \infty$, $Z+J$ is symmetric, and 
\begin{flalign}
\label{probabilityxij}
\frac{C_1}{\big( |t| + 1 \big)^\alpha} \le \P \big[ |Z + J | \ge t   \big] \le \frac{C_2}{\big( |t| + 1 \big)^\alpha} \quad \text{for each $t \ge 0$ and some constants $C_1, C_2 > 0$.}
\end{flalign} 

	\noindent Denoting $\mathfrak{z} = Z + J$, the symmetry of $J$ and the condition $\mathbb{E} [J^2] < \infty$ are equivalent to imposing a coupling between $\mathfrak{z}$ and $Z$ such that $\mathfrak{z} - Z$ is symmetric and has finite variance, respectively. 
	
	For each positive integer $N$, let $\{ H_{ij} \}_{1 \le i \le j \le N}$ be mutually independent random variables that each have the same law as $N^{-1 / \alpha} (Z + J) = N^{-1 / \alpha} \mathfrak{z}$. Set $H_{ij} = H_{ji}$ for each $i, j$, and define the $N \times N$ random matrix $\textbf{H} = \textbf{H}_N = \{ H_{ij} \} = \{ H_{i, j}^{(N)} \}$, which we call	 an \emph{$\alpha$-L\'{e}vy matrix}.

\eed

The $N^{-1 / \alpha}$ scaling of the entries $H_{ij}$ is different from the usual $N^{-1 / 2}$ scaling for Wigner matrices. It makes the typical row sum of $\textbf{H}$ of order one.  The constant $\sigma$ is chosen so that our notation is consistent with previous works \cite{arous2008spectrum,bordenave2013localization,bordenave2017delocalization}, but can be altered by rescaling $\mathbf{H}$ without affecting our main results.

By \cite[Theorem 1.1]{arous2008spectrum}, the empirical spectral distribution of $\bH$ converges to a deterministic measure that we denote $\mu_\alpha$, which is absolutely continuous with respect to the Lebesgue measure and symmetric about $0$. We denote its probability density function and Stieltjes transform by $\varrho_\alpha(x)$ and 
\beq m_\alpha(z) = \int_{\mathbb R} \frac{\varrho_\alpha(x)\, dx}{ x -z}, \eeq
defined for $z\in \cplus$, respectively. 

The Siteltjes transform $m_\alpha(z)$ may be characterized as the solution to a certain self-consistent equation \cite[Section 3.1]{bordenave2013localization}. We note it here, although we will not need this representation for our work. For any $z \in \mathbb{H}$, define the functions $\varphi = \varphi_{\alpha,z}\colon \mathbb{K} \rightarrow \mathbb{C}$ and $\psi = \psi_{\alpha,z}\colon \mathbb{K} \rightarrow \mathbb{C}$ by
\beq
\label{psi}
\varphi_{\alpha,z}(x) = \frac{1}{\Gamma(\alpha/2)} \int_{\mathbb{R}_+} t^{\alpha/2 -1} e^{\iu t z} e^{-\Gamma(1-\alpha/2) t^{\alpha/2} x} dt, \quad \psi_{\alpha, z}(x) = \int_{\mathbb{R}_+} e^{\iu t z} e^{-\Gamma( 1 - \alpha/2) t^{\alpha/2} x} dt, 
\eeq
for any $x \in \mathbb{K}$. For each $z \in \mathbb{H}$, there exists a unique solution $y = y(z) \in \mathbb{K}$ to the equation $y(z) = \varphi_{\alpha,z} \big( y(z) \big)$. Then, the Stieltjes transform $m_{\alpha} (z): \mathbb{H} \rightarrow \mathbb{H}$ is defined by setting $m_{\alpha}(z) = \iu \psi_{\alpha,z} \big( y(z) \big)$.
 
We recall that, like any Stieltjes transform of an absolutely continuous measure, $\Im m_\alpha(z)$ extends  to the real line with 
\beq \lim_{\eta \rightarrow 0} \Im m_\alpha(E + \iu \eta ) =  \pi \varrho_\alpha(E)
\eeq
 for $E \in \mathbb R$. It is known that $\varrho_{\alpha} (x) \sim \frac{\alpha}{2 x^{\alpha + 1}}$ as $x$ tends to $\infty$ \cite[Theorem 1.6]{bordenave2011spectrum}.
\bed The classical eigenvalue locations $\gamma_i = \gamma^{(\alpha)}_i$ for $\varrho_\alpha(x)$ are defined by the quantiles
\begin{equation}
\label{gammaalphai}
\gamma_i = \inf \left\{ y \in \mathbb{R} :  \int_{-\infty}^y \varrho_\alpha(x)\, dx \ge \frac{i}{N} \right\}.
\end{equation}
\eed

Given a random matrix $\mathbf A$, it is common to study its resolvent $(\mathbf A - z)^{-1}$. Contrary to those for the Wigner model, the diagonal entries $G_{ii}(z)$ of the resolvent $\mathbf G(z) = (\mathbf{H} - z)^{-1}$ of a L\'evy matrix do not converge to a constant value but instead converge to a nontrivial limiting distribution as $N$ tends to infinity and $z\in \mathbb H$ remains fixed. This was shown in \cite{bordenave2011spectrum}, where the limit $\rstar(z)$ was identified as the resolvent of a random operator defined on a space known as the Poisson Weighted Infinite Tree \cite{aldous1992asymptotics, aldous2004objective}, which is a weighted and directed rooted tree, evaluated at its root. We note the basic construction here and refer to \cite[Section 2.3]{bordenave2011spectrum} for details.

Set $d\nu = (1/2)\, d\mu$, where $\mu$ is the Lebesgue measure on $\mathbb R$. The vertex set of the tree is given by ${V = \bigcup_{k \in \mathbb N } \mathbb N^k}$, where the root is $\mathbb N^0 = \varnothing$, and the children of $\mathbf v \in \mathbb N^k$ are denoted ${(\mathbf v ,1), (\mathbf v, 2), \ldots \in \mathbb N^{k+1}}$. To determine the weights, let $\{\Xi_{\mathbf v}\}_{\mathbf v \in V}$ be a collection of independent Poisson point processes with intensity measure $\nu$ on $\mathbb R$. Let $\Xi_{\varnothing} = \{ y_1, y_2, \dots \}$ be ordered so that $|y_1| \le |y_2| \le \cdots$, and set $y_i$ to be the weight of edge connecting $\varnothing$ to the vertex $(i)$. This process is repeated for all vertices so that, for any  $\textbf{v} \in \mathbb N^k$, the edge between vertices $\mathbf v$ and $(\mathbf v, i)$ is weighted with the value $y_{(\mathbf v , i)}$, where $\Xi_{\mathbf v} = \{ y_{(\mathbf v ,1) }, y_{(\mathbf v ,2)}, \dots \}$ is labeled so that $|y_{(\mathbf v ,1)}| \le |y_{(\mathbf v ,2)}| \le \cdots$. 

Let $\mathcal{F}$ be the (dense) subset of $L^2(V)$ of vectors with finite support and, for any $\textbf{v} \in V$, let $\delta_{\textbf{v}} \in \mathcal{F}$ denote the unit vector supported on $\textbf{v}$. Then, define the linear operator $\mathbf T\colon \mathcal{F} \rightarrow L^2(V)$ by setting
\beq
\langle \delta_{\mathbf v}, \mathbf T \delta_{\mathbf w} \rangle = \threepartdef
{\sign(y_{\mathbf w}) | y_{\mathbf w}|^{-1/\alpha}} {\mathbf w = (\mathbf v, k) \mbox{ for some } k,}
{\sign(y_{\mathbf v}) | y_{\mathbf v}|^{-1/\alpha}} {\mathbf v = (\mathbf w, k) \mbox{ for some } k,}
{0} {}
\eeq
We identify $\mathbf T$ with its closure, which is self-adjoint \cite[Section 2.3]{bordenave2011spectrum}. It can be considered a weak limit of the matrix $\mathbf H$, as $N$ tends to $\infty$. 

\bed 

\label{limitresolvent}

For any $z \in \mathbb{H}$, we define $\rstar(z)\colon \cplus \rightarrow \cplus$ to be the resolvent entry $\langle \delta_{\varnothing}, (\mathbf T - z )^{-1} \delta_{\varnothing} \rangle$.

\eed

	A key property of $\rstar (z)$, shown in \cite{bordenave2011spectrum}, is that it satisfies a ``recursive distributional equation,'' which may be considered as a limiting analogue of the usual Schur complement formula.
	
	\bel[{\cite[Theorem 4.1]{bordenave2011spectrum}}]
	
	\label{d:rde} 
	
	Denote by $\{\xi_k \}_{k \ge 1}$ a Poisson process on $\mathbb R_+$ with intensity measure $\big( \frac{\alpha}{2} \big) x^{-\alpha/2 -1}\, dx$. For any $z \in \mathbb{H}$, the random variable $\rstar(z) \colon \mathbb H \rightarrow \mathbb H$ satisfies the equality in law
	\beq\label{e:rde}
	\rstar (z) \overset{d}{=} - \Bigg( z + \displaystyle\sum_{k=1}^\infty \xi_k R_k(z) \Bigg)^{-1},
	\eeq
	where $\big( R_k (z) \big)_{k \ge 1}$ is an i.i.d.\ sequence with distribution $\rstar (z)$ independent from the process $\{\xi_k \}_{k \ge 1}$.
	\eel
\begin{rmk}
To see heuristically why \eqref{e:rde} holds, one applies the Schur complement formula to the matrix $\mathbf{H}$ to compute a diagonal resolvent element $G_{ii}(z)$, and then takes the large $N$ limit after ignoring the off-diagonal terms (which are negligible):
\begin{equation}
G_{ii} \approx - \frac{1}{z + \sum_{j \neq i }^N h_{ij}^2 G_{jj}^{(i)}} \approx - \frac{1}{z + \sum_{j = 1}^N h_{ij}^2 G_{jj}}.
\end{equation}
Here $G_{jj}^{(i)}$ is the resolvent of the $(N-1)\times (N - 1)$ matrix given by $\mathbf H$ with the $i$-th row and column removed. In the second statement we implement the (standard) approximation $G_{jj}^{(i)} \approx G_{jj}$.

 At this point we see a difference with the derivation of the semicircle law for Wigner matrices (as given in, for example, \cite{benaych2016lectures}): the sum involving the $h_{ij}^2$ no longer concentrates due to its heavy-tailed nature. Instead, the $\{ h_{ij}^2 \}$ converge in the large $N$ limit to the Poisson point process $\{ \xi_j \}$, which yields (2.9).
\end{rmk}

For any $z \in \mathbb{H}$, define the $N \times N$ matrix $\mathbf G (z) = \big\{ G_{ij} (z) \big\}$ by $\mathbf G (z)= ( \mathbf H - z)^{-1}$, which is the resolvent of $\mathbf H$. It is known from \cite[Section 2]{bordenave2011spectrum} that any diagonal entry $G_{jj} (z)$ converges to $\rstar (z)$ in distribution for fixed $z \in \mathbb{H}$, as $N$ tends to $\infty$.

Next, we require the following result and definition concerning the limit of $\rstar (z)$ as $\Im z$ tends to $0$. The following proposition will be proved in \Cref{s:scaling} below.

\bep\label{p:tightness} There exists a (deterministic) countable set $\mathcal A \subset (0,2)$ with no accumulation points in $(0,2)$ such that the following two statements hold. First, for all $\alpha \in (0,2)\setminus \mathcal A$, there exists a constant $c=c(\alpha)>0$ such that, for every real number $E \in [-c,c]$, the sequence of random variables $\big\{ \Im \rstar (E + \iu \eta) \big\}_{\eta >0}$ is tight as $\eta$ tends to $0$. Second, for any fixed $p \in {\mathbb N}$, all limit points ${\mathcal R}(E)$ of this sequence under the weak topology have the same moment $\E \big[ {\mathcal R}(E)^{p} \big]$.
\eep

\noindent The set $\mathcal A$ is non-explicit and originally appeared in \cite{bordenave2017delocalization} from an application of the implicit function theorem for Banach spaces to a certain self-consistent equation. In our context, $\mathcal{A}$ will come from a local law, given by \Cref{l:Xlocallaw} below.

\bed\label{d:arbitrary} Let $\vstar (E)$ be an arbitrary limit point (under the weak topology) as $\eta$ tends to $0$ of the sequence $\big\{ \Im \rstar (E + \mathrm{i} \eta) \big\}_{\eta > 0}$. By Proposition 2.4 and Prokhorov's theorem, there exists at least one. Given $\vstar (E)$, define the random variable $\mathcal{U}_{\star} (E) = \big( \pi \varrho_{\alpha} (E) \big)^{-1} \vstar (E)$. 

\eed

\noindent We also need the following definition to state our results.

\bed\label{d:momentconvergence}
Let $\mathbf w = (w_i)_{1 \le i \le n} \in \mathbb R^n$ be a random vector and $\mathbf w^{(j)} = (w^{(j)}_i)_{1 \le i \le n}$, defined for $j \ge 1$, be a sequence of random vectors in $\mathbb R^n$. We say that $\mathbf w^{(j)}$ \emph{converges in moments} to $\mathbf w$ if for every polynomial $P\colon \mathbb R^n \rightarrow \mathbb R$ in $n$ variables, we have
\beq
\lim_{N\rightarrow \infty} \E \Big[  P \big( \mathbf w^{(N)} \big) \Big] =  \E \big[ P ( \mathbf w ) \big].
\eeq
\eed

\subsection{Results}\label{s:ourresults}

In this section, we state our results, which are proved in \Cref{s:proofs}. Our first, \Cref{t:main2}, identifies the joint moments of different entries of the same eigenvector. Our second, \Cref{t:thedistribution1}, does this for the same entries of different eigenvectors. We let $\lambda_1 (\textbf{H}) \le \lambda_2 (\textbf{H}) \le \cdots \le \lambda_N (\textbf{H})$ denote the eigenvalues of $\textbf{H}$ in non-decreasing order and, for each $k \in [1, N]$, we write $\mathbf{u}_k = \big( u_k (1), u_k (2), \ldots , u_k (N) \big)$ for a unit eigenvector of $\mathbf H$ corresponding to $\lambda_k (\textbf{H})$. 

In the theorem statements, certain index parameters (for instance $i$ and $k$) may depend on $N$. For brevity, we sometimes suppress this dependence in the notation, writing for example $\u_k$ instead of $\u_{k(N)}$. Throughout, we recall the countable set $\mathcal{A} \subset (0, 2)$ from \Cref{p:tightness}. 

\bet

\label{t:main2}

 For all $\alpha \in (0,2)\setminus \mathcal A$, there exists a constant $c=c(\alpha)>0$ such that the following holds. Fix an integer $n > 0$ (independently of $N$) and index sequences $ \left\{ i_j ( N ) \right\}_{1\le j \le n}$ such that for every $N$, $ \left\{ i_j ( N ) \right\}_{1\le j \le n}$ are distinct integers in $[1, N]$. Further let $k=k(N) \in [1, N]$ be an index sequence such that $\lim_{N\rightarrow\infty} \gamma_k = E$ for some $E \in [-c,c]$. Then the vector 
\beq
\left(N \u_{k}(i_1)^2, N \u_{k}(i_2)^2, \dots, N \u_{k}(i_n)^2\right)
\eeq
converges in moments to 
\beq
\big( {\mathcal N}^2_1 \cdot \mathcal{U}_1 (E), {\mathcal N}^2_2 \cdot \mathcal{U}_2 (E), \dots, {\mathcal N}^2_n \cdot \mathcal{U}_n (E) \big),
\eeq
where the ${\mathcal N}_j$ are independent, identically distributed (i.i.d.) standard Gaussians and the $\mathcal{U}_j (E)$ are i.i.d.\ random variables with law $\mathcal{U}_{\star} (E)$ that are independent from the $\mathcal{N}_j$. 
\eet

\bet

\label{t:thedistribution1} 

For all $\alpha \in (0,2)\setminus \mathcal A$, there exists a constant $c=c(\alpha)>0$ such that the following holds. Fix an integer $n >0$ (independently of $N$) and index sequences $\{k_j(N)\}_{1 \le j \le n}$  such that for every $N$, $\{k_j(N)\}_{1\le j \le n}$ are distinct integers in $[1, N]$ and $| k_1 - k_j| < N^{1/2}$ for each $j \in [2, n]$. Suppose that $\lim_{N\rightarrow\infty} \gamma_{k_1} = E$ for some $E\in [-c,c]$. Further let $i = i(N) \in [1, N]$ be an index sequence. Then the  vector
\beq \label{e:davector} \left(N \u_{k_1}(i )^2, N \u_{k_2}(i )^2, \dots , N \u_{k_n}(i )^2\right) \eeq
converges in moments to
\beq
\big({\mathcal N}^2_1 \cdot \mathcal{U}_{\star} (E), {\mathcal N}^2_2\cdot \mathcal{U}_{\star} (E), \dots,{\mathcal N}^2_n \cdot \mathcal{U}_{\star} (E)  \big),
\eeq
where the $\mathcal N_j$ are i.i.d.\ standard Gaussians that are independent from $\mathcal{U}_{\star} (E)$.
\eet

\Cref{t:main2} shows that entries of the same eigenvector are asymptotically independent, as in the Wigner case \cite[Corollary 1.3]{bourgade2013eigenvector}. However, unlike in the Wigner case \cite[Theorem 1.2]{bourgade2013eigenvector}, \Cref{t:thedistribution1} indicates that entries of different eigenvectors with the same index can be asymptotically  correlated. This can be seen by taking $n = 2$ and $k_2 = k_1 + $1 in that result, in which case $N \u_{k_1} (m)^2$ and $N \u_{k_2} (m)^2$ are correlated through $\mathcal{U}_{\star} (E)$.

For almost all $E \in [-c, c]$, the random variable $\mathcal{U}_{\star} (E)$ is not explicit. However, as a consequence of \cite[Theorem 4.3]{bordenave2011spectrum}, an exception occurs at $E = 0$, where $\mathcal{U}_{\star} (0)$ is given by the inverse of a stable law. In this case, the $n = 1$ cases of \Cref{t:main2} and \Cref{t:thedistribution1} reduce to the following corollary.

\bec\label{c:median} Retain the notation of \Cref{t:main2}. Choose $k$ so that $E=0$, and set $n=1$ and $m=i_1$. Then $N \u_{k}(m)^2$ converges in moments to 
\beq
\frac{ 1 }{ \Gamma\left( 1 + \frac{2}{\alpha} \right) }\cdot {\mathcal N}^2 \cdot  \vartheta,
\eeq
where  $\mathcal N$ is a standard Gaussian and $\vartheta$ is independent with law $S^{-1}$, where $S$ is a $(1, 1)$ $\frac{\alpha}{2}$-stable law.
\eec

The non-triviality of the random variable $\vartheta$ shows that the entries of $\textbf{u}_k$ are asymptotically non-Gaussian; this is again different from the eigenvector behavior in the Wigner case. It is natural to wonder whether $\vstar(E)$ is non-constant for $E\neq 0$. As a consequence of the last statement of \Cref{l:boundarylimit} below, for all $p \in \mathbb N$, the moments $\E \big[ \big( \mathcal{R}_{\star} (E)\big)^p \big]$ are continuous in $E$, for $|E|$ sufficiently small. This implies that moments of $\mathcal{U}_{\star} (E)$ are non-constant for all $E$ in a neighborhood of $0$, so the eigenvectors of $\textbf{H}$ corresponding to sufficiently small eigenvalues are also non-Gaussian. 

It is also natural to ask whether our results hold for convergence in distribution. In the case $\alpha \in (1, 2) \setminus \mathcal{A}$ we will address this in \Cref{s:cvgdist} through \Cref{p:dconvergence} by studying the rate of growth of the moments of the limiting distribution. If $\alpha < 1$, then the moments of $\mathcal{U}_{\star} (E)$ grow too quickly for this to determine the law of $\mathcal N^2 \cdot \mathcal{U}_{\star} (E)$.

Finally, we note that we consider the squared eigenvector entries $\mathbf u_k(i)^2$ to avoid ambiguity in the choice of sign for $\mathbf u_k(i)$, since given an eigenvalue $\lambda_k$ of a real symmetric matrix and a corresponding eigenvector $\mathbf v_k$, the vector $-\mathbf v_k$ is also an eigenvector. In the context of L\'evy random matrices, if one chooses this sign independently with probability $1/2$ for each possibility, then our methods show the above results hold with the conclusion of \Cref{t:main2} replaced by the convergence in moments of $\big(\sqrt{N} \u_{k}(i_1), \sqrt{N} \u_{k}(i_2), \dots, \sqrt{N} \u_{k}(i_n) \big)$ to
$\big( {\mathcal N}_1 \cdot \mathcal{U}^{1/2}_1 (E), {\mathcal N}_2 \cdot \mathcal{U}^{1/2}_2 (E), \dots, {\mathcal N}_n \cdot \mathcal{U}^{1/2}_n (E) \big)$, where the $\mathcal N_k$ remain i.i.d.\ standard Gaussians, and similarly for \Cref{t:thedistribution1}.

\section{Proofs of main results}\label{s:proofs}

Assuming some claims proven in later parts of this paper, we will establish in this section the results stated in \Cref{s:ourresults}. This will proceed through the following steps.
\begin{enumerate}
\item We define a matrix $\X$, obtained by setting the small entries of the original L\'{e}vy matrix $\textbf{H}$ to zero, and the Gaussian perturbation $\X_s = \X + \sqrt{s} \mathbf W$, where $\mathbf W$ is a GOE matrix. For a specific choice of $s = t$, with $N^{-1/2} \ll t \ll 1$, we show as \Cref{l:maincomparison} that the eigenvector statistics of $\mathbf H$ (corresponding to small eigenvalues) are approximated by those of $\X_t$. 
\item We show as \Cref{t:dynamics} that moments of the eigenvector entries of $\X_t$ (corresponding to small eigenvalues) can be identified through resolvent entries of $\textbf{X}_t$.
\item We compute as \Cref{t:main1} the limits of these resolvent entries as $N$ and $\eta$ tend to $\infty$ and $0$, respectively.  
\end{enumerate}

In \Cref{s:proofsmain}, we prove \Cref{t:main2}, \Cref{t:thedistribution1}, and \Cref{c:median}, given that the results enumerated above and \Cref{p:tightness} hold. The remaining sections of the paper verify these prerequisite results.

\subsection{Notation}

\label{Notation}

Throughout, we write $C$ for a large constant and $c$ for a small constant. These may depend on other constants and may change line to line, but only finitely many times, so that they remain finite. We say $X \ll Y$ if there exists a small constant $c > 0 $ such that $N^c |X| \le Y$. Constants in this paper may depend on the constant $c>0$ implicit in the claim $X \ll Y$, but we suppress this in the notation. We write $X \dom Y$ if there exists $C> 0$ such that $|X| \le C Y$; we also say $X \lesssim_u Y$, or equivalently $X = O_u (Y)$, if $|X| \le C_u |Y|$ for some constant $C_u >0$ depending on a parameter $u$. 

In what follows, for any function (or vector) $f$, we let $\|  f \|_\infty$ denote the $L^\infty$-norm of $f$.  We also denote $\text{Mat}_{N \times N}$ by the set of $N \times N$ real, symmetric matrices. Given $\mathbf M \in \text{Mat}_{N \times N}$, we denote its eigenvalues by $\lambda_1 (\textbf{M}), \lambda_2 (\textbf{M}), \ldots , \lambda_N (\textbf{M})$ in non-decreasing order. We further let $\u_i(\mathbf M)$ denote the unit eigenvector corresponding to the eigenvalue $\lambda_i(\mathbf M)$ for each $i$. We also make the following definition.

\bed We say a (sequence of) vectors $\q = \q (N) = (q_1, q_2, \dots q_N) \in \mathbb R^N$ has \emph{stable support} if there exists a constant $C > 0$ such that the set $ \{ (i, q_i)  : q_i \neq 0 \}$ does not change for $N > C$. We let $\supp \q = \{ i : q_i \neq 0 \}$ denote the support of $\q$. 
\eed

We next introduce the notion of overwhelming probability.

\bed
\label{d:op}
 We say that a family of events $\{\mathcal F(u)\}$ indexed by some parameter(s) $u \in U^{(N)}$, where  $U^{(N)}$ is a parameter set which may depend on $N$, holds with \emph{overwhelming probability} if, for any $D>0$, there exists $N\big(D, U^{(N)}\big) > 0$ such that for $ N \ge N\big(D, U^{(N)}\big)$,
\begin{equation}
 \inf_{u \in U^{(N)}} \P \big( \mathcal F(u) \big) \ge 1 - N^{-D}.
\end{equation}
\eed

Next, given $\alpha \in (0,2)$ we may select positive real numbers $b = b(\alpha) > 0$; $\nu = \nu(\alpha) > 0$; $\mathfrak{a} = \mathfrak{a} (\alpha) > 0$; and $\rho = \rho(\alpha) > 0$ such that 
\beq
\label{e:parameters}
\nu = \frac{1}{\alpha} - b > 0; \quad \frac{1}{4-\alpha} < \nu < \frac{1}{4-2\alpha}; \quad (2 - \alpha) \nu  <  \mathfrak a <  \frac{1}{2}; \quad 0 < \rho < \nu < \frac{1}{2}; \quad \alpha \rho < (2 - \alpha) \nu.
\eeq

\noindent These parameters will be fixed throughout the paper, and we will let other constants depend on them (and on $\alpha$), even when not explicitly noted. We always assume $\alpha \in (0,2 ) \setminus \mathcal A$, where $\mathcal A$ is the set from \Cref{l:Xlocallaw} below (or, equivalently, the one from \Cref{p:tightness}).

\subsection{Comparison}\label{s:comparisonstatements}
We first recall the definition of the removed model $\X$ from \cite[Definition 3.2]{aggarwal2018goe}. 

\bed\label{abremovedmatrix}
	 Recalling the notation of \Cref{momentassumption}, let $X = (Z + J) \one_{|Z + J| > N^b}$. We call $X$ the \emph{$b$-removal of $Z + J$}. Further, let $\{ X_{ij} \}_{1 \le i \le j \le N}$ be mutually independent random variables that each have the same law as $N^{-1 / \alpha} X$. Set $X_{ij} = X_{ji}$ for each $1 \le j < i \le N$, and define the $N \times N$ symmetric matrix $\mathbf{X} = \{ X_{ij} \}$. We call $\mathbf{X}$ a \emph{$b$-removed $\alpha$-L\'{e}vy matrix}. 	
	 \eed

We also recall a resampling and coupling of $\X$ and $\bH$ that was described in \cite[Section 3.3.1]{aggarwal2018goe}. 

\bed

\label{abcpsidefinitions}
	
	\label{chipsi}
	We define mutually independent random variables $\{ a_{ij}, b_{ij}, c_{ij}, \psi_{ij}, \chi_{ij} \}_{1\le i \le j \le N}$ as follows. Let $\psi_{ij}$ and $\chi_{ij}$ denote $0$--$1$ Bernoulli random variables with distributions
	\begin{flalign}\label{e:phidef}
	\mathbb{P} \big[ \psi_{ij} = 1 \big] = \mathbb{P} \big[ |H_{ij}| \ge N^{-\rho} \big], \quad\mathbb{P} \big[ \chi_{ij} = 1 \big] = \displaystyle\frac{\mathbb{P} \big[ |H_{ij}| \in [N^{-\nu}, N^{-\rho}) \big]}{\mathbb{P} \big[ |H_{ij}| < N^{-\rho} \big]}.
	\end{flalign}

	\noindent Additionally, let $a_{ij}$, $b_{ij}$, and $c_{ij}$ be random variables such that 
	\begin{flalign}
	\P(a_{ij} \in I ) = & \frac{\P \Big[ H_{ij} \in  (-N^{-\nu}, N^{-\nu})  \cap I \Big]}{\P \big[ |H_{ij}| < N^{-\nu}) \big]}, \quad \P [c_{ij} \in I] = \frac{\P \Big[ H_{ij} \in \big( (-\infty, -N^{-\rho}] \cup [N^{-\rho}, \infty) \big) \cap I \Big] }{\P \big[ |H_{ij}| \ge N^{-\rho} \big]}, \\
	& \quad\quad \P(b_{ij} \in I ) = \frac{\P \Big[ H_{ij} \in \big( (- N^{-\rho}, - N^{-\nu}]  \cup [N^{-\nu}, N^{-\rho}) \big)  \cap I \Big]}{\P \big[ |H_{ij}| \in  [N^{-\nu}, N^{-\rho}) \big]},
	\end{flalign} 
	
	\noindent for any interval $I\subset \mathbb R$. For each $1 \le j < i \le N$, define $a_{ij} = a_{ji}$ by symmetry, and similarly for each of $b_{ij}$, $c_{ij}$, $\psi_{ij}$, and $\chi_{ij}$.
	
\eed

Because $a_{ij}$, $b_{ij}$, $c_{ij}$, $\psi_{ij}$, and $\chi_{ij}$ are mutually independent, $H_{ij}$ has the same law as \beq(1 - \psi_{ij}) (1 - \chi_{ij}) a_{ij} + (1 - \psi_{ij}) \chi_{ij} b_{ij} + \psi_{ij} c_{ij}\eeq and $X_{ij}$ has the same law as $(1 - \psi_{ij}) \chi_{ij} b_{ij} + \psi_{ij} c_{ij}$. Therefore, although the random variables $H_{ij} \mathbf{1}_{|H_{ij}| \ge N^{-\rho}}$, $H_{ij} \mathbf{1}_{N^{-\nu} \le |H_{ij}| < N^{-\rho}}$, and $H_{ij} \mathbf{1}_{|H_{ij}| < N^{-\nu}}$ are correlated, this decomposition expresses their dependence through the Bernoulli random variables $\psi_{ij}$ and $\chi_{ij}$. 

\bed
	
	\label{hsumabc}
	
	For each $1 \le i, j \le N$, set
	\begin{flalign}
	\label{abc}
	A_{ij} = (1 - \psi_{ij}) (1 - \chi_{ij}) a_{ij}, \quad B_{ij} = (1 - \psi_{ij}) \chi_{ij} b_{ij}, \quad C_{ij} = \psi_{ij} c_{ij},
	\end{flalign}
	
	\noindent and define the four $N \times N$ matrices $\mathbf{A} = \{ A_{ij} \}$, $\mathbf{B} = \{ B_{ij} \}$, $\mathbf{C} = \{ C_{ij} \}$, and $\Psi = \{ \psi_{ij} \}$.
\eed	

	For the remainder of the paper we sample $\mathbf{H}$ and $\mathbf{X}$ by setting $\mathbf{H} = \mathbf{A} + \mathbf{B} + \mathbf{C}$ and $\mathbf{X} = \mathbf{B} + \mathbf{C}$, inducing a coupling between the two matrices. We commonly refer to $\Psi$ as the \emph{label} of $\mathbf{H}$ (or of $\mathbf{X}$). Defining $\mathbf{H}$ and $\mathbf{X}$ in this way ensures that their entries have the same laws as in \Cref{momentassumption} and \Cref{abremovedmatrix}, respectively. 
	
	For any $s \in \mathbb R_+$, we define the matrix $\X_s \in \text{Mat}_{N \times N}$ by setting 
 \begin{equation}
\label{e:Xt}
\X_s  = \X +  \W_s, 
\end{equation}
where $\W_s= \big( w_{ij} (s) \big)_{1 \le i , j \le N} \in \text{Mat}_{N \times N}$ and $w_{ij}$ are mutually independent Brownian motions with symmetry constraint $w_{ij} = w_{ji}$ and variance $(1+\one_{i = j}) N^{-1}$. %We write $\mathbf{u}_k(s)$ for the $k$-th $L^2$-normalized eigenvector of $\mathbf H$.

We now make a specific choice of the time $t$ to enable our comparison argument. Define $t$ by
\begin{flalign}
\label{t}
t = N \mathbb{E} \Big[ H_{11}^2  \one_{ |H_{11}| < N^{-\nu}} \big| |H_{11}| < N^{-\rho} \Big] = \displaystyle\frac{N \mathbb{E} \big[ H_{11}^2  \one_{ |H_{11}| < N^{-\nu}} \big]}{\mathbb{P} \big[ |H_{11}| < N^{-\rho} \big]}. 
\end{flalign}
The following estimate is \cite[Lemma 3.5]{aggarwal2018goe} and can be quickly deduced from \eqref{probabilityxij} and \eqref{t}. 
\bel[{\cite[Lemma 3.5]{aggarwal2018goe}}]

\label{t0estimate} 

Under the choice of \eqref{t}, we have that
\begin{flalign}
\label{c1c2t0}
 c N^{(\alpha -2) \nu} \le t \le C N^{(\alpha -2) \nu}.
\end{flalign}
\eel

Observe in particular that \eqref{c1c2t0} implies that $N^{-1 / 2} \ll t \ll 1$, by the third inequality in \eqref{e:parameters}. The next theorem is proved in \Cref{s:comparison} and completes the first step of the outline given in the beginning of \Cref{s:proofs}.

\bet

\label{l:maincomparison} 

There exist constants $c_1, c_2 > 0$ such that the following holds. Let $t$ be as in \eqref{t}, $P\colon \mathbb R^n \rightarrow \mathbb R^n$ be a polynomial in $n$ variables, and $\q \in \mathbb{R}^N$ be a unit vector with stable support. Then there exists a constant $C= C \big( P, |\supp \q|\big)>0$ such that, for indices $i_1, i_2, \dots, i_n \in \big[ (1/2- c_1) N, (1/2 + c_1) N \big]$, 
\beq
\Bigg| \E \Bigg[  P \bigg( \Big( N \big\langle \q, \u_{i_k}(\X_t) \big\rangle^2 \Big)_{1\le k \le n} \bigg)  \Bigg]  - \E \Bigg[  P \bigg( \Big( N \big\langle \q, \u_{i_k}( \mathbf H ) \big\rangle^2 \Big)_{1\le k \le n} \bigg)  \Bigg] \Bigg| \le C N^{-c_2}.
\eeq
\eet

\begin{rmk}

 One might wonder why we implement the three-tiered composition $\mathbf H = \mathbf A + \mathbf B + \mathbf C$ in \Cref{abcpsidefinitions} instead of a two-tiered one. The reason is that the parameters $\nu$ and $\rho$ play different roles. In particular, $\nu$ dictates the size of the $\mathbf{A}$-entries, which we would like to be small so as to bound moments in our comparison argument (see \eqref{530} below); thus, we should take $\nu$ sufficiently large. However, $\rho$ dictates the threshold for the $\mathbf{C}$-entries in our matrix, which are very large in the sense that their moments are unbounded; we would thus like $\rho$ sufficiently small to ensure that there are not too many of them (see \eqref{e:hoeffding} below). Using a two-tier composition would force $\nu = \rho$, which would not be feasible for all $\alpha \in (0, 2)$, which is why we implement a three-tiered one (see for example \eqref{e:largefield}). (However, the choice $\nu = \rho$ works when $\alpha \in (1, 2)$.)
\end{rmk}

\subsection{Short-time universality}

\label{s:dynamicsstatements}

For each integer $k \in [1, N]$ and real number $s \ge 0$, abbreviate $\lambda_k(s) = \lambda_k (\textbf{X}_s)$, and set $\bm \lambda(s) = \big( \lambda_1 (s), \lambda_2 (s), \dots, \lambda_N (s) \big)$. Further let $\textbf{u}_k (s) \in \mathbb{R}^N$ denote the unit eigenvector of $\textbf{X}_s$ associated with $\lambda_k (s)$, and set $\u (s) = \big( \u_1 (s), \dots , \u_N(s)\big)$.

Next, for any unit vector $\q \in \mathbb R^N$ and $k \in [1, n]$, set $z_k (s) = z_k (s, \textbf{q}) = \sqrt{N} \big\langle \q , \textbf{u}_k(s) \big\rangle$. For any integer $m \ge 1$; indices $i_1, i_2, \ldots , i_m \in [1, N]$; and integers $j_1, j_2, \ldots , j_m \ge 0$, define  
\beq\label{e:Qdef}
Q^{j_1,\dots, j_m}_{i_1, \dots , i_m} (s) = \prod_{l=1}^m z_{i_l} (s)^{2j_l} \prod_{l=1}^m a(2j_l)^{-1}, \quad \text{where}\quad a(2j) = (2j - 1)!!.
\eeq
The normalization factors $a(2j)$ are chosen because they are the moments of a standard Gaussian.

To any index set $\{ (i_1, j_1), \dots , (i_m, j_m)\}$ with distinct $i_k \in [1, N]$ and positive $j_k$, we may associate the vector $\bm \xi = (\xi_1, \xi_2, \dots, \xi_N) \in \mathbb N^N$ with $\xi_{i_k} = j_k$ for $1 \le k \le m$ and $\xi_p =0$ for $p \notin \{ i_1, \dots, i_m\}$. We think of $\bm \xi$ as a particle configuration on the integers, with $j_k$ particles at site $i_k$ for all $k$ and zero particles on the sites not in $\{ i_1, \dots, i_m\}$. We call the set $\{ i_1, \dots, i_m\}$ the support of $\bm \xi$, denoted $\supp \bm \xi$. Denote $\mathcal N (\bm \xi) = \sum_{j=1}^m j_k$, the total number of particles. The configuration $\bm \xi^{ij}$ is defined as the result of moving one particle in $\bm \xi$ from $i$ to $j$, that is, if $i \ne j$ then $\xi^{ij}_k$ equals $\xi_k + 1$, $\xi_k - 1$, or $\xi_k$ for $k = j$, $k = i$, and $k \notin \{ i, j \}$, respectively. Under this notation, we define an observable $F_s (\bm \xi)$ by the expectation 
\beq\label{e:observable}
F_s (\bm \xi) = \E \big[ Q^{j_1,\dots, j_m}_{i_1, \dots , i_m} (s)  \big].
\eeq

 Now fix $\fc \in \mathbb{R}_{> 0}$, later chosen to be sufficiently small. Recalling $\mathfrak{a}$ from \eqref{e:parameters} and $t$ from \eqref{t}, define
\begin{equation}
\label{e:etadefine}
\psi  = N^\fc, \quad \eta = N^{-\mathfrak a}\psi, \quad \text{so that} \quad N^{-1/2} \ll \eta \ll t,
\end{equation}

\noindent where the last inequality in \eqref{e:etadefine} follows from the third bound in \eqref{e:parameters}, \eqref{c1c2t0}, and the fact that $\fc$ is small.  For each $s \in \mathbb{R}_{> 0}$ we define the resolvent $\textbf{R} (s, z)$, Stieltjes transform $m_N (s, z)$ of $\textbf{X}_s$, and the expectation of $m_N (s,z)$ by
\beq 
\label{rsz}
 \R(s,z) = ({\X_s - z })^{-1}, \qquad m_N (s, z) = N^{-1} \tr \textbf{R} (s, z), \qquad \widehat m_N (s, z) = \E \big[ m_N (s,z) \big].
\eeq
We also define the (random) empirical spectral measure for $\mathbf X_s$ by
\beq\label{xsempiricalmeasure}
\mu_s = \frac{1}{N} \sum_{i=1}^N \delta_{\lambda_i(s)},
\eeq
where $\delta_x$ is the discrete probability measure that places all its mass at $x$. Further, we define $\widehat \mu_s =  \E[ \mu_s ]$ and observe that the Stieltjes transform of $\widehat \mu_s$ is $\widehat m_N (s,z)$.  The classical
eigenvalue locations for $\widehat \mu_s$ are given by
\begin{equation}
\label{hatgammaalphai}
\widehat \gamma_i(s) = \inf \left\{ y \in \mathbb{R} :  \widehat \mu_s\big( (-\infty, y] \big) \ge \frac{i}{N} \right\}.
\end{equation}

\noindent Recalling the specific choice of $t$ from \eqref{t}, we abbreviate $\widehat \gamma_i = \widehat \gamma_i(t)$.

The following theorem is proved in \Cref{s:dynamics} and completes the second step of the above outline.  We recall from \eqref{gammaalphai} the notation $\gamma_k = \gamma^{(\alpha)}_k$.
\bet

\label{t:dynamics} %Let $\mathcal E$ be the event defined in \Cref{l:goodset}. 

Fix $m \in \mathbb N$, let $\q \in \mathbb R^N$ be a unit vector with stable support, and let $t$ be the time defined in \eqref{t}. There exist constants $c_1 > 0$, $c_2 = c_2 (m) >0$, and $C=C (m, |\supp \q|) >0$ such that, if $\mathfrak{c} < c_2$, then
\begin{equation}
\label{ftexpectation}
\max_{\substack{\bm \xi \colon \mathcal N (\bm \xi) = m \\ \supp \bm \xi \in  [ (1/2 - c_1) N  , (1/2 + c_1) N  ]}} \left| F_t (\bm \xi) - \E \left[\prod_{k=1}^N \left( \frac{\Im \big\langle \q , \R ( t , \widehat \gamma_{k}  + \iu \eta ) \q \big\rangle }{\Im m_\alpha ( \gamma_k   + \iu \eta )  }  \right)^{\xi_k} \right] \right| \le C N^{-c_2},
\end{equation}
where $\bm \xi = (\xi_1, \xi_2, \ldots , \xi_N)$.
\eet

\begin{rmk}
It is natural to ask whether the results in this section, such as \Cref{t:dynamics}, hold for a broader class of vectors $\q$ than those with stable support. Similar results were proved for finite variance sparse matrices in \cite{bourgade2017eigenvector}. Addressing more general $\q$ would require isotropic versions of several of the preliminary results used in the proof of \Cref{t:dynamics} (such as \eqref{e:deloc}), which are not currently available in the literature. We therefore do not pursue this question here.
\end{rmk}

\subsection{Scaling limit}\label{s:scalingstatements}

The next theorem will be proven in \Cref{s:scaling} and establishes the scaling limit of the quantity compared to $F_t$ on the left side of \eqref{ftexpectation}, completing step 3 of the above outline. Here, we recall $\mathcal{U}_{\star} (E)$ from \Cref{d:arbitrary}. 

\bet\label{t:main1} 

There exist constants $c_1, c_2 > 0$ such that the following holds. Fix an integer $n >0$ (independently of $N$) and index sequences $\{k_j(N)\}_{1 \le j \le n}$  such that for every $N$, $\{k_j(N)\}_{1\le j \le n}$ are distinct integers in $[1, N]$ and $| k_1 - k_j| < N^{1/2}$ for each $j \in [2, n]$. Let $\q = ( q_1, \dots, q_N) \in \mathbb{R}^N$ be a unit vector with stable support; let $t$ be as in \eqref{t}; and assume that $\lim_{N \rightarrow \infty} \gamma_{k_1}  = E$, for some $E \in [-c_1, c_1]$, and $\mathfrak{c} < c_2$. Then the vector 
\beq
\left( \frac{\Im \big\langle \q , \R ( t , \widehat \gamma_{k_j}  + \iu \eta ) \q \big\rangle }{\Im m_\alpha ( \gamma_{k_j}   + \iu \eta ) }\right)_{1 \le j \le n}
\eeq
converges in moments to 
\beq\label{e:claim1}
(1, 1, \ldots , 1) \cdot \displaystyle\sum_{i \in \supp \textbf{\emph{q}}} q_i ^2 \mathcal{U}_i (E),
\eeq
where the random variables $\mathcal{U}_{i}(E )$ are independent and identically distributed with law $\mathcal{U}_{\star}(E )$.
\eet

\subsection{Proofs}\label{s:proofsmain}

In this section we establish \Cref{t:thedistribution1}, \Cref{t:main2}, and \Cref{c:median}.

\begin{proof}[Proof of \Cref{t:thedistribution1}.]
	
	By symmetry, we may suppose $i=1$ in the theorem statement. Recalling $t$ from \eqref{t} and applying \Cref{l:maincomparison} with $\mathbf q = \e_1= (1, 0, 0, \dots, 0)$ gives
	\beq
	\label{hestimateu}
	\lim_{N\rightarrow \infty} \Bigg| \E \Bigg[  P \bigg( \Big(N  \big\langle \u_{k_j}(\textbf{H}) , \e_1 \big\rangle^2  \Big)_{1\le j \le n}  \bigg) \Bigg] - \E \Bigg[  P \bigg( \Big(N  \big\langle \u_{k_j}(\textbf{X}_t) , \e_1 \big\rangle^2  \Big)_{1\le j \le n}  \bigg) \Bigg] \Bigg| = 0,
	\eeq
	
	\noindent Next, \Cref{t:dynamics} yields
	\beq
	\label{xtestimateu}
	\lim_{N\rightarrow \infty} \left| \E \Bigg[  P \bigg( \Big(N  \big\langle \u_{k_j}(\X_t) , \e_1 \big\rangle^2  \Big)_{1\le j \le n}  \bigg) \Bigg] - \E \left[ P \left( \left( \mathcal N^2_j \cdot  \frac{\Im R_{11} ( t, \widehat \gamma_{k_j}  + \iu \eta ) }{ \Im m_\alpha ( \gamma_{k_j}  + \iu \eta)} \right)_{1 \le j \le n} \right) \right] \right| = 0,
	\eeq
	
	\noindent where the $\mathcal N_j$ are i.i.d.\ standard Gaussians that are independent from $\Im R_{11} ( t, \widehat{\gamma}_{k_j}   + \iu \eta )$. Here we used \eqref{e:Qdef} and the fact that $a(2j) = \E \left[ \mathcal{N}^{2j} \right]$ for a standard Gaussian $\mathcal N$. 
	
	Now the theorem follows from \eqref{hestimateu}, \eqref{xtestimateu}, and \Cref{t:main1}.\end{proof}

\begin{proof}[Proof of \Cref{t:main2}.]

%Let $t$ be the time given in \eqref{t}. 

By symmetry, we may suppose that $i_j = j$ for each $j\in [ 1, n ]$. Let \beq
\mathbf v = \big( \mathcal{U}_1(E), \mathcal{U}_2(E), \dots, \mathcal{U}_n(E) \big)
\eeq
be a vector of i.i.d.\ random variables with distribution $\mathcal{U}_{\star} (E)$, where $\mathcal{U}_{\star} (E)$ is as in \Cref{d:arbitrary}. 

For any vector $\q$ with stable support such that $q_i = 0$ for $i\notin [1, n ]$, let $\mathbf{w} \in \mathbb{R}^N$ denote the vector 
$ \mathbf{w} = ( q_1^2, q_2^2, \dots , q_N^2).$
 Fix $m\in \mathbb N$, recall $a (2m) = (2m - 1)!!$ from \eqref{e:Qdef}, abbreviate $\u_k = \u_k(\mathbf H)$, and consider the polynomial
\beq
Q(q_1,\dots, q_n ) = \E \Big[  \big( N  \langle \q, \u_k \rangle^2 \big)^m  \Big]  - a(2m) \E \big[ \langle \mathbf{w}, \mathbf v \rangle ^{m}  \big].
\eeq

\noindent Then together \Cref{l:maincomparison}, \Cref{t:dynamics}, and (the $n = 1$ case of) \Cref{t:main1} imply for any unit vector $\textbf{q} \in \mathbb{R}^N$ with $\supp \textbf{q} \subseteq \{ 1, 2, \ldots,  n \}$ that 
\beq\label{q0estimate2}
\lim_{N\rightarrow \infty} Q(q_1,\dots, q_n ) = 0.
\eeq

\noindent Here we recalled \eqref{e:Qdef} and the fact that $a(2j) = \E \left[ \mathcal{N}^{2j} \right]$ for a standard Gaussian $\mathcal N$. 
Now observe that $Q$ is a polynomial of degree $2m$ in the $q_i$, that is, there exists coefficients $B_{\textbf{d}} \in \mathbb{R}$ such that
\begin{flalign}
\label{q0estimate} 
Q (q_1, q_2, \ldots , q_n) = \displaystyle\sum_{|\textbf{d}| = 2m} B_{\textbf{d}} \displaystyle\prod_{j = 1}^n q_j^{d_j},
\end{flalign}

\noindent where $\textbf{d} = (d_1, d_2, \ldots , d_n) \in \mathbb{Z}_{\ge 0}^n$ is summed over all $n$-tuples of nonnegative integers with $|\textbf{d}| = \sum_{j = 1}^n d_j = 2m$. Thus, since \eqref{q0estimate2} holds for all $(q_1, q_2, \ldots , q_n)$ with $\sum_{j = 1}^n q_j^2 = 1$, we have 
\begin{flalign}
 \displaystyle\lim_{N \rightarrow \infty} \displaystyle\max_{|\textbf{d}| = 2m} \left| B_{\textbf{d}} \right| = 0,
\end{flalign}

\noindent where again $\textbf{d}$ ranges over all $n$-tuples of nonnegative integers summing to $2m$. In particular, fixing some $n$-tuple $(m_1, m_2, \ldots , m_n)$ of nonnegative integers summing to $m$ and taking $\textbf{d} = (2m_1, 2m_2, \ldots , 2m_n)$ gives  	
\beq
\displaystyle\frac{(2m)!}{\prod_{j = 1}^n (2m_i)!} \lim_{N\rightarrow\infty} \E\left[  \prod_{j = 1}^n \big( N \u_k(j)^2 \big)^{m_{j}} \right] = \displaystyle\frac{m! (2m - 1)!!}{\prod_{j = 1}^n m_i!} \E\left[ \prod_{j = 1}^n \mathbf v(j)^{m_j} \right], 
\eeq

\noindent which implies that 
\beq\label{e:allm}
 \lim_{N\rightarrow\infty} \E\left[  \prod_{j = 1}^n \big( N \u_k(j)^2 \big)^{m_{j}} \right] = \E\left[ \prod_{j = 1}^n  a (2m_j) \mathbf v(j)^{m_j} \right].
\eeq

\noindent This yields the desired conclusion, since \eqref{e:allm} holds for all $(m_1, m_2, \ldots , m_n) \in \mathbb{Z}_{\ge 0}^n$ and $\mathbb{E} [\mathcal{N}^{2j}] = a (2j)$, for any integer $j \ge 0$, if $\mathcal{N}$ is a standard Gaussian random variable. 
\end{proof}

\begin{proof}[Proof of \Cref{c:median}.] By \cite[Theorem 1.6(ii)]{bordenave2011spectrum}, 
\beq
\label{global0}
\varrho_\alpha(0) =  \frac{1}{\pi}\Gamma\left( 1 + \frac{2}{\alpha}  \right) \left(  \frac{\Gamma \big( 1 - \frac{\alpha}{2} \big)}{\Gamma \big( 1 + \frac{\alpha}{2} \big)} \right)^{1/\alpha}.
\eeq
By \cite[Lemma 4.3(ii)]{bordenave2011spectrum}, $\vstar(0)$ has the same law as $\Upsilon^{-1}$, where $\Upsilon$ is a one-sided $\frac{\alpha}{2}$-stable law with Laplace transform
\beq
\E \big[ \exp(-t \Upsilon) \big] = \exp \Bigg(-t^{\alpha/2} \bigg( \frac{\Gamma\left( 1 + \frac{\alpha}{2}\right)}{\Gamma\left( 1 - \frac{\alpha}{2}  \right)} \bigg)^{1/2} \Bigg), \quad \text{for $t \ge 0$.}
\eeq

\noindent Since a $(1, 1)$ $\frac{\alpha}{2}$-stable law $S$ has Laplace transform $\exp\left(-t^{\alpha/2} \right)$, $\Upsilon$ has the same law as \beq \left(\frac{\Gamma\left( 1 + \frac{\alpha}{2}\right)}{\Gamma\left( 1 - \frac{\alpha}{2}  \right)} \right)^{1/\alpha} S,\eeq so the conclusion follows from \Cref{t:main2}, \eqref{global0}, and the fact (see \Cref{d:arbitrary}) that $\mathcal{U}_{\star} (E) = \big( \pi \varrho_{\alpha} (E) \big)^{-1} \mathcal{R}_{\star} (E)$.
\end{proof}

\section{Preliminary results}

\label{s:preliminary}

In this section we collect several miscellaneous known results that will be used throughout the paper. After recalling general estimates and identities on resolvent matrices in \Cref{EstimatesResolvent}, we state several facts about the density $\rho_{\alpha}$ in \Cref{EstimatesMeasure}. In \Cref{s:removed}, we recall several results about the removed model $\textbf{X}_s$. Finally, in \Cref{s:hgammaestimates} we recall properties of a certain matrix interpolating between $\textbf{X}_0$ and $\textbf{H}$.

\subsection{Resolvent identities and estimates}

\label{EstimatesResolvent}

For any invertible $\textbf{K}, \textbf{M} \in \text{Mat}_{N \times N}$, we have 
\begin{flalign}
\label{ab}
\textbf{K}^{-1} - \textbf{M}^{-1} = \textbf{K}^{-1} (\textbf{M} - \textbf{K}) \textbf{M}. 
\end{flalign}

\noindent Next, assume $z = E + \mathrm{i} \eta \in \mathbb{H}$ and $\textbf{K} = \{ K_{ij} \} =(\textbf{M} - z)^{-1}$. Then, we have the bound 
\begin{flalign}
\label{kij}
\displaystyle\max_{1 \le i, j \le N} |K_{ij}| \le \frac{1}{\eta},
\end{flalign}

\noindent and the Ward identity
\begin{flalign}
\label{kijsum} 
\displaystyle\sum_{j = 1}^N |K_{ij}|^2 = \displaystyle\frac{\Im K_{ii}}{\eta}.
\end{flalign}

 \subsection{The density $\varrho_\alpha$}

\label{EstimatesMeasure}

The following properties of the density $\varrho_\alpha$ are proved in \Cref{s:appendixA}; here, we recall the $\gamma_i$ from \eqref{gammaalphai}. %The continuity estimate is a small improvement on \cite[Proposition 1.1]{belinschi2009spectral}, which states that $\rho_\alpha(x)$ is continuous on $\mathbb R \setminus \{ 0 \}$. 

\bel\label{l:rhocontinuity}\label{l:mabounds}
There exists a (deterministic) countable set $\mathcal A \subset (0,2)$ with no accumulation points in $(0,2)$ and constants $C, c>0$  such that the following statements hold for $\alpha \in (0,2) \setminus \mathcal A$. 
\begin{enumerate}
\item For real numbers $E_1, E_2 \in [-c, c]$, 
\beq\label{e:rhoacontinuity}
\big| \varrho_\alpha(E_1) - \varrho_\alpha (E_2) \big| \le C|E_1 - E_2|, \qquad c \le \varrho_\alpha(E_1) \le C.
\eeq
\item For real numbers $E_1, E_2 \in [-c, c]$, and any $\eta >0$, we have
% If $|E_1 - E_2 | \le c/10$ and $|E_1| \le c/2 $, 
%then for any $\eta >0$ we have 
\beq\label{e:macontinuity}
	\big| \Im m_\alpha(E_1 + \iu \eta ) - \Im m_\alpha(E_2 + \iu \eta) \big| \le C| E_1 - E_2| + C \eta.  %, \quad \big| m_N(E_1 + \iu \eta ) - m_N(E_2 + \iu \eta) \big| \le C| E_1 - E_2|
	\eeq

\item 
	For real numbers $|E| < c $ and $\eta \in (0, c ]$, and any integer $j \in \big[ (1/2 - c) N, (1/2 + c) N \big]$,  
	\beq\label{e:mabounds}
	c \le \big| \Im m_\alpha( E + \iu \eta ) \big| \le C, \quad c \le \big| \Im m_\alpha( \gamma_j + \iu \eta ) \big| \le C.
	\eeq

\end{enumerate}
	\eel

\subsection{Removed model}

\label{s:removed}

In this section we recall several results concerning the resolvent $\textbf{R} (s, z)$ and Stieltjes transform $m_N (s, z)$ of $\textbf{X}_s$ (recall \eqref{e:Xt} and \eqref{rsz}). In what follows, we recall that the $i$-th eigenvalue of $\X_s$ is denoted by $\lambda_i(s)$ and its associated unit eigenvector is denoted by $\u_i(s)$. For any constants $C, \delta > 0$, we define the two spectral domains 
\begin{flalign}
& \mathcal D_{C, \delta} = \left\{  z = E + \iu\eta \colon |E| \le \frac{1}{C},\, N^{-1 + \delta} \le \eta \le \frac{1}{C}  \right\}, \label{d1} \\
&  \widetilde {\mathcal D}_{C,\delta} = \left\{  z = E + \iu\eta \colon |E| \le \frac{1}{C},\, N^{-1/2 + \delta } \le \eta \le \frac{1}{C}  \right\} \label{d2}.
\end{flalign}

We also recall the free convolution of $\X$ with the semicircle law is defined to be the probability measure on $\mathbb R$ whose Stiletjes transform satisfies the equation
\begin{equation}\label{e:mfcdef}
\mfc (s, z)  = \frac{1}{N} \sum_{i=1}^N \frac{1}{\lambda_i(0) - z - s \mfc(s,z)}.
\end{equation}
Basic facts about the free convolution, including its existence and uniqueness, may be found in \cite{biane1997free}. It has a density, $\dsc(s,x )\, dx$, and its classical eigenvalue locations are defined for $1 \le i \le N$ by 
\begin{equation}\label{e:gammaidef}
\gamma_i(s) = \inf \left\{  y \in \mathbb{R} \colon \int_{-\infty}^y \dsc(s,x)\, dx \ge \frac{i}{N} \right\}.
\end{equation} 
These are random variables which depend on the initial data $\big( \lambda_i(0) \big)_{i=1}^N$ determined by $\textbf{X}_0$.

The following intermediate local law for $\textbf{R} (s, z)$ (on scale $\eta \gg N^{-1/2 + \delta }$) was essentially shown as \cite[Theorem 3.5]{aggarwal2018goe}. 

\begin{lem}[{\cite[Theorem 3.5]{aggarwal2018goe}}]
	
	\label{l:Xlocallaw}
	
	There exists a (deterministic) countable set $\mathcal A \subset (0,2)$ with no accumulation points in $(0,2)$  such that the following holds for $\alpha \in (0,2) \setminus \mathcal A$. 
	For any fixed real number $\delta > 0$ with $\delta < \max \big\{ \frac{(b- 1/\alpha)(2 - \alpha)}{20}, \frac{1}{2} \big\}$, there exists a constant $C = C(\delta) > 0$ such that for $s\in [ 0, N^{-\delta}]$, we have with overwhelming probability that
	\begin{flalign}
	\label{mn0z}
	\sup_{z \in \widetilde {\mathcal{D}}_{C, \delta}} \big |  m_N(s, z )   - m_\alpha(z)\big| < C N^{ - \alpha \delta / 8}, \quad   
		\sup_{z \in \widetilde {\mathcal{D}}_{C, \delta}} \max_{1 \le j \le N} \big| R_{jj}(s, z) \big| < (\log N)^C,
	\end{flalign}
	where we recall $\widetilde{D}_{C, \delta}$ from \eqref{d2}.
	\eel
	In fact, \cite[Theorem 3.5]{aggarwal2018goe} was only stated in the case $s = 0$, but it is quickly verified that the same proof applies for arbitrary $s \in [0, N^{-\delta}]$, especially since $\textbf{H} + s^{1 / 2} \textbf{W}$ satisfies the conditions in \Cref{momentassumption} for $s \in [0, N^{-\delta}]$ if $\textbf{H}$ does.

	 From now on, we always assume $\alpha \in (0,2) \setminus \mathcal A$, where $\mathcal A$ is the set from \Cref{l:Xlocallaw}, even when this is not noted explicitly. \eer The next lemma provides more local estimates on $\textbf{R} (s, z)$ and $\textbf{X}_s$ (on scales around $N^{-1}$), if $N^{-1/2  } \ll s \ll 1$; they are consequences of \Cref{l:Xlocallaw} using results of \cite{landon2017convergence}. Specifically, \eqref{rijestimate} follows from \cite[Theorem 3.3]{landon2017convergence}, and \eqref{mszestimate} follows from \eqref{rijestimate} and the first estimate in \eqref{mn0z}; \eqref{e:xtrigidity} follows from \cite[Theorem 3.5]{landon2017convergence}; and \eqref{lambdaii1s} follows from \cite[Theorem 3.6]{landon2017convergence}. The hypotheses of these statements from \cite{landon2017convergence} are all verified by the first bound in \eqref{mn0z}.  The final estimate is an immediate consequence of \eqref{rijestimate}, \eqref{mszestimate}, and \eqref{e:macontinuity}.\eer
	
	In the following, we recall the $\gamma_i (s)$ were defined in \eqref{e:gammaidef}.

	\bel[{\cite{landon2017convergence}}] 
	
	\label{l:thatlaw} \label{l:xtlevel}
	
	There exists a constant $K  > 0$ such that the following holds for any real numbers $r, \delta > 0$.%,
	
	\begin{enumerate}
		\item Set $\mathcal{D} = \mathcal{D}_{K, \delta}$ and $\widetilde{\mathcal{D}} = \widetilde{\mathcal{D}}_{K, \delta}$, where we recall the definitions \eqref{d1} and \eqref{d2}, respectively. With overwhelming probability, we have that 
		\begin{flalign}
		\label{rijestimate}
		& \sup_{s \in [N^{-1/2 + \delta}, N^{-\delta}]}\sup_{z \in \mathcal D} \big|  m_N (s, z)   - \mfc (s,z ) \big| <  \frac{N^\delta}{N\eta},  \\
		& \sup_{s \in [N^{-1/2 + \delta}, N^{-\delta}]}\sup_{ z\in \widetilde{\mathcal D}} \big| m_\alpha(z) - \mfc(s, z) \big| < N^{-\alpha \delta/16} \label{mszestimate}.
		\end{flalign}
		
		\item With overwhelming probability, we have that 
		\begin{equation}\label{e:xtrigidity}
		\sup_{s \in [N^{-1/2 + \delta}, N^{-\delta}]} \big| \lambda_i(s) - \gamma_i(s) \big| \le N^{-1 + \delta}.
		\end{equation}

		\item For any $\varepsilon \in (0, 1)$ and $ s \in [N^{-1/2 + \delta}, N^{-\delta}]$, we have for sufficiently large $N$ that
		\beq
		\label{lambdaii1s} 
		\P\left(  \big| \lambda_i (s) - \lambda_{i+1} (s) \big| \le \frac{\eps}{N} \right) \le N^\delta \eps^{2-r}.
		\eeq
		
		\item 	 For $E_1, E_2 \in [-K^{-1}, K^{-1}]$ and $\eta \in [ N^{-1/2 + \delta}, N^{-\delta}]$, we have
\beq\label{e:mncontinuity}
	\big| \Im m_N(E_1 + \iu \eta ) - \Im m_N(E_2 + \iu \eta) \big| \le C| E_1 - E_2| + C\eta + C N^{- \alpha \delta /16}
	\eeq
	with overwhelming probability.

		%\item \ber There exist constants $C, c$ such that for any $|x| \le K^{-1}$,
		%\beq\label{e:abovebelow} c \le \varrho_\alpha (x) \le C.\eeq
		%\eer
	\end{enumerate}
	\eel

	The following lemma provides resolvent and delocalization estimates for $\X_s$. The first estimate in \eqref{rijestimate2} below follows from \cite[Proposition 3.9]{aggarwal2018goe}, whose hypotheses are verified by the second bound in \eqref{mn0z}. We omit the proof of the second since, given the first, it follows by standard arguments (for example, see the proof of \cite[Theorem 2.10]{benaych2016lectures}). 
	
	\bel[{\cite[Proposition 3.9]{aggarwal2018goe}}] \label{l:utdeloc}
	
	There exists a constant $K > 0$ such that the following holds. Fix real numbers $\delta > 0$ and $s \in [N^{-1 / 2 + \delta}, N^{-\delta}]$, and a unit vector $\q \in \mathbb{R}^N$ with stable support. For each index $i \in [1, N]$ such that  $| \gamma_i (s)| < K^{-1}$, we have  with overwhelming probability that 
	\beq
	\label{rijestimate2}
	\sup_{z \in \mathcal D} \max_{1 \le j, k \le N} \big|R_{jk}(s, z)\big| < N^{\delta}, \qquad \big\langle \u_i(s) , \q  \big\rangle^2 \le  N^{-1 + \delta}.
	\eeq
	\eel

	\subsection{Interpolating matrix} \label{s:hgammaestimates} 
	
	Recalling $t$ from \eqref{t}, define the interpolating matrix 
	\beq 
	\label{hgamma}
	\bH^\gamma = \{ H^\gamma_{ij} \} = \gamma\bA + \X + (1 -\gamma^2)^{1/2} t^{1/2} \W,
	\eeq
	where $\W$ is an independent $N \times N$ GOE matrix. Namely, it is an $N \times N$ real symmetric random matrix $\mathbf{W}_N = \{ w_{ij} \}$, whose upper triangular entries $w_{ij}$ are mutually independent Gaussian random variables with variances $(1 + \one_{i = j}) N^{-1}$.
		
	The following lemma estimates the entries of the resolvent matrix $\mathbf G^\gamma = \big\{ G_{ij}^{\gamma} (z) \big\} = ( \bH^\gamma - z)^{-1}$ and provides complete eigenvector delocalization for $\textbf{H}^{\gamma}$. The first bound in \eqref{estimateuihgamma} below was obtained as \cite[Theorem 3.16]{aggarwal2018goe}; given this, the second bound there follows by standard arguments (again, see the proof of \cite[Theorem 2.10]{benaych2016lectures}).

	\bel[{\cite[Theorem 3.16]{aggarwal2018goe}}]\label{l:gammagreen}
	
	There exists a constant $K > 0$ such that the following holds. Fix real numbers $\delta > 0$ and $\gamma \in [0, 1]$, and abbreviate $\mathcal D = \mathcal{D}_{K, \delta}$ (recall \eqref{d1}). Let $\u_i(\mathbf H^\gamma)$ be a unit eigenvector of $\bH^\gamma$ such that the corresponding eigenvalue ${\lambda_i(\mathbf H^\gamma)}$ satisfies $\big|\lambda_i(\mathbf H^\gamma) \big| \le K^{-1}$. Then, with overwhelming probability we have the bounds  
	\beq
	\label{estimateuihgamma}
	\sup_{0\le \gamma \le 1} \sup_{z \in \mathcal D} \max_{1 \le j, k \le N} \left|  G^\gamma_{jk}(z) \right| < N^\delta; \qquad 	\big\| \u_i (\textbf{\emph{H}}^{\gamma}) \big\|_\infty  \le N^{ - 1/2 + \delta}.
	\eeq

	\eel
	
	In view of the identity  
	\begin{flalign}
	\eta \displaystyle\sum_{j = 1}^N \Big(\big| \lambda_i (\textbf{M}) - E \big|^2 + \eta^2\Big)^{-1} = N^{-1} \Im \text{Tr} \big( \textbf{M} - z \big)^{-1},
	\end{flalign}
	
	\noindent which holds for any $N \times N$ matrix $\textbf{M}$ and complex number $z = E + \mathrm{i} \eta \in \mathbb{H}$, \Cref{l:thatlaw} and \Cref{l:gammagreen} together quickly imply the following lemma that bounds the number of eigenvalues of $\textbf{H}^{\gamma}$ or $\mathbf X_s$ in a given interval.

	\bel
	
	\label{l:disperse}

	For any real number $\delta > 0$, there exist constants $K > 0$ and $C = C(\delta) > 0$ such that the following holds. For any interval $I \subseteq [-K^{-1}, K^{-1}]$ of length $| I |  \ge N^{-1 + \delta}$, we have with overwhelming probability that
	\begin{flalign}\label{e:disperse}
		& \displaystyle\sup_{\gamma \in [0, 1]} \Big| \big\{ i : \lambda_i(\mathbf H^\gamma) \in I \big\} \Big| \le C |I| N^{1 + \delta}; \qquad  \displaystyle\sup_{s \in [N^{-1 / 2 + \delta}, N^{-\delta}]} \Big| \big\{ i : \lambda_i(\textbf{\emph{X}}_s) \in I \big\} \Big| \le C |I| N.
	\end{flalign}
	
	\eel

	The following result states that the $i$-th eigenvalue of $\textbf{H}^{\gamma}$ and $\textbf{X}_s$ is close to $0$ if $i$ is close to $\frac{N}{2}$. Its proof will be given in \Cref{s:appendixA}.

	\bel\label{l:hgammarigid}
	
	For each real number $c_1>0$, there exists a constant $c_2>0$ such that the eigenvalues $\lambda_i(\bH^\gamma)$ of $\bH^\gamma$ and $\lambda_i(\mathbf{X}_s)$ of $\mathbf{X}_s$ satisfy
	\beq
	\sup_{\gamma\in[0,1]}\big| \lambda_i(\bH^\gamma) \big| < c_1; \qquad \sup_{s\in[0,1]} \big| \lambda_i(\mathbf{X}_s)\big| < c_1; \qquad \displaystyle\sup_{s \in [0, 1]} \big| \gamma_i (s) \big| < c_1, 
	\eeq
	
	\noindent for each $i \in \big[ (1/2 - c_2 ) N, (1/2 + c_2) N \big]$, with overwhelming probability.
	\eel

In \Cref{s:appendixA}, we use \Cref{l:hgammarigid} and \Cref{l:thatlaw} to deduce the following rigidity statements comparing the classical locations $\widehat{\gamma}_i (s)$ to the $\gamma_i$, and the $\widehat\gamma_i(s)$ to the $\gamma_i(s)$ (recall \eqref{gammaalphai}, \eqref{e:gammaidef}, and \eqref{hatgammaalphai}). 

	\bel  \label{l:newrigidity}
	Fix $\delta >0$. There exist constants $C, c_1 >0$ and $c = c(\delta) > 0$ such that for each $i \in \big[ (1/2 - c_1 ) N, (1/2 + c_1) N \big]$, we deterministically have the bound
	\beq\label{e:gammacompare}
	\sup_{s \in [ N^{-1/2 + \delta}, N^{-\delta} ] }  \big| \widehat \gamma_i(s) - \gamma_i  \big| \le C N^{-c},
	\eeq
	
	\noindent and with overwhelming probability the bound
	\beq
	\label{e:newrigidity} \sup_{s \in [ N^{-1/2 + \delta}, N^{-\delta} ] }  \big| \widehat \gamma_i(s) - \gamma_i(s) \big| \le C N^{ - 1/2 + \delta}.
	\eeq
	
	\eel

\section{Comparison} \label{s:comparison}

This section establishes \Cref{l:maincomparison}, which compares the eigenvector statistics of $\X_t$ to those of $\mathbf H$. \Cref{CompareProofv} establishes this result assuming a general comparison estimate, certain derivative bounds, and a level repulsion estimate for $\textbf{H}^{\gamma}$. We then prove the comparison estimate in \Cref{ProofCompareFunction}; the necessary derivative bounds in \Cref{s:lrh}; and the level repulsion estimate in \Cref{s:vectorcompare}.

\subsection{Proof of \Cref{l:maincomparison}}

\label{CompareProofv}

In this section we establish \Cref{l:maincomparison} assuming \Cref{gcompare}, \Cref{l:smoothev}, \Cref{deterministicqestimate}, and \Cref{l:lrh} below. In what follows, for any $\kappa \in [0, 1]$, $\textbf{M} = \{ m_{ij} \} \in \text{Mat}_{N \times N}$, and $a, b \in [1, N]$, we  define $\Theta_{\kappa}^{(a, b)} \textbf{M} \in \text{Mat}_{N \times N}$ as follows. Recalling $\rho$ from \eqref{e:parameters}, if $|m_{ab}| = |m_{ba}| \ge N^{-\rho}$, then set $\Theta_{\kappa}^{(a, b)} \textbf{M} = \textbf{M}$. Otherwise, if $|m_{ab}| = |m_{ba}| < N^{-\rho}$, set $\Theta_{\kappa}^{(a, b)} \textbf{M}$ to be the $N \times N$ matrix whose $(i, j)$ entry is equal to $m_{ij}$ if $(i, j) \notin \big\{ (a, b), (b, a) \big\}$ and is equal to $\kappa m_{ab} = \kappa m_{ba}$ otherwise. Moreover, for any differentiable function $F: \text{Mat}_{N \times N} \rightarrow \mathbb{C}$ and indices $a, b \in [1, N]$, we define $\partial_{ab} F$ to be the derivative of $F$ with respect to $m_{ab}$. 

We first state the following comparison theorem between functions of $\textbf{H}^0 = \textbf{X}_t$ and $\textbf{H}^{\gamma}$ (recall \eqref{hgamma}), which will be established in \Cref{ProofCompareFunction} below. 

\begin{lem} 
	
\label{gcompare}

	There exists a constant $c > 0$ such that the following holds. Let $F\colon \text{\emph{Mat}}_{N \times N} \rightarrow \mathbb{C}$ denote a smooth function, and suppose $K, L > 1$ are such that  
	\begin{flalign}
	\label{definitionk} 
	\displaystyle\max_{0 \le j \le 4}\displaystyle\sup_{0 \le \gamma \le 1} \displaystyle\max_{1 \le a, b \le N} \displaystyle\sup_{0 \le \kappa \le 1} \Big| \partial_{ab}^{(j)} F \big( \Theta_{\kappa}^{(a, b)} \textbf{\emph{H}}^{\gamma} \big) \Big| \le K
	\end{flalign} 
	
	\noindent holds with overwhelming probability, and 
	\begin{flalign}
	\label{definitionl} 
	\displaystyle\max_{0 \le j \le 4}\displaystyle\sup_{0 \le \gamma \le 1} \displaystyle\max_{1 \le a, b \le N} \displaystyle\sup_{0 \le \kappa \le 1} \Big| \partial_{ab}^{(j)} F \big( \Theta_{\kappa}^{(a, b)} \textbf{\emph{H}}^{\gamma} \big) \Big| \le L
	\end{flalign}
	
	\noindent holds deterministically. Then, for any constant $D > 0$, there exists a constant $C = C(D) > 0$ such that 
	\begin{flalign}
	\displaystyle\sup_{0 \le \gamma \le 1} \big| F (\textbf{\emph{H}}^{\gamma}) - F (\textbf{\emph{H}}^0) \big| \le K N^{-c} + C L N^{-D}. 
	\end{flalign}

\end{lem} 

Next we require the following function, originally introduced in \cite[Section 3.2]{TV11}, that measures how close eigenvalues of some matrix $\textbf{A}$ are to a given eigenvalue.

\begin{defn} \label{fmdef}

Let $\textbf{A} \in \text{Mat}_{N \times N}$. If $1 \le i \le N$ is such that $\lambda_i(\mathbf A)$ is an eigenvalue of a matrix $\mathbf A$ with multiplicity one, we define
\beq
Q_i(\mathbf A) = \frac{1}{N^2} \sum_{\substack{1 \le j \le N \\ j \neq i}} \big| \lambda_j(\mathbf A) - \lambda_i(\mathbf A) \big|^{-2}.
\eeq
To deal with the case of multiplicity greater than one, we introduce a cutoff. For any $M > 0$, we fix a smooth function $ f_M\colon \mathbb{R}_{\ge 0} \rightarrow \mathbb{R}$ such that there exists a constant $C > 0$ (independent of $M$ and $N$) satisfying the following two properties.
\begin{enumerate}
\item For any $x \in \mathbb{R}_{> 0}$, we have that $\big| f'_M(x) \big|  +  \big| f''_M(x) \big|  + \big| f'''_M (x) \big|  \le C$.

\item If $x\in [ 0 , M]$ then $\big| f_M (x) - x \big| \le 1$, and if $x \ge M$, then $ f_M (x) = M$. 
\end{enumerate}
The function $ f_M \big( Q_i(\mathbf A) \big)$ is then well-defined and smooth on real symmetric matrices. 

\end{defn} 

The following two lemmas control the derivatives of the $Q_i (\textbf{H})$ and of the eigenvector entries of $\mathbf H^\gamma$ with respect to the matrix entries of $\textbf{H}$; the first is an overwhelming probability bound, and the second is a deterministic bound. They will be established in \Cref{s:lrh} below. 

\bel

\label{l:smoothev}

There exists a constant $c>0$ such that the following holds. Fix real numbers $\gamma, \kappa \in [0, 1]$, a constant $\omega > 0$, and integers $i, a, b \in [1, N]$. Set $M = N^{2 \omega}$, and assume that 
\begin{flalign}
\label{e:qiassume}  
\Big| \lambda_i \big( \Theta_{\kappa}^{(a, b)} \textbf{\emph{H}}^{\gamma} \big) \Big| < c, \quad \text{and} \quad  Q_i \big( \Theta_{\kappa}^{(a, b)} \bH^\gamma \big) \le M = N^{2 \omega}
\end{flalign}

\noindent both hold with overwhelming probability. Then
\begin{equation}
\label{e:qibound}
 \bigg| \partial^{(k)}_{ab}  \Big(  Q_i\big( \Theta_{\kappa}^{(a, b)} \bH^\gamma \big) \Big) \bigg| \le C N^{10k(\omega +\delta)}
\end{equation}

\noindent also holds with overwhelming probability, for any integer $0 \le k \le 4$.

Moreover, for any $\textbf{\emph{q}} \in \mathbb{R}^N$, there exists a constant $C = C \big( |\supp \q| \big) > 0$ such that 
\begin{equation}
\label{derivativeuihgammaestimate}
 \Bigg| \partial^{(k)}_{ab} \bigg( \Big\langle \q, \u_i \big(\Theta_{\kappa}^{(ab)} \bH^\gamma \big) \Big\rangle^2 \bigg) \Bigg| \le C N^{-1 +10k(\omega + \delta)}
 \end{equation}

\noindent also holds with overwhelming probability, for any integer $0 \le k \le 4$. 
\eel 

\begin{lem} 

\label{deterministicqestimate} 

Fix real numbers $\gamma, \kappa \in [0, 1]$, a constant $\omega > 0$, and integers $i, a, b \in [1, N]$; assume that  $Q_i \big( \Theta_{\kappa}^{(a, b)} \textbf{\emph{H}}^{\gamma} \big) < N^{2 \omega}$. Then, for any integers $k \in [0, 4]$ and $1 \le i \le N$, we have the deterministic bounds
\beq\label{e:detbounds}
\bigg| \partial^{(k)}_{ab}  \Big( Q_i \big(\Theta_{\kappa}^{(a, b)} \textbf{\emph{H}}^{\gamma} \big) \Big) \bigg| \le C N^{10 + 6 \omega}; \qquad \Bigg| \partial^{(k)}_{ab} \bigg( \Big\langle \q, \u_i \big(\Theta_{\kappa}^{(ab)} \bH^\gamma \big) \Big\rangle^2 \bigg) \Bigg| \le C N^{15 + 10 \omega}.
\eeq

\end{lem} 

Next we state a level repulsion estimate, which will be established in \Cref{s:vectorcompare} below. 

\bel 

\label{l:lrh} 

There exist constants $c, \upsilon>0$ such that, for any fixed index $i \in \big[ (1/2- c) N, (1/2 + c) N \big]$ and real number $\gamma \in [0, 1]$, we have that 
\beq
   \P \big(Q_i(\bH^\gamma) \ge N^{\upsilon} \big) \le 2 N^{-\upsilon/4}.
\eeq
\eel

Given these statements, we now prove \Cref{l:maincomparison}. The argument follows \cite[Theorem 1.1]{bourgade2017eigenvector}.

\begin{proof}[Proof of \Cref{l:maincomparison}.]
	
	For brevity we consider just $n=1$; the general case is no harder. 
	
	By \Cref{l:lrh}, there exists some $\omega > 0$ such that, for each $\gamma \in \{ 0, 1 \}$,
	\beq 
	\label{qiomega}
	\P \big(Q_i(\bH^\gamma) \ge N^{\omega} \big) \le 2 N^{-\omega/4}.
	\eeq
	
	Denote the degree of $P$ by $m$, so that $P(x) \le C( x^m + 1)$ for $x\ge 0$. Delocalization for $\bH^\gamma$, \eqref{estimateuihgamma}, implies that $N \big \langle \q, \u_i (\textbf{H}^{\gamma}) \big\rangle^2 \le N^\delta$ with overwhelming probability for each $\delta >0$ and $\gamma\in[0,1]$, if $N$ is sufficiently large. Therefore, 
	\beq 
	\label{pmdelta}
	\E \bigg[ P \Big(N \big\langle \q, \u_i(\gamma)\big\rangle^2 \Big)^2 \bigg] \le C N^{2m\delta}.
	\eeq
	
	\noindent Now set $M = N^{2 \omega}$, and let $g=g_M$ be a smooth function with uniformly bounded derivatives such that $0 \le g(x) \le 1$ for each $x \in \mathbb{R}_{> 0}$; $g(x)=1$ for $x\le M$; and $g(x) = 0$ for $x \ge 2M$. Then,
	\begin{align}
	 \Bigg| \E & \bigg[  P \Big( N \big\langle \q, \u_{i} (\textbf{H}^1) \big\rangle^2 \Big)  \bigg]  - \E \bigg[  P \Big( N\big\langle \q, \u_i (\textbf{H}^0) \big\rangle^2 \Big) \bigg]    \Bigg|\\
	 & \le \Bigg| \E \bigg[  P \Big( N \big\langle \q, \u_{i} (\textbf{H}^1) \big\rangle^2 \Big) g \big(Q_i(\bH^1) \big)  \bigg]  - \E \bigg[  P \Big( N \big\langle \q, \u_{i}(\textbf{H}^0) \big\rangle^2 \Big) g \big(Q_i(\bH^0) \big)  \bigg]    \Bigg| \\
	& \quad +  \E \bigg[  P \Big( N \big\langle \q, \u_{i}(\textbf{H}^1) \big\rangle^2 \Big)^2  \bigg]^{1/2} \P \big( Q_i(\bH^1)  \ge M \big)  +  \E \bigg[  P \Big( N \big\langle \q, \u_{i}(\textbf{H}^0) \big\rangle^2 \Big)^2  \bigg]^{1/2} \P\big( Q_i(\bH^0)  \ge M \big)  \\
	\label{nquiestimate}
	&  \le \Bigg| \E \bigg[  P \Big( N \big\langle \q, \u_{i}(\textbf{H}^1) \big\rangle^2 \Big) g \big(Q_i(\bH^1) \big)  \bigg]  - \E \bigg[  P \Big( N \big\langle \q, \u_{i}(\textbf{H}^0) \big\rangle^2 \Big) g \big(Q_i(\bH^0) \big)    \bigg]  \Bigg| +  CN^{-\omega/4 + m\delta}, 
	\end{align}
	
	\noindent where in the last estimate we applied \eqref{qiomega} and \eqref{pmdelta}. 
	
	Now let us define the function $h: \text{Mat}_{N \times N} \rightarrow \mathbb{R}$ by setting
	\beq 
	h(\mathbf A) = h_i (\textbf{A}) = P \Big( N \big\langle \q, \u_{i}(\mathbf A) \big\rangle^2 \Big) g \big(Q_i(\mathbf A)\big)
	\eeq 
	
	\noindent for any $\textbf{A} \in \text{Mat}_{N \times N}$. By \Cref{l:smoothev}, \Cref{deterministicqestimate}; a union bound over $1 \le i, a, b \le N$ and $\gamma$ and $\kappa$ in an $N^{-30}$-net of $[0, 1]$; and the fact that $h (\textbf{A}) = 0$ if $Q_i (\textbf{A}) \ge 2M$, we have that $h$ deterministically satisfies 
	\begin{flalign}
	\displaystyle\sup_{0 \le k \le 4} \displaystyle\sup_{\gamma \in [0, 1]} \displaystyle\max_{1 \le a, b \le N} \displaystyle\sup_{\kappa \in [0, 1]} \Big| \partial^{(k)}_{ab} h \big( \Theta_{\kappa}^{(a, b)} \bH^\gamma \big) \Big| \le C N^{15 + 15 m \omega},
	\end{flalign}
	
	\noindent and with overwhelming probability satisfies 
	\beq
	\displaystyle\sup_{0 \le k \le 4} \displaystyle\sup_{\gamma \in [0, 1]} \displaystyle\max_{1 \le a, b \le N} \displaystyle\sup_{\kappa \in [0, 1]} \Big| \partial^{(k)}_{ab} h \big( \Theta_{\kappa}^{(a, b)} \bH^\gamma \big) \Big| \le C N^{20 m (\omega + \delta)}.
	\eeq
	
	\noindent Therefore, upon setting $\omega$ and $\delta$ sufficiently small, \Cref{gcompare} implies $\big| \mathbb{E} \big[ h (\textbf{H}^1) \big] - \mathbb{E} \big[ h(\textbf{H}^0) \big] \big| < CN^{-c}$. Inserting this into \eqref{nquiestimate} yields
	\begin{flalign}
	\Bigg| \E & \bigg[  P \Big( N \big\langle \q, \u_{i}(1) \big\rangle^2 \Big)  \bigg]  - \E \bigg[  P \Big( N\big\langle \q, \u_{i}(0) \big\rangle^2 \Big) \bigg]    \Bigg| \le C N^{-c} +  CN^{-\omega/4 + m\delta}.
	\end{flalign}
	
	\noindent The lemma follows from further imposing that $5 m \delta < \omega$. 
\end{proof}

\subsection{Proof of \Cref{gcompare}}

\label{ProofCompareFunction}

In this section we establish \Cref{gcompare}. 

\begin{proof}[Proof of \Cref{gcompare}]
	
	Observe (by \eqref{ab}, for instance) that
	\beq
	\label{sumf}
	\partial_\gamma \E \big[ F (\bH^\gamma) \big] = \sum_{1 \le i, j \le N} \E \Bigg[  \partial_{ij} F (\bH^\gamma) \bigg( A_{ij}  - \frac{\gamma t^{1/2}}{( 1- \gamma^2)^{1/2} } w_{ij}  \bigg)  \Bigg].
	\eeq
	
	Now, we condition on the label $\Psi$ of $\bH$ (recall \Cref{hsumabc}) and denote the associated conditional expectation by $\E_{\Psi}$. We first consider the case $\psi_{ij} =1$. This implies $A_{ij} = B_{ij} =0$, and Gaussian integration by parts (see for instance \cite[Appendix A.4]{talagrand2010mean}) yields
	\beq
	\label{psiij1} 
	\E_{\Psi} \Bigg[  \partial_{ij} F (\bH^\gamma) \bigg( \frac{\gamma t^{1/2}}{( 1- \gamma^2)^{1/2} } w_{ij}  \bigg)  \Bigg] = \frac{t\gamma}{N}  \E_{\Psi} \left[  \partial_{ij}^2 F (\bH^\gamma)   \right], \quad \text{whenever $\psi_{ij} = 1$.}
	\eeq
	Hoeffding's inequality applied to the Bernoulli random variable $\psi_{ij}$, whose distribution was defined in \eqref{e:phidef}, implies that there are likely at most $C N^{1 + \alpha \rho}$ pairs $(i, j)$ such that $\psi_{ij} = 1$. Specifically, 
	\beq\label{e:hoeffding} 
	\mathbb{P} \bigg[ \Big| \big\{ (i, j) \in [1, N] \times [1, N]: \psi_{ij} = 1  \big\} \Big| < C N^{1 + \alpha\rho } \bigg] \ge 1 - C \exp \big( - N^{\alpha \rho} \big).
	\eeq 
	
	By \eqref{definitionl}, the contribution of \eqref{psiij1} over the complement of the event described in \eqref{definitionk} or \eqref{e:hoeffding} is bounded by $C L N^{-D}$, for some constant $C = C(D) > 0$. This, together with \eqref{definitionk}, \eqref{psiij1}, \eqref{e:hoeffding}, \eqref{c1c2t0}, and \eqref{e:parameters} imply that the sum of \eqref{psiij1} over all $(i, j)$ such that $\psi_{ij} = 1$ or $i = j$ is at most
	\beq\label{e:largefield}
	C K t N^{-1} N^{\alpha \rho + 1} + C L N^{-D} \le C K N^{\alpha \rho - (2 - \alpha) \nu} + C L N^{-D} < K N^{-c} + C L N^{-D},
	\eeq
	
	\noindent for some constants $c > 0$ (only dependent on the fixed parameters $\alpha$, $\rho$, and $\nu$) and $C = C (D) > 0$. 
	
	We next consider the case when $\psi_{ij} =0$ and $i \ne j$. Then, $A_{ij} = a_{ij}( 1- \chi_{ij})$ and $B_{ij} = b_{ij} \chi_{ij}$; abbreviate $a_{ij} =a$, $b_{ij} = b$, $\chi_{ij} = \chi$, and $w_{ij} = w$. Set
	\beq
	\label{hij}
	h = \gamma ( 1- \chi) a + \chi b + ( 1 -\gamma^2)^{1/2} t^{1/2} w.
	\eeq
	Fix $(i, j) \in [1, N]^2$ such that $\psi_{ij} = 0$, abbreviate $F^{(k)} = \partial_{ij}^{(k)} F$, and abbreviate $\textbf{S} = \Theta_0^{(i, j)} \textbf{H}$. Then a Taylor expansion yields
\beq
F' (\bH^\gamma) = F'(\S) + h F''(\S) + \frac{1}{2!} h^2 F^{(3)}(\S) + \frac{1}{3!} h^3 F^{(4)} \big( \Theta_{\kappa}^{(i, j)} \textbf{H}^{\gamma} \big),
\eeq
	for some $\kappa \in [0, 1]$. Hence, the $(i, j)$ term in the sum on the right side of \eqref{sumf} is equal to
	\beq
	\label{1234faw}
	\E_{\Psi} \Bigg[   \bigg( (1 - \chi) a - \frac{\gamma t^{1/2}}{( 1- \gamma^2)^{1/2} } w  \bigg) \Big( F'(\S) + h F''(\S) + \frac{1}{2!}h^2 F^{(3)}(\S) 
+ \frac{1}{3!} h^3 F^{(4)} \big( \Theta_{\kappa}^{(i, j)} \textbf{H}^{\gamma} \big) \Big) \Bigg].\eeq
	Using the mutual independence between $\S$, $a$, $b$, $\chi$, and $w$, and the fact that $a$, $b$, and $w$ are all symmetric, we conclude that \eqref{1234faw} is equal to
	\beq
	\label{1234faw2} 
	\E_{\Psi} \Bigg[   \bigg( (1 - \chi) a  - \frac{\gamma t^{1/2}}{( 1- \gamma^2)^{1/2} } w  \bigg) h F''(\S) \Bigg] +  \frac{1}{3!} \E_{\Psi} \Bigg[   \bigg( (1 - \chi) a  - \frac{\gamma t^{1/2}}{( 1- \gamma^2)^{1/2} } w  \bigg) h^3 F^{(4)}\big( \Theta_{\kappa}^{(i, j)} \textbf{H}^{\gamma} \big) \Bigg].
	\eeq
	Again using the mutual independence between $\S$, $a$, $b$, $\chi$, and $w$; the fact that $a$, $b$, and $w$ are all symmetric; and \eqref{hij}, we find that the first term in \eqref{1234faw2} is 
	\beq
	\E_{\Psi} \Bigg[  F''(\S) \bigg( (1 - \chi) a  - \frac{\gamma t^{1/2}}{( 1- \gamma^2)^{1/2} } w  \bigg)  h   \Bigg] = \gamma \E \big[ F''(\S) \big] \E \big[a^2 ( 1- \chi) - tw^2 \big] = 0,\eeq
	where the final equality follows from the choice of $t$ in \eqref{t}.
	
	The second term in \eqref{1234faw2} is bounded above by 
	\beq\label{e:fourthsecond}
	C \E_{\Psi} \bigg[ \Big| F^{(4)} \big( \Theta_{\kappa}^{(i, j)} \textbf{H}^{\gamma} \big)\Big|  \big( ( 1- \chi) a^4  + t^2 w^4  + \chi tw^2 b^2  \big) \bigg].
	\eeq
	
	\noindent On the complement of the event in \eqref{definitionk}, this expectation is bounded by $C L N^{-D - 2}$, for some constant $C = C(D) > 0$. On this event, we use \eqref{c1c2t0}; the facts that $\mathbb{E}[w^2] \le N^{-1}$ and $\mathbb{E}[w^4] \le N^{-2}$; and the estimates (which can be quickly deduced from \Cref{abcpsidefinitions}; see \cite[Lemma 4.2]{aggarwal2018goe} for details) 
	 \beq\label{530} \E_{\Psi} \big[ (1-\chi) a^4 \big] \le C N^{\nu(\alpha -4) - 1}, \quad \mathbb{E}_{\Psi} \big[\chi b^2 \big] \le C N^{\rho(\alpha -2) -1}, \eeq 
		
	\noindent to bound it by
	\beq
	C K \big( N^{\nu(\alpha -4) - 1 } + N^{2\nu(\alpha -2) - 2} + N^{(\rho + \nu)(\alpha -2) - 2} \big) \le C K N^{-2 - c},
	\eeq
	
	\noindent for some constant $c > 0$ (only dependent on the fixed parameters $\alpha$, $\nu$, and $\rho$), where we have used \eqref{e:parameters} in the last inequality. So, the sum of \eqref{e:fourthsecond} over all $(i, j) \in [1, N]^2$ such that $i \ne j$ and $\psi_{ij} = 0$ is at most 
	\begin{flalign}
	\label{estimate4}
	K N^{-c} + C L N^{-D}. 
	\end{flalign}
	
	Now the lemma follows from the fact that the contribution to \eqref{sumf} from all terms corresponding to $(i, j)$ with $i = j$ or $\psi_{ij} = 1$ is bounded by \eqref{e:largefield} and the fact that the contribution of all terms coming from $(i, j)$ with $i \ne j$ and $\psi_{ij} = 0$ is bounded by \eqref{estimate4}.
\end{proof}

\subsection{Proof of \Cref{l:smoothev} and \Cref{deterministicqestimate}}

\label{s:lrh}

In this section we prove \Cref{l:smoothev} and \Cref{deterministicqestimate}. We begin with the following estimate on the resolvent entries of $\Theta_{\kappa}^{(a, b)} \textbf{H}^{\gamma}$. In the below, we recall $\mathcal{D}_{C; \delta}$ from \eqref{d1}.

\bel\label{l:perturbS}

There exists a constant $K > 0$ such that following holds. For any $\delta > 0$, the bound 
\beq\label{e:Gsbounds}
\displaystyle\sup_{0 \le \kappa \le 1} \displaystyle\max_{1 \le a, b \le N} \sup_{0 \le \gamma \le 1} \max_{1 \le i,j \le N} \sup_{z \in \mathcal D_{K; \delta}} \bigg| \Big( \big( \Theta_{\kappa}^{(a, b)} \textbf{\emph{H}}^{\gamma} -z \big)^{-1}  \Big)_{ij} \bigg| < N^\delta  
\eeq

\noindent holds with overwhelming probability. Moreover, for each $c_1 > 0$, there exists some $c_2 > 0$ such that
\beq
\label{uihgammaestimate}
\begin{aligned}
& \displaystyle\sup_{0 \le \kappa \le 1} \displaystyle\max_{1 \le a, b \le N} \sup_{0 \le \gamma \le 1} \Big\| \textbf{\emph{u}}_i \big( \Theta_{\kappa}^{(a, b)} \textbf{\emph{H}}^{\gamma} \big) \Big\|_{\infty}   < N^{\delta - 1/2}; \\
& \displaystyle\sup_{0 \le \kappa \le 1} \displaystyle\max_{1 \le a, b \le N} \sup_{0 \le \gamma \le 1} \Big| \lambda_i \big( \Theta_{\kappa}^{(a, b)} \textbf{\emph{H}}^{\gamma} \big) \Big| < c_1,
\end{aligned}	
\eeq

\noindent both hold for each $(1/2 - c_2) N \le i \le (1/2 + c_2)N$ with overwhelming probability. Additionally, for any interval $I \subset [-c_1, c_1]$ of length $| I |  \ge N^{-1 + \delta}$,
\beq 
\label{intervalhgammaestimate}
\displaystyle\sup_{0 \le \kappa \le 1} \displaystyle\max_{1 \le a, b \le N} \displaystyle\sup_{0 \le \gamma \le 1} \bigg| \Big\{ i \in [1, N] \colon \lambda_i \big( \Theta_{\kappa}^{(a, b)} \mathbf H^\gamma \big) \in I \Big\} \bigg| \le C |I| N^{1+\delta},
\eeq

\noindent holds with overwhelming probability.

\eel 

\begin{proof} 
	
	The second bound in \eqref{uihgammaestimate} follows from \Cref{l:hgammarigid} and the Weyl interlacing inequality for eigenvalues of symmetric matrices. Furthermore, the proofs of the first bound in \eqref{uihgammaestimate} and \eqref{intervalhgammaestimate} given \eqref{e:Gsbounds} follow from standard arguments (see for example the proofs of \cite[Theorem 2.10]{benaych2016lectures} and \cite[Lemma 7.4]{landon2017convergence}). So, we only establish \eqref{e:Gsbounds}.
	
	To that end, let $K$ be as in \Cref{l:gammagreen}, and fix indices $a, b \in [1, N]$; real numbers $\kappa, \gamma \in [0, 1]$; and a complex number $z \in \mathcal{D}_{K; \delta}$. Set $\textbf{E} = \Theta_{\kappa}^{(a, b)} \textbf{H}^{\gamma}$, $\textbf{V} = (\textbf{E} - z)^{-1} = \{ V_{ij} \}$, and $\Delta = \textbf{H}^{\gamma} - \textbf{E}$. We may assume throughout this proof that $|h_{ab}| \le N^{-\rho}$, for otherwise $\textbf{E} = \textbf{H}^{\gamma}$, and the result follows from \Cref{l:gammagreen}. 
	
	For any $M \in \N$, the resolvent identity \eqref{ab} gives
\beq
\label{e:resolventexpand}
\textbf{V} - \mathbf G^\gamma = \sum_{k=0}^M (\mathbf G^\gamma   \Delta)^k \mathbf G^\gamma + (\mathbf G^\gamma  \Delta)^{M+1} \textbf{V}.
\eeq

Now select $M$ in \eqref{e:resolventexpand} such that $M \rho > 10$. Then \eqref{ab}; \Cref{l:gammagreen}; the deterministic bound \eqref{kij}; and the fact that $\Delta$ is supported on at most two entries, each of which is bounded by $N^{-\rho}$, implies for sufficiently small $\delta > 0$ that 
\begin{flalign}
\label{vestimate} 
\displaystyle\max_{1 \le i, j \le N} |V_{ij}| \le \displaystyle\max_{1 \le i, j \le N}  \big| G_{ij}^{\gamma} \big| + \displaystyle\sum_{k = 0}^M 2^k N^{(k + 1) \delta - k \rho} + 2^{M + 1} N^{(M + 1) \delta - \rho - 10} \eta^{-1} \le N^{\delta},
\end{flalign}

\noindent for sufficiently large $N$, with overwhelming probability. Taking a union bound of \eqref{vestimate} over all $a, b \in [1, N]$; $\kappa$ and $\gamma$ in a $N^{-10}$-net of $[0, 1]$; and $z$ in a $N^{-10}$-net of $\mathcal{D}$, and also applying \eqref{ab}, then yields \eqref{e:Gsbounds}. 
\end{proof}

Next we require the following result that essentially provides level repulsion estimates for $\X_t$. 

\bel\label{l:dyadic}

For all $\omega > 0$, there exist constants $c > 0$ (independent of $\omega$) and $C = C(\omega) > 0$ such that the following holds. Set $M = N^{2 \omega}$; recall $t$ from \eqref{t}; and fix an index $i \in \big[ (1/2 - c) N, (1/2 + c)N \big]$. Then, 
\beq
\label{fmqixsestimate}
\E \bigg[  f_M \big( Q_i(\X_t) \big) \Big] \le C N^{3\omega/2}.
\eeq

\noindent Further fix $\delta > 0$ and, for fixed real numbers $\kappa, \gamma \in [0, 1]$ and indices $1 \le a, b \le N$, abbreviate $\mu_j = \lambda_j \big( \Theta_{\kappa}^{(a, b)} \textbf{\emph{H}}^{\gamma} \big)$ for each $j \in [1, N]$. Then we have with overwhelming probability that
\begin{flalign}
\label{sumlambdailambdaj}
\one_{Q_i (\Theta_{\kappa}^{(a, b)} \textbf{\emph{H}}^{\gamma}) < M} \displaystyle\sum_{j \ne i} \displaystyle\frac{1}{|\mu_j - \mu_i|} \le N^{1 + \omega + \delta}.
\end{flalign} 

\eel
\begin{proof}

	Throughout this proof, we may assume that $\delta < \frac{\omega}{4}$. Define the sets
	\beq
	U_0 = \Big\{ j \in [1, N] \setminus \{ i \} \colon \big| \lambda_j(t) - \lambda_i(t) \big| \le N^{-1 + \delta/2} \Big\}.
	\eeq
	
	\noindent and
	\beq
	U_n = \Big\{ j \in [1, N] \colon 2^{n-1} N^{-1 + \delta/2} < \big| \lambda_j(t) - \lambda_i(t) \big| \le 2^n N^{-1 + \delta/2} \Big\}.
	\eeq
	
	\noindent for each integer $n \ge 1$. 
	
	Now choose the $c > 0$ here with respect to the $K$ from \Cref{l:disperse} to satisfy $c < \frac{1}{4K}$, and define $L = \big\lfloor \log_2 (2c N^{1 - \delta / 2}) \big\rfloor$. Then \Cref{l:disperse} and \Cref{l:hgammarigid} together imply (after further decreasing $c$ if necessary) that 
	\beq
	\label{e:unbound}
	|U_n| \le  C 2^n N^{\delta}.
	\eeq 
	
	\noindent holds with overwhelming probability, for each $n \in [0, L]$. Next, for any $\theta \in (0, 1)$, also define the event 
	\begin{flalign} 
	E (\theta) = \bigg\{ \displaystyle\min \big\{ \lambda_i (t) - \lambda_{i - 1} (t), \lambda_{i + 1} (t) - \lambda_i (t) \big\} > \displaystyle\frac{\theta}{N} \bigg\},
	\end{flalign}  
	
	\noindent and let $E(\theta)^c$ denote the complement of $E (\theta)$. Then \eqref{e:unbound} implies with overwhelming probability that
	\beq\label{e:dispersion1}
	\frac{\one_{E (\theta)}}{N^2} \sum_{n=0}^L \sum_{j \in U_n} \big|\lambda_j(t) - \lambda_i(t) \big|^{-2} \le  C N^{\delta }\theta^{-2}.
	\eeq
	Further, we deterministically have that
	\beq\label{e:dispersion2}
	\frac{1}{N^2} \displaystyle\sum_{n = L + 1}^{\infty} \sum_{j \in U_n} \big| \lambda_j(t) - \lambda_i(t) \big|^{-2} \le C N^{-1}.
	\eeq
	
	\noindent Then combining \eqref{e:dispersion1} and \eqref{e:dispersion2} bounds 
	\begin{flalign} 
	\label{ethetaqixs}
	\mathbb{E} \Big[ Q_i \big( \Theta_{\kappa}^{(a, b)} \textbf{X}_t \big) \one_{E (\theta)}\Big] < C N^{\delta} \theta^{-2}.
	\end{flalign} 
	
	\noindent On $E(\theta)^c$, we use the third part of \Cref{l:xtlevel} with $\eps = \theta$ and $r= 1/3$, and the fact that $\big| f_M (x) \big| < M$ holds for all $x > 0$, to deduce that 
	\begin{flalign} 
	\label{ethetaqixs2}
	\mathbb{E} \bigg[ f_M \Big( Q_i \big(  \Theta_{\kappa}^{(a, b)} \textbf{X}_t \big) \Big) \one_{E (\theta)^c} \bigg] < C N^{2 \omega + \delta} \theta^{5/3}. 
	\end{flalign} 	
	
	\noindent Then selecting $\theta = N^{- \omega / 2}$, using the fact that $\delta < \frac{\omega}{4}$, and combining \eqref{ethetaqixs} and \eqref{ethetaqixs2} yields \eqref{fmqixsestimate}. We omit the proof of \eqref{sumlambdailambdaj}, as it is entirely analogous, and follows from replacing the above application of \Cref{l:disperse} and \Cref{l:hgammarigid} (to establish \eqref{e:unbound}) with \eqref{intervalhgammaestimate} and the second bound in \eqref{uihgammaestimate}, respectively, and using the fact that $\mu_i - \mu_{i - 1}, \mu_{i + 1} - \mu_i \ge N^{-\omega}$ holds on the event that $Q_i \big( \Theta_{\kappa}^{(a, b)} \textbf{H}^{\gamma} \big) < M$.
\end{proof}

Now we can establish the derivative bounds given by \Cref{l:smoothev} and \Cref{deterministicqestimate}.  

\begin{proof}[Proof of \Cref{l:smoothev} (Outline)]
In outline, the bound \eqref{e:qibound} is proven by expanding $\partial^{(k)}_{ab} Q_i(\bH^\gamma)$ using contour integration into a sum of terms which are then bounded individually. Since the proof of \eqref{derivativeuihgammaestimate} uses a similar expansion, we only discuss that of \eqref{e:qibound} here (for the former, see the proof of \cite[Proposition 4.2]{bourgade2017eigenvector} for further details).

Our claim \eqref{e:qibound} is essentially the same as that of \cite[Proposition 4.6]{huang2015bulk}, but there are two differences. First, one of our hypotheses is weaker: we only have complete delocalization at small energies and not throughout the spectrum. Second, our conclusion is stronger: \cite[Proposition 4.6]{huang2015bulk} bounded derivatives up to third order, but here we bound fourth order derivatives. This extension to fourth order derivatives parallels the proof for the third order derivatives and requires no new ideas. Therefore, let us only show how the proof in \cite[Proposition 4.6]{huang2015bulk} may be modified to accommodate the fact that our delocalization estimate is weaker than the one used in that reference. In what follows, we also assume for notational convenience that $\kappa = 1$, so that $\Theta_{\kappa}^{(a, b)} (\textbf{H}^{\gamma}) = \textbf{H}^{\gamma}$, as the proof for general $\kappa \in [0, 1]$ is entirely analogous by replacing our use of \eqref{estimateuihgamma} below by \Cref{l:perturbS}. 

For any vector $\mathbf v$, let $\mathbf v^*$ denote its transpose. Set $\theta_{jk} = \u_j^* {\mathbf V} \u_k$, where ${\mathbf V}= {\mathbf V}^{(a, b)} = \{ V_{ij} \}$ is the $N \times N$ matrix whose entries are zero except for $V_{ab} = V_{ba} =1$. In the proof of \cite[Proposition 4.6]{huang2015bulk}, $\partial^{(k)}_{ab} Q_i(\bH^\gamma)$ was expanded into a sum of certain terms using a contour integral representation and resolvent identities. For instance, in the expansion of $\partial^{(3)}_{ab} Q_i(\bH^\gamma)$ there are $13$ distinct terms, which are listed after line (4.18) in \cite[Proposition 4.6]{huang2015bulk}. Setting $\lambda_i = \lambda_i( \mathbf H^\gamma)$, one such term is
\beq\label{e:qsum}
\frac{1}{N^2} \sum_{\substack{1 \le j_1, j_2, j_3 \le N \\ j_1, j_2, j_3 \neq i}} \frac{\theta_{j_1 j_2} \theta_{j_2 j_3} \theta_{j_3 j_1} } { (\lambda_i - \lambda_{j_1})^3 (\lambda_i  - \lambda_{j_2}) (\lambda_i - \lambda_{j_3})}.
\eeq
The terms produced by expanding $\partial^{(k)}_{ab} Q_i(\bH^\gamma)$ are fractions with a product of $k$ $\theta_{\alpha\beta}$ terms in the numerator, where each of $\alpha, \beta$ may be a summation index or $i$, and a product of $k+2$ eigenvalue differences $\lambda_i - \lambda_j$ in the denominator, where $j$ is a summation index. We call $k$ the \emph{order} of such a term. The proof of \cite[Proposition 4.6]{huang2015bulk} shows that to prove the claim \eqref{e:qibound}, it suffices to bound each of the order $k$ terms appearing in its expansion by $C N^{(2k + 2)\delta + (k + 2)\omega}$. 

For illustrative purposes, we consider just the term \eqref{e:qsum} in the $k=3$ case here; other terms of the same order and the cases $k \in \{ 1, 2, 4 \}$ are analogous. So, let us show that
\beq\label{e:qsum2}
\left| \frac{1}{N^2} \sum_{\substack{1 \le j_1, j_2, j_3 \le N \\ j_1, j_2, j_3 \neq i}} \frac{\theta_{j_1 j_2} \theta_{j_2 j_3} \theta_{j_3 j_1} } { (\lambda_i - \lambda_{j_1})^3 (\lambda_i  - \lambda_{j_2}) (\lambda_i - \lambda_{j_3})} \right| \le C N^{8 \delta + 5\omega} .
\eeq

\noindent To prove \eqref{e:qsum2}, we consider various cases depending on the locations of the eigenvalues  $\lambda_{j_1}, \lambda_{j_2},  \lambda_{j_3}$. Let $K$ be the constant from \Cref{l:gammagreen}, and set $c = (2K)^{-1}$. When $\lambda_{j_1}, \lambda_{j_2},  \lambda_{j_3} \in [- 2c , 2c ]$,  \eqref{estimateuihgamma} shows the corresponding eigenvectors are completely delocalized, and the proof of \cite[Proposition 4.6]{huang2015bulk} requires no modification. There are three remaining cases: exactly one of the $j_{\ell}$ is such that $|\lambda_{j_\ell}| > 2c$, exactly two are, or all three are. 

In the first case, suppose for example that $|\lambda_{j_2}| > 2c$. Then, \eqref{estimateuihgamma} implies
 \begin{align}
| \theta_{j_1 j_2} | &\le \Big( \big| u_{j_1}(a) \big| \big| u_{j_2}(b) \big|  + \big|u_{j_1}(b) \big| \big| u_{j_2}(a) \big| \Big) \le  N^{-1/2 + \delta} \Big( \big| u_{j_2}(a) \big| +  \big| u_{j_2}(b) \big| \Big) ,\\
| \theta_{j_2 j_3} | &\le \Big( \big|u_{j_2}(a)\big| \big| u_{j_3}(b) \big|  + \big| u_{j_2}(b) \big| \big| u_{j_3}(a) \big| \Big) \le  N^{-1/2 + \delta} \Big( \big| u_{j_2}(a) \big| +  \big| u_{j_2}(b) \big| \Big), \\
|\theta_{j_3 j_1} | & \le \Big( \big|u_{j_1}(a) \big| \big| u_{j_3}(b) \big|  + \big| u_{j_1}(b) \big| \big| u_{j_3}(a) \big| \Big) \le N^{-1 + \delta},
\end{align}
with overwhelming probability. Inserting these bounds in \eqref{e:qsum} decouples the sum into a product of a sum over $j_2$ and a sum over $j_1, j_3$:
\beq\label{e:247}
 \eqref{e:qsum}  \le c^{-1} N^{-4 + 3\delta} \Bigg(\sum_{j_2 = 1}^N \Big( \big| u_{j_2}(a) \big| +  \big| u_{j_2}(b) \big| \Big)^2 \Bigg)  \Bigg( \sum_{\substack{1 \le j_1, j_3 \le N}} \frac{1}{|\lambda_i - \lambda_{j_1}|^3  |\lambda_i - \lambda_{j_3}|} \Bigg).
\eeq
Here we used $|\lambda_i - \lambda_{j_2}| > c$. The first factor is at most a constant, since the matrix of eigenvectors is orthonormal: 
 \begin{equation}
 \label{uj2ab}
  \sum_{j_2 = 1}^N \Big( \big| u_{j_2}(a) \big| + \big| u_{j_2}(b) \big| \Big)^2 \le  2 \sum_{j_2 = 1}^N \Big( \big| u_{j_2}(a) \big|^2  +  \big| u_{j_2}(b) \big|^2 \Big) = 4.
 \end{equation}
 For the second factor, the assumption \eqref{e:qiassume} implies that $|\lambda_i - \lambda_{i \pm 1} | \ge N^{-1 - \omega}$, and so \eqref{sumlambdailambdaj} yields
 \beq
 \sum_{j\neq i} |\lambda_j - \lambda_i|^{-1}  \le C N^{1 + \omega + \delta}.
 \eeq
 Further, for $k\ge 2$, the hypothesis \eqref{e:qiassume} yields
 \beq
 \label{lambdaijk}
 \sum_{j\neq i} |\lambda_j - \lambda_i|^{-k}  \le \Bigg(  \sum_{j\neq i} |\lambda_j - \lambda_i|^{-2} \Bigg)^{k/2} \le C^{k / 2} N^{k(1 + \omega + \delta)}.
 \eeq
These inequalities imply the second factor of \eqref{e:247} is at most $C N^{4 + 4\omega + 4 \delta}$, and so \eqref{e:qsum} is at most $C N^{7 \delta + 4 \omega}$. 

In the second case, suppose for example that $|\lambda_{j_2}| > c$ and $|\lambda_{j_3}| > c$. Since then $|\lambda_i - \lambda_{j_2}|$ and $|\lambda_i - \lambda_{j_3}|$ are bounded below by $c$, in this case \eqref{e:qsum} is bounded above by
\beq
\frac{1}{c^2 N^2} \sum_{\substack{ 1 \le j_1, j_2, j_3 \le N \\ j_1, j_2, j_3 \neq i}} \left| \frac{\theta_{j_1 j_2} \theta_{j_2 j_3} \theta_{j_3 j_1} } { (\lambda_i - \lambda_{j_1})^3} \right|.
\eeq

\noindent Proceeding as in the previous case, we find that \eqref{e:qsum} is bounded above by 
\begin{flalign}
c^{-2} N^{2 \delta - 3} \Bigg(\sum_{1 \le j_2, j_3 \le N} & \Big( \big|u_{j_2}(a) \big| +  \big| u_{j_2}(b) \big| \Big) \Big( \big|u_{j_3}(a) \big| +  \big| u_{j_3}(b) \big| \Big) \\
& \quad \times \Big( \big|u_{j_3}(a) \big| \big|u_{j_2}(b) \big|  + \big| u_{j_3}(b) \big| \big| u_{j_2}(a) \big| \Big) \Bigg) \Bigg( \sum_{j_1 \ne i} \frac{1  } { |\lambda_i - \lambda_{j_1}|^3 }  \Bigg).
\end{flalign}

\noindent By \eqref{lambdaijk}, the sum over $j_1$ is at most $C N^{3 + 3 \omega + 3 \delta}$. Moreover, by the orthogonality of the eigenvectors of $\textbf{H}^{\gamma}$ (following \eqref{uj2ab}), the sum over $j_2$ and $j_3$ is bounded by $8$. Thus, \eqref{e:qsum} is bounded above by $C N^{5 \delta + 3 \omega}$. 

Finally, when $|\lambda_{j_\ell}| > c$ for all $\ell$, \eqref{e:qsum} is bounded by 
\begin{flalign}
\frac{1}{N^2 c^5} \sum_{1 \le j_1, j_2, j_3 \le N} \left| \theta_{j_1 j_2} \theta_{j_2 j_3} \theta_{j_3 j_1}\right| \le \frac{1}{N^2 c^5} & \sum_{1 \le j_1, j_2, j_3 \le N} \Big( \big|u_{j_1}(a) \big| \big| u_{j_2}(b) \big|  + \big| u_{j_1}(b) \big| \big| u_{j_2}(a) \big| \Big)  \\
& \qquad \qquad \times \Big( \big|u_{j_2}(a) \big| \big| u_{j_3}(b) \big|  + \big| u_{j_2}(b) \big| \big| u_{j_3}(a) \big| \Big) \\
& \qquad \qquad \times \Big( \big|u_{j_1}(a) \big| \big| u_{j_3}(b) \big|  + \big| u_{j_1}(b) \big| \big| u_{j_3}(a) \big| \Big).
	\end{flalign}
	
	\noindent Again by the orthogonality of the eigenvectors of $\textbf{H}^{\gamma}$ (following \eqref{uj2ab}), the latter is bounded by $8$, and so \eqref{e:qsum} is bounded above by obtain $8 c^{-5} N^{-2}$.
	
	This completes our demonstration of how to bound the sum \eqref{e:qsum} and concludes the proof.
\end{proof}

\begin{proof}[Proof of \Cref{deterministicqestimate} (Outline)]
	
	We again only discuss the first bound in \eqref{e:detbounds}, as the proof of the second is similar. To that end observe, since we have assumed $Q_i (\textbf{H}^{\gamma}) \le N^{2 \omega}$, we must have $|\lambda_i - \lambda_j| \ge N^{-1 - \omega}$, for each $j \in [1, N] \setminus \{ i \}$. As in the proof of \Cref{l:smoothev}, the derivative $\partial_{ab}^{(k)} \big( Q_i (\Theta_{\kappa}^{(a, b)} \textbf{H}^{\gamma}) \big)$ for $k \le 4$ can be expressed as a sum of a uniformly bounded number of terms similar to \eqref{e:qsum}, in which at most four indices $j_k$ are being summed over and in which the denominator is of degree at most six in the gaps $\lambda_i - \lambda_{j_k}$.  Thus, each such term is bounded by at most $N^{6 + 6 \omega}$, and so their sum is bounded by a multiple of $N^{10 + 6 \omega}$. This yields the first estimate in \eqref{e:detbounds} and, as mentioned previously, the proof of the second is omitted.  
\end{proof}

\subsection{Proof of \Cref{l:lrh}}

\label{s:vectorcompare}

Now we can establish \Cref{l:lrh}. 

\begin{proof}[Proof of \Cref{l:lrh}]

It suffices to show that there exists some $\omega > 0$ such that, if $M = N^{2 \omega}$, then  
\beq
\label{e:pgoal} 
\mathbb{E} \Big[ f_M \big( Q_i (\bH^\gamma) \big) \Big] \le 2 N^{3 \omega /2}, 
\eeq

\noindent since then a Markov inequality would imply
\beq\label{wouldimply}
 \P \big( Q_i(\bH^\gamma)  \ge N^{2\omega} \big) \le \P \Big( f_M \big(Q_i(\bH^\gamma) \big)  \ge N^{2 \omega} \Big) \le N^{-2 \omega} \E \Big[ f_M \big( Q_i(\bH^\gamma) \big) \Big] \le 2 N^{- \omega / 2},
\eeq
and the lemma follows after setting $\upsilon = 2 \omega$.
To prove \eqref{e:pgoal}, we apply \Cref{gcompare} to interpolate between $\X_t = \bH^0$ and $\bH^{\gamma}$, carrying the level repulsion estimate \eqref{fmqixsestimate} from the former to the latter. To implement this argument, let us take $c_1> 0$ sufficiently small so that $i \in \big[ (1/2- c_1) N, (1/2 + c_1) N \big]$ implies that $\big| \lambda_i (\Theta_{\kappa}^{(a, b)} \textbf{H}^{\gamma}) \big|$ is less than the $c$ from \Cref{l:smoothev} for each $\gamma, \kappa \in [0, 1]$ and $1 \le a, b \le N$ with overwhelming probability; such a $c_1$ exists by the second bound of \eqref{uihgammaestimate}. We take $c=c_1$ in the statement of \Cref{l:lrh}.
Further, we take $\omega, \delta > 0$ sufficiently small so that $100 (\omega + \delta)$ is less than the constant $c$ from \Cref{gcompare}, and set $M = N^{2 \omega}$.

Then, we may apply \Cref{gcompare}, with the $F (\textbf{A})$ there equal to $f_M \big( Q_i (\textbf{A}) \big)$ here; the $K$ there equal to $C N^{40 (\omega + \delta)}$ here; and the $L$ there equal to $C N^{15 + 7 \omega}$ here. Then \eqref{definitionk} and \eqref{definitionl} follow from \Cref{l:smoothev} and \Cref{deterministicqestimate}, respectively, and so \Cref{gcompare} implies that 
\begin{flalign*}
\mathbb{E} \Big[ f_M \big( Q_i (\textbf{H}^{\gamma}) \big) \Big] \le \mathbb{E} \Big[ f_M \big( Q_i (\textbf{X}_t) \big) \Big] + C N^{- c / 2} < 2 N^{3 \omega / 2}, 
\end{flalign*}

\noindent where we used \eqref{fmqixsestimate} to deduce the second inequality. This verifies \eqref{e:pgoal}. 
\end{proof}

\section{Dynamics}\label{s:dynamics}

This section determines the eigenvector statistics of $\X_t$. Our main goal is a proof of \Cref{t:dynamics}. In \Cref{s:emf} we recall the definition of the eigenvector moment flow from \cite{bourgade2013eigenvector} and some of its properties. \Cref{s:3aux} contains continuity estimates used in the proof of \Cref{t:dynamics}. In \Cref{s:relax} we use the eigenvector moment flow to establish \Cref{t:dynamics}, assuming several results that will be shown in \Cref{s:dichoproof} and \Cref{s:proofofcrefaux1}. 

\subsection{Eigenvector moment flow}\label{s:emf}

 Recall the matrix $\textbf{X}_s$ from \eqref{e:Xt} and that its eigenvalues are given by $\lambda_1 (s) \le \lambda_2 (s) \le \cdots \le \lambda_N (s)$ with associated unit eigenvectors $\textbf{u}_1 (s), \textbf{u}_2 (s), \ldots , \textbf{u}_N (s)$, respectively. By \cite[Theorem 2.3]{bourgade2013eigenvector}, these eigenvalues and eigenvectors are governed by two stochastic differential equations (SDEs):
\begin{align}
d\lambda_k(s) &= \frac{d b_{kk}(s)}{\sqrt{N}} + \frac{1}{N} \sum_{l \neq k} \frac{ds}{\lambda_k(s) - \lambda_l(s)},\\
d\u_k(s) &= \frac{1}{\sqrt{N}} \sum_{l\neq k} \frac{db_{kl}(s)}{\lambda_k(s) - \lambda_l(s)} \u_l(s) - \frac{1}{2N} \sum_{l\neq k} \frac{ds}{(\lambda_k - \lambda_l)^2} \u_k(s),
\end{align}
where $\big( b_{ij}(s) \big)_{1 \le i \le j \le N}$ are mutually independent Brownian motions with variance $1 + \one_{i = j}$. The first equation, for the eigenvalues, is called Dyson Brownian motion. The second, for the eigenvectors, is called the Dyson vector flow. Using these SDEs, we define the stochastic processes $\bm{\lambda} = \big( \bm{\lambda} (s) \big)_{0 \le s \le 1}$ and $\u = \big( \u (s) \big)_{0 \le s \le 1}$.

A key tool for analyzing the Dyson vector flow is the \emph{eigenvector moment flow}, introduced in  \cite[Section 3.1]{bourgade2013eigenvector}, which characterizes the time evolution of the observable $f_s = f_{\mathbf{\bm \lambda}, s}$  defined by
\beq\label{e:observable2}
f_s (\bm \xi) = f_{\mathbf{\bm \lambda}, s}(\bm \xi) = \E \big[ Q^{j_1,\dots, j_m}_{i_1, \dots , i_m} (s) \mid {\bm \lambda} \big],
\eeq
where we recall $Q^{j_1,\dots, j_m}_{i_1, \dots , i_m} (s)$ was defined in \eqref{e:Qdef}.

\bet[{\cite[Theorem 3.1]{bourgade2013eigenvector}}] 

\label{t:momentflow} 

Let $\q \in \mathbb R^N$ be a unit vector and, for each $s \in [0, 1]$, set 
\beq
\label{cij}
 {c_{ij}(s) = N^{-1} \big(\lambda_i(s) - \lambda_j(s) \big)^{-2}}.
\eeq
Then,	 
\beq\label{e:evmf}
\partial_s f_s = \B(s) f_s, \quad \text{where} \quad \B(s) f_s (\bm \xi) = \sum_{i\neq j} c_{ij} (s) 2 \xi_i (1 + 2\xi_j) \big( f_s (\bm \xi^{ij})- f_s (\bm \xi) \big).
\eeq
\eet

Recall from \Cref{s:dynamicsstatements} that we view $N$-tuples $\bm \xi = (\xi_1, \xi_2, \ldots , \xi_N) \in \mathbb N^N$ as particle configurations, with $\xi_k$ particles at location $k$ for each $k \in [1, N]$; the total number of particles in this configuration is $n = \mathcal{N} (\bm \xi) = \sum_{j = 1}^N \xi_j$. We label the locations of these particles in non-decreasing order by 
\beq
x_1(\bm \xi) \le \cdots \le x_n (\bm \xi).
\eeq
Given another particle configuration ${\bm \zeta}$ with the same number $n$ of particles, whose locations are labeled by $\big( y_j ({\bm \zeta}) \big)$ in non-decreasing order, we define the distance 
\beq
d(\bm \xi, \bm \zeta) = \sum_{j =1}^n \big| x_j (\bm \xi) - y_j (\bm  \zeta) \big|.
\eeq

Next, we recall $\fc, \eta, \psi$ from \eqref{e:etadefine}, fix $n \in \mathbb N$, and define the parameter 
\begin{flalign}
\label{e:ell}
\ell = \ell (n) = \psi^{4n +1} N^{1+ \fd}\eta,\quad\text{where}\quad \fd = \mathfrak{d} (n) = 50 (n + 1) \fc. 
\end{flalign}  

\noindent Recalling $t$ from \eqref{t} (which satisfies \eqref{c1c2t0}), we also set
\beq\label{e:taudefine}
\tau = \tau (n) = t - N^{7 \fd} \psi \eta, \quad t_0 =  t - \psi \eta. 
\eeq

\noindent These are chosen so that for fixed $n > 0$, recalling the choices from \eqref{e:etadefine} (after selecting $\mathfrak{c} = \mathfrak{c} (n) > 0$ to be sufficiently small),
\begin{equation}
N^{-1/2} \ll \psi \eta \ll N^{7 \mathfrak{d}} \psi \eta \ll \tau <  t_0 < t < t_0 + \displaystyle\frac{\ell}{N} \ll 1.
\end{equation}
Recalling the operator $\B = \mathcal{B} (s)$ from \eqref{e:evmf}, we decompose it into the sum of a ``short range'' and ``long range'' operator, given explicitly by $\B = {\mathcal S} + \L$, where $\mathcal{S} =  \mathcal{S}_n = \mathcal{S}_n (s)$ and $\mathcal{L} = \mathcal{L}_n = \mathcal{L}_n (s)$ are defined by
\begin{align}
\label{sfs}
({\mathcal S} f_s)(\bm \xi) &= \sum_{ 0 < |j - k| \le \ell} c_{jk} (s) 2 \xi_j ( 1 + 2 \xi_k) \left(f_s (\bm \xi^{jk}) - f_s (\bm \xi)\right),\\
(\L f_s)(\bm \xi) &= \sum_{ |j - k| > \ell} c_{jk} (s) 2 \xi_j ( 1 + 2 \xi_k) \left(f_s (\bm \xi^{jk}) - f_s (\bm \xi)\right).
\end{align}
We let $\U_{\B}(s_1, s_2)$ be the semigroup associated with $\B$ and likewise define $\U_{{\mathcal S}}(s_1, s_2)$ and $\U_{\L}(s_1, s_2)$.

 We also let $\mathcal F_{t_0}$ denote the $\sigma$-algebra generated by $\{ \X_s \}_{0 \le s \le t_0}$, and define
\beq\label{e:observable3}
h_s (\bm \xi) = h_{\mathbf{\bm \lambda}, s}(\bm \xi) = \E \big[ Q^{j_1,\dots, j_m}_{i_1, \dots , i_m} (s) \mid {\bm \lambda}, \mathcal F_{t_0} \big] = \E \big[ Q^{j_1,\dots, j_m}_{i_1, \dots , i_m} (s) \mid {\bm \lambda}, \mathbf X_{t_0} \big].
\eeq
In the last equality, we used the Markov properties for Dyson Brownian motion and the eigenvector moment flow. Informally, $h_s (\bm \xi)$ corresponds to a solution of the eigenvector moment flow dynamics, run for time $s - t_0$, with initial data $\mathbf X_{t_0}$.

For consistency with \cite{bourgade2017eigenvector} we introduce the following notation. Let $K > 1$ be such that $K^{-1}$ is less than the constant $c$ from \Cref{l:rhocontinuity} and $K$ is greater than those from \Cref{l:thatlaw}, \Cref{l:utdeloc}, \Cref{l:gammagreen}, and \Cref{l:disperse}; and define 
\beq\label{e:dkappa}
r = \displaystyle\frac{1}{2K}; \qquad \mathcal D_r = \left\{ z = E + \iu \eta \colon \, |E| \le  r, \,  \frac{\psi^4}{N} \le \eta \le r \right\}.
\eeq
We define the function $\widetilde d = \widetilde{d}_n$ on $n$-particle configurations by 
\begin{equation}
\widetilde d (\bm \xi, \bm \zeta) = \max_{1\le \beta \le n} \bigg| \Big\{   1\le i \le N \colon \, \big| \gamma_i(t_0) \big| \le  r ,\,  \min \big\{ x_\beta(\bm \xi), y_\beta(\bm \zeta) \big\} \le i   \le  \max \big\{ x_\beta(\bm \xi),  y_\beta(\bm \zeta) \big\} \Big\} \bigg|.
\end{equation}

The next lemma provides several estimates necessary to analyze the eigenvector moment flow. Its first and second parts follow from \Cref{l:thatlaw} and \Cref{l:utdeloc}, respectively, together with a standard stochastic continuity argument. Its third part constitutes a special case of \cite[Corollary 3.3]{bourgade2017eigenvector}, whose assumptions are verified by \Cref{l:Xlocallaw}. Its fourth part is a consequence of \cite[(3.48)]{bourgade2017eigenvector}. In what follows, we recall $m_{\text{fc}} (s, z)$ from \eqref{e:mfcdef}, $\gamma_i (s)$ from \eqref{e:gammaidef}, and $\textbf{R} (s, z)$ from \eqref{rsz}.

\bel[{\cite{bourgade2017eigenvector}}]

\label{l:goodset} 

The initial data $\textbf{\emph{X}}_0$ and the Dyson Brownian motion $\{ \X_s \}_{0 < s \le t}$ together induce a measure  $\mathcal M$ on the space of eigenvalue and eigenvector  trajectories $\big( \bm\lambda(s), \u(s) \big)_{0\le s \le t}$ on which the following event of trajectories holds with overwhelming probability. 
\begin{enumerate}
\item Eigenvalue rigidity holds: $\sup_{t_0 \le s \le t} \big| m_N (s,z) - \mfc (s,z) \big| \le \psi (N\eta)^{-1}$ uniformly for $z\in \mathcal D_r$ and  $\sup_{t_0 \le s \le t} \big| \lambda_i(s) - \gamma_i(s) \big| \le \psi N^{-1}$ uniformly for indices $i$ such that $\big| \gamma_i(s) \big| \le r$. 

\item Delocalization holds: Conditional on $\bm \lambda$, for any $\textbf{\emph{q}} \in \mathbb{R}^N$ with stable support we have that
\beq\label{e:deloc}
\displaystyle\sup_{z \in \mathcal D_r} \sup_{t_0 \le s \le t} \Big| \big\langle \q , \R(s,z) \q \big\rangle  \Big| \le C(q) \psi; \qquad \displaystyle\sup_{z \in \mathcal D_{r}} \sup_{t_0 \le s \le t}  N \big\langle \u_i(s), \q \big\rangle^2   \le C(q) \psi,
\eeq

\noindent where $C (q) >0$ is a constant depending on $|\supp \q|$.

\item Finite speed of propagation holds: Let $n > 0$ be an integer, and abbreviate $\ell = \ell (n)$ and $\mathcal{S} = \mathcal{S}_n$.   Conditional on $\bm \lambda$, we have the following estimate that is uniform in any function $g \colon \big\{ \bm \xi \in \mathbb N^N: \mathcal{N} (\bm \xi) = n \big\} \rightarrow \mathbb{R}$. For a particle configuration $\bm \xi \in \mathbb N^N$ with $\mathcal{N} (\bm \xi) = n$ such that $\widetilde d_n (\bm \xi, \bm \zeta		) \ge \psi \ell$ for each $\bm \zeta$ in the support of $g$, we have
\beq
\label{e:finitespeed}
\sup_{t_0 \le s \le t} \big| \U_{{\mathcal S}}(t_0,s) g (\bm \xi) \big| \le  N^n e^{-c\psi} \| g \|_{\infty}.
\eeq
\item  For  any interval $I \subset \big[ -  r,  r \big]$ of length $|I|  \ge \psi N^{-1}$, we have 
\beq
\label{e:xtdispersion}
 C^{-1} | I | N \le \Big| \big\{ i \colon \lambda_i(s) \in I \big\} \Big| \le C |I| N,
\eeq

\noindent uniformly in $s \in [t_0, t]$. 

\end{enumerate}

\eel

The next estimate on the short range operator $\mathcal S$ is a consequence of \cite[Lemma 3.5]{bourgade2017eigenvector} (whose conditions are verified by \Cref{l:Xlocallaw} and \Cref{l:hgammarigid}). 

\bel[{\cite[Lemma 3.5]{bourgade2017eigenvector}}]\label{l:shortrange} 

There exists a constant $C > 0$ such that the following holds with overwhelming probability with respect to $\mathcal{F}_{t_0}$. Fix an integer $n > 0$, and abbreviate $\ell = \ell (n)$ (recall \eqref{e:ell}) and $\mathcal{S} = \mathcal{S}_n$. There exists an event $\mathcal E$ of trajectories $\big( \bm\lambda(s), \u(s) \big)_{t_0 \le s \le t}$ of overwhelming probability on which we have
\beq
\displaystyle\sup_{s \in [t_0, t]} \Big| \big( \U_\B(t_0,s)  h_{t_0}  - \U_{\mathcal S}(t_0,s)  h_{t_0}\eer \big) \left( \bm \xi \right) \Big| \le C \frac{ \psi^{ n } N(t - t_0)  } {\ell}
\eeq

\noindent for any configuration $\bm \xi \in \mathbb N^N$ such that $\mathcal{N} (\bm \xi) = n$ and $\supp \bm \xi \subset \big[(1/2 - c) N - 2\psi \ell, (1/2 + c) N  + 2\psi\ell \big]$.
\eel

\subsection{Continuity estimates}\label{s:3aux}

To prepare for the proof of \Cref{t:dynamics}, we require the following continuity estimates for entries of $\textbf{R} (t, z)$. We recall $\mathfrak{a}$ from \eqref{e:parameters}; $t_0$ and $\tau$ from \eqref{e:taudefine}; and $r$ from \eqref{e:dkappa}, and define
\beq 
\widehat{\mathcal{D}} = \left\{ z = E + \iu \eta \colon \, |E| \le \displaystyle\frac{ r}{4}, \,  N^{-\mathfrak{a}} \le \eta \le \displaystyle\frac{r}{4} \right\}.
\eeq

\bel

\label{l:tconcentration} 

Fix an integer $n > 0$, a real number $ \delta >0$, and a unit vector $\q$ with stable support; set $q = \left| \supp \q \right|$, and abbreviate $\tau = \tau (n)$. Then, there exist constants $c > 0$ (independent of $n$, $\delta$, and $q$) and $C = C(\delta, q, n) >0$ such that, uniformly in $t_1, t_2 \in  [t_0 , t] \eer $ with $t_1 < t_2$, we have
\begin{flalign}
\label{rt1t2estimate}	
& \displaystyle\sup_{z \in \widehat{\mathcal{D}}} \Big| \big\langle \q , \R(t_1, z ) \q \big\rangle   - \big\langle \q , \R(t_2, z ) \q \big\rangle \Big|  \le C \frac{\psi^4}{\sqrt{N\eta}} +  CN^{\delta} \bigg( \frac{t_2 - t_1}{t_2 - \tau} + N^{-c} \bigg) ; \\
& \displaystyle\sup_{\substack{z_1, z_2 \in \widehat{\mathcal{D}} \\ \Im z_1 = \Im z_2}} \Big| \big\langle \q , \R(t_1, z_1 ) \q \big\rangle   - \big\langle \q , \R(t_1, z_2 ) \q \big\rangle \Big|  \le C \frac{\psi^4}{\sqrt{N\eta}} + C N^{\delta} \bigg(N^{-c} + \displaystyle\frac{|z_1 - z_2|  + \Im z_1  }{t_1 - \tau} \bigg), \label{e:secondconclusion}
\end{flalign}

\noindent both with overwhelming probability.
\eel

\begin{proof}[Proof of \Cref{l:tconcentration}]
	
	For $s\ge \tau$, define
	\beq
	\label{rszi}
	r_i (s,z) = \frac{1}{\lambda_i(\tau ) - z - (s - \tau) \mfc(s,z)},
	\eeq
	which is similar to the terms appearing in the definition of the free convolution \eqref{e:mfcdef} (but, in a sense, ``started'' at $\bm{\lambda} (\tau)$ instead of at $\bm{\lambda} (0)$). 
	
	Now let us apply \cite[Theorem 2.1]{bourgade2017eigenvector}, with the $t$ there equal to $s - \tau$ here and the $H_0$ given by $\textbf{X}_{\tau}$ here; the assumptions of that theorem are verified by \eqref{rijestimate}, \eqref{mszestimate}, and the facts that $\tau \gg N^{-1 / 2}$ and $s - \tau \ge  t_0 \eer - \tau \gg N^{-1 / 2}$. For any $s \in  [t_0 , t] \eer$, this gives with overwhelming probability that
	\begin{equation}\label{e:t21}
	\left| \big\langle \textbf{q}, \R(s, z) \q \big\rangle - \sum_{i=1}^N \big\langle \u_i (\tau), \q \big\rangle^2 r_i(s, z) \right| \le \frac{\psi^2}{\sqrt{N\eta}}\Im \left( \sum_{i=1}^N \big\langle \u_i (\tau), \q \big\rangle^2 \big| r_i (s, z) \big| \right).
	\end{equation}
	
	\noindent  Thus, \eqref{ab}, \eqref{kij}, and a union bound over $s$ in an $N^{-10}$-net of $ [t_0 , t] \eer$ together yield that \eqref{e:t21} holds with overwhelming probability, uniformly in $s \in  [t_0 , t] \eer$.
	
	To bound the right side of \eqref{e:t21}, observe that \eqref{rijestimate2},  \cite[(2.3)]{bourgade2017eigenvector}, and the exchangeability of the eigenvector entries together yield the bound 
	\begin{equation}
	\label{uiqestimate}
	\sum_{i=1}^N \big\langle \u_i (\tau),  \textbf{q} \big\rangle^2 \big|  r_i(s,z) \big| \le  C \psi (\log N)^2,
	\end{equation}
	
	\noindent uniformly for any standard basis vector $\textbf{q} \in \{ \textbf{e}_1, \textbf{e}_2, \ldots , \textbf{e}_N \}$. Therefore, \eqref{uiqestimate} also holds for unit vectors $\q \in \mathbb{R}^N$ of stable support (where the $C$ there now depends on $q = |\supp \textbf{q}|$), by expanding $\q$ in the standard basis and using the inequality $(a+ b)^2 \le 2(a^2 + b^2)$ on the products $\langle \u_i , \e_j \rangle \langle \u_i , \e_k \rangle$ that appear in the corresponding expansion of $\big\langle \u_i (\tau),  \q \big\rangle^2$. Thus, \eqref{e:t21} and \eqref{uiqestimate} together imply  
	\begin{equation}\label{e:t22}
	\left| \big\langle \textbf{q}, \R(s, z) \q \big\rangle - \sum_{i=1}^N \big\langle \u_i (\tau), \q \big\rangle^2 r_i(s, z) \right| \le \displaystyle\frac{C(q)\psi^4}{\sqrt{N\eta}}. 
	\end{equation}
	
	We now establish \eqref{rt1t2estimate} by subtracting \eqref{e:t22} evaluated at $s=t_2$ from that equation evaluated at $s=t_1$. To that end, we have from \eqref{rszi} that
	\beq 
	\label{rit1t2}
	\big| r_i (t_1 ,z) - r_i(t_2 ,z) \big| = \frac{ \big| (t_2-\tau) \mfc(t_2,z) -  (t_1-\tau) \mfc(t_1,z)  \big|  }{\Big| \big( \lambda_i(\tau ) - z - (t_1-\tau) \mfc(t_1,z) \big)  \big( \lambda_i(\tau ) - z - (t_2-\tau) \mfc(t_2,z) \big) \Big|}.
	\eeq
	
	\noindent The numerator of the right side of \eqref{rit1t2} is with overwhelming probability bounded by
	\begin{flalign}
	\label{t1t2mz}
	\begin{aligned}
	& \big| (t_1 - \tau) \mfc(t_1, z) -  (t_2 - \tau) \mfc(t_2, z) \big| \\ 
	& \quad \le   |t_1 - t_2| \big| \mfc(t_2,z) \big|  +  |t_1 - \tau| \eer \big| \mfc(t_1,z) - \mfc(t_2,z) \big| \le  C (\log N)^2 \eer (t_2 - t_1) +  C (t_1 - \tau) \eer N^{-c},
	\end{aligned} 
	\end{flalign} 
	
	\noindent Here, in the last inequality we used the overwhelming probability bound $\big | \mfc(t_2, z) \big|  \le C (\log N)^2$ (which follows from  \cite[(2.3)]{bourgade2017eigenvector}, whose hypotheses are satisfied by \eqref{mn0z} with $s=0$) and also the overwhelming probability estimate 
	\begin{flalign}
	\big| \mfc(t_1,z) - \mfc(t_2,z) \big| & \le \big| \mfc(t_1,z) - m_{\alpha}  (z) \big| + \big| \mfc(t_2,z) - m_{\alpha} (z) \big|  \le C N^{-c},
	\end{flalign}
	
	\noindent where the last inequality follows from \eqref{mszestimate} (with the $\delta$ there equal to $\frac{1}{2} \big( \frac{1}{2} - \mathfrak{a} \big)$ here).

	To bound the denominator of the right side of \eqref{rit1t2} observe, since ${\Im \mfc(t_1, z) \ge c'}$ for some uniform constant $c' > 0$ (which is a consequence of \eqref{mszestimate} and \eqref{e:mabounds}), we have with overwhelming probability that
	\beq
	\frac{ 1 }{ \big( \lambda_i(\tau ) - z - (t_1-\tau) \mfc(t_1,z) \big) \big( \lambda_i(\tau ) - z - (t_2-\tau) \mfc(t_2,z) \big) } \le C \frac{\big| r_i(t_1, z) \big| }{ t_2 - \tau}.
	\eeq
	
	\noindent Altogether, we obtain with overwhelming probability that
	\begin{multline}
	\big| r_i (t_1 ,z) - r_i(t_2 ,z) \big|  \le C \frac{\big| r_i(t_1, z) \big|}{ t_2 - \tau} \big( (\log N)^2 (t_2 - t_1) + ( t_1 \eer - \tau) N^{-c} \big)  \\ \le C \big| r_i(t_1, z) \big|   \bigg(  (\log N)^2 \frac{t_2 - t_1}{t_2 - \tau} + N^{-c} \bigg) \eer, 
	\end{multline}
	 where we used $t_1 - \tau \le t_2 - \tau$ in the last inequality. \eer It follows that
	\begin{flalign} 
	\label{uirisum}
	\begin{aligned}
	\Bigg| \sum_{i=1}^N \big \langle \u_i (\tau), \q \big\rangle^2 r_i(t,z) & - \sum_{i=1}^N \big\langle \u_i (\tau ), \q \big\rangle^2 r_i(\tau ,z) \Bigg| \\
	& \qquad  \le  C (\log N)^2  \bigg(  \frac{t_2 - t_1}{t_2 - \tau} + N^{-c} \bigg)\eer  \displaystyle\sum_{i = 1}^N \big\langle \textbf{u}_i (\tau), \textbf{q} \big\rangle^2 \big| r_i (t_1, z) \big|. 
	\end{aligned} 
	\end{flalign}
	
	To bound the right side of \eqref{uirisum}, observe that 
	\begin{flalign}
	\label{uqisum} 
	\displaystyle\sum_{i = 1}^N \big\langle \textbf{u}_i (\tau), \textbf{q} \big\rangle^2 \big| r_i (t_1, z) \big| = \displaystyle\sum_{\substack{1 \le i \le N \\ |\lambda_i (\tau)| < r}} \big\langle \textbf{u}_i (\tau), \textbf{q} \big\rangle^2 \big| r_i (t_1, z) \big| + \displaystyle\sum_{\substack{1 \le i \le N \\ |\lambda_i (\tau)| \ge r}} \big\langle \textbf{u}_i (\tau), \textbf{q} \big\rangle^2 \big| r_i (t_1, z) \big|. 
	\end{flalign}
	
	\noindent We bound the first term on the right side of \eqref{uqisum} using \Cref{l:utdeloc}, which yields 
	\begin{flalign}
	\label{lambdarestimate1} 
	 \sum_{\substack{1 \le i \le N \\ |\gamma_i (\tau)| < r}} \big\langle \u_i (\tau), \q \big\rangle^2 \big| r_i (t_1, z) \big| \le N^{\delta / 2 - 1} \displaystyle\sum_{i = 1}^N \big| r_i (t_1, z) \big| \le N^{\delta},
	\end{flalign}
	
	\noindent for any $\delta > 0$, with overwhelming probability. Here, in the last bound we used the fact that $\sum_{i = 1}^N \big| r_i (t_1, z) \big| \le N (\log N)^2$ (which is a consequence of \cite[Lemma 7.5]{landon2017convergence}, whose assumptions are verified by \eqref{rijestimate} and the fact that $t_1 - \tau \gg N^{-1 / 2}$). 
	
	To bound the second term in \eqref{uqisum}, observe for $\big| \gamma_i(\tau) \big| > r$ we have $\big| \lambda_i(\tau) \big| > \frac{3 r}{4}$ (by \eqref{e:xtrigidity}), so
	\beq
	\big | \lambda_i(\tau ) - z - (t_1-\tau) \mfc(t_1,z) \big|  \ge c,
	\eeq 
	for some constant $c >0$, where we used $ t_1-\tau \ll 1$ and $ |\mfc(t_1,z)| < C (\log N)^2$ (again by \cite[(2.3)]{bourgade2017eigenvector}). We then obtain by \eqref{rszi} that 
	\beq
	\label{lambdarestimate2}
	\sum_{\substack{ 1 \le i \le N \\ | \gamma_i(\tau) | \ge r }} \big\langle \u_i (\tau), \q \big\rangle^2 \big| r_i(t_1, z)  \big|  \le C \sum_{i = 1}^N \big\langle \u_i (\tau), \q \big\rangle^2 \le C.
	\eeq
	
	\noindent Now the first bound \eqref{rt1t2estimate} of the lemma follows from \eqref{e:t22}, \eqref{uirisum}, \eqref{uqisum}, \eqref{lambdarestimate1}, and \eqref{lambdarestimate2},  after absorbing the $(\log N)^2$ prefactor into $N^{\delta}$ and adjusting $\delta$ appropriately. \eer We omit the proof of the second as it is analogous, but obtained by replacing \eqref{rit1t2} with the bound
	\beq 
	\big| r_i (t_1, z_1) - r_i(t_1, z_2) \big|  \le \frac{ (t_1 - \tau) \big| \mfc(t_1, z_1) - \mfc(t_1, z_2)  \big| + |z_1 - z_2|}{\Big| \big( \lambda_i(\tau ) - z_1 - (t_1-\tau) \mfc(t_1,z_1) \big)  \big( \lambda_i(\tau ) - z_2 - (t_1-\tau) \mfc(t_1,z_2) \big) \Big|},
	\eeq
		
	\noindent and \eqref{t1t2mz} with the bound 
	\begin{flalign}
	 (t_1 - \tau) & \big| \mfc(t_1, z_1) - \mfc(t_1, z_2) \big| \\
	& \le (t_1 - \tau) \Big( \big| \mfc(t_1, z_1) - m_{\alpha} (z_1) \big| + \big| \mfc(t_1, z_2) - m_{\alpha} (z_2) \big| + \big| m_{\alpha} (z_1) - m_{\alpha} (z_2) \big|  \Big) \\
	& \le C (t_1 - \tau) \big( N^{-c} + |z_1 - z_2|  + \Im z_1 \eer \big), \label{t1z1z2estimate}
	\end{flalign} 
	
	\noindent where \eqref{t1z1z2estimate} follows from \eqref{mszestimate} and  \eqref{e:macontinuity}.\eer
	%the fact that $m_{\alpha} (z)$ is uniformly differentiable for $z \in \mathcal{D}_{\kappa}$. 
\end{proof}

\subsection{Short-time relaxation}

\label{s:relax} 

The proof of short-time relaxation here is similar to that of \cite[Theorem 3.6]{bourgade2017eigenvector}. However, certain changes are necessary, since the diagonal resolvent entries $R_{ii}(t,z)$ for the removed model $\X_t$ do not converge to a deterministic quantity, unlike those of the matrix model considered in \cite{bourgade2017eigenvector}. This causes the observable $f_{\mathbf{\bm \lambda}, t}(\bm \xi)$ from \eqref{e:observable2} to now converge to the random variable $A(\textbf{q}, \bm \xi)$, which is defined as follows. 

Recall $t$ and $t_0$ from \eqref{t} and \eqref{e:taudefine}, respectively; recall that $\{ \lambda_j \}$ are the eigenvalues of $\X_s$ and that $\{ \gamma_j (s) \big\}$ are given by \eqref{e:gammaidef}; fix a unit vector $\q \in \mathbb{R}^N$ with stable support; and set $q = |\supp \textbf{q}|$. For any integer $k \in [1, N]$; particle configuration $\bm \xi = (\xi_1, \xi_2, \ldots , \xi_N)$; and eigenvalue trajectory $\bm \lambda$, define  $A (\q, k ) = A_{t, \bm \lambda} (\textbf{q}, k)$ and $A (\textbf{q}, \bm \xi) = A_{t, \bm \lambda} (\q, \bm \xi)$ by \eer
\begin{equation}
\label{atqk}
A (\q, k ) =%\E\left[
 \frac{\Im \big\langle \q , \R ( s , \widehat \gamma_{k}   + \iu \eta ) \q \big\rangle }{  \Im m_\alpha (\gamma_k + \iu \eta)  },
 %\;\middle|\; \bm \lambda ,  \mathcal F_{t_0}\eer \right],
  \quad A (\textbf{q}, \bm \xi) = \prod_{k = 1}^N A (\q, k )^{\xi_k},
\end{equation}

\noindent where we have recalled $\gamma_k = \gamma_k^{(\alpha)}$ from \eqref{gammaalphai} and $\widehat{\gamma}_k$ from \eqref{hatgammaalphai}. 
	
	The initial data $\textbf{X}_0$ and Dyson Brownian motion $\X_s$ for $0\le s \le  t$ together induce a measure on the space of eigenvalues and eigenvectors $\big( \bm\lambda(s), \u(s) \big)_{0\le s \le t}$, which we denote by $\mathcal M$. Let $\bm \lambda = \big( \bm \lambda(s) \big)_{0 \le s \le t}$ be an eigenvalue trajectory with initial data given by a realization of the spectrum of $\X_0$, and recall the observable $ h_s (\bm \xi)\eer$ from \eqref{e:observable3} that is associated to an eigenvalue trajectory $\bm \lambda$ and ``starts'' at time $t_0$. 
	
	Before proceeding, we first fix a small constant $c_0 = c_0 (\alpha) > 0$ such that the conclusions of \Cref{l:rhocontinuity}, \Cref{l:thatlaw}, \Cref{l:goodset}, \Cref{l:shortrange}, and \Cref{l:tconcentration} apply to any $z= E + \iu \eta \in \mathcal{D}_{r}$ with $|E| \le 16 c_0$ and any $i \in [1, N]$ with $|\gamma_i(s)| < 16 c_0$. By \Cref{l:hgammarigid}, we may choose $c_1>0$ such that, for any fixed real number $s \in [t_0, t]$ and index $i \in \big[ (1/2- c_1) N, (1/2 + c_1) N \big]$, we have that $\big| \gamma_i(s) \big| < c_0$ with overwhelming probability. Hence, we will fix this choice of $c_1 >0$ in what follows and apply the five lemmas listed above without further comment.

	Then, to establish \Cref{t:dynamics}, it suffices prove the following proposition. 
	
	\begin{prop} 
		
	\label{ftatqxi} 
	
	For any integer $m \ge 0$, there exist constants $c_2 = c_2 (m) > 0$ and $C=C(m, q) >0$ such that for $\mathfrak{c} < c_2$ we have
	\begin{equation}
	\label{mxiestimate}
	\max_{\substack{\bm \xi \in \mathbb{N}^N \colon \mathcal N (\bm \xi) = m \\ \supp \bm \xi \in  [ (1/2 - c_1) N  , (1/2 + c_1) N  ]}} \big|  h_t(\bm \xi) \eer - A(\q, \bm \xi) \big| \le C N^{-c_2}, 
	\end{equation} 
	
	\noindent with overwhelming probability with respect to $\mathcal M$.
	
	\end{prop}
	
	\begin{proof}[Proof of \Cref{t:dynamics}]
		
		Recall from \eqref{e:observable} and \eqref{e:observable3} that $F_t (\bm \xi) = \mathbb{E} \big[ h_t(\bm \xi) \big]$, where the expectation is over $\bm \lambda$ and $\textbf{X}_{t_0}$. Therefore, the theorem follows from applying \eqref{mxiestimate} on an event of overwhelming probability, and applying the deterministic bounds $\big| h_t(\bm \xi) \big| \le N^m$ and $\big| A(\q, \bm \xi) \big| \le C^m \eta^{-m}$ (which holds by applying \eqref{kij} to bound the numerator of $A(\q, \bm \xi)$ by $\eta^{-m}$ and \eqref{e:mabounds} to bound its denominator by $c^m$) off of this event.
	\end{proof}

	We now introduce some notation. Fix a particle configuration $\bm \zeta = (\zeta_1, \zeta_2, \ldots , \zeta_N) \in \mathbb{N}^N$ with $m = \mathcal{N} (\bm \zeta)$ particles such that $ \supp \bm \zeta \in \big[ (1/2 - c_1) N, (1/2 + c_1) N \big]$. We must verify \eqref{mxiestimate} for $\bm \xi = \bm \zeta$. 
	
	For notational simplicity, we assume that $|\supp \bm \zeta| = 2$. The cases where $\bm \zeta$ is supported on one site and on more than two sites constitute straightforward modifications of the following two-site argument and will be briefly outlined in \Cref{OneInterval}. Denote $\supp \bm \zeta = \{ i_1, i_2 \}$, with $i_1 < i_2$, and let $j_1$ and $j_2$ denote the number of particles in $\bm \zeta$ at sites $i_1$ and $i_2$, respectively. Thus $\zeta_{i_1} = j_1$, $\zeta_{i_2} = j_2$, and $j_1 + j_2 = m$.

	 Recalling $\fd = \mathfrak{d} (m) = 50 ( 1 + m) \fc$ and $\ell = \ell (m) = \psi^{4m + 1} N^{1 + \mathfrak{d}} \eta$ from \eqref{e:ell}, we define a ``short-range averaging parameter''
	\beq
	\label{e:ddef}
	\widetilde d = \lfloor \ell \psi^{5 m} N^\fd \rfloor,
	\eeq
	
	\noindent which by \eqref{e:etadefine} and \eqref{e:taudefine} satisfies $\psi^2  \ell  \ll   \fa \ll N t_0$ (assuming $\mathfrak{c} = \mathfrak{c} (m) > 0$ is sufficiently small). For $a\in \mathbb{R}$ and $b\in \mathbb N$, we further define the interval $I^{(b)}_a = I^{(b)}_a(\bm \zeta)$ by 
\beq\label{e:twosites}
I^{(b)}_a =I^{(b)}_{a,1} \cup I^{(b)}_{a,2},
\eeq
where the intervals $I^{(b)}_{a, 1} = I^{(b)}_{a, 1} (\bm \zeta)$ and $I^{(b)}_{a, 2} = I^{(b)}_{a, 2} (\bm \zeta)$ are given by  

\beq
\label{iba1iba2}
I^{(b)}_{a,1} = \big[ i_1 - 10b\fa - a , i_1 + 10b\fa + a  \big], \qquad I^{(b)}_{a,2} =   \big[ i_2 - 10b\fa - a, i_2 +  10b\fa + a \big].
\eeq
\eer
We assume the intervals $I^{(b)}_{a,1}$ and $I^{(b)}_{a,2}$ are disjoint for all $a \in [0, 2 \tilde d]$ and $b \in [0, m]$.\footnote{The reason for this assumption will be seen below, in the material immediately following \eqref{gsxigsxi}.} When this is not true, the argument below is carried out analogously, but instead using a single connected interval. We describe the necessary modifications in \Cref{OneInterval}.

\begin{defn} 
	
	\label{intervalsomegab}
	
	\noindent For any particle configuration $\bm \xi \in \mathbb{N}^N$ with $\supp \bm \xi \subset  I_{\widetilde d + \psi \ell}^{(b)}$, we further set
\beq
\chi^{(b)}_1 (\bm \xi) = \sum_{i \in  I^{(b)}_{\widetilde d + \psi \ell, 1 } } \xi_i, 
\qquad  \chi^{(b)}_2 (\bm \xi) =  \sum_{i \in  I^{(b)}_{\widetilde d + \psi \ell, 2 } } \xi_i,
\eeq
	
\noindent which denote the number of particles in $\bm \xi$ in $I^{(b)}_{\widetilde d + \psi \ell, 1 }$ and $I^{(b)}_{\widetilde d + \psi \ell, 2}$, respectively. For any integers $k_1, k_2, b \ge 0$ and $n \ge 1$, recall $A (\textbf{q}, k)$ from \eqref{atqk} and set
\begin{flalign}\label{e:Atk1k2}
& \Omega^{(b)}(k_1, k_2) = \Big\{ \bm \xi \in \mathbb{N}^N: \supp \bm \xi \subset I^{(b)}_{\widetilde d + \psi \ell},   \chi^{(b)}_1 (\bm \xi)  = k_1,    \chi^{(b)}_2 (\bm \xi)  = k_2 \Big\}; \\
& \Omega^{(b)}(n) = \bigcup_{\substack{ k_1 + k_2 = n }} \Omega^{(b)} (k_1, k_2); \qquad A (k_1, k_2) = A ( \q , i_1)^{k_1} A ( \q , i_2)^{k_2}. 
\end{flalign}

\noindent  We also define the restricted intervals
\begin{flalign}
& \Phi^{(b)}(k_1, k_2) = \Bigg\{ \bm \xi \in \mathbb{N}^N: \supp \bm \xi \subset I^{(b)}_{ -  \psi \ell},   \sum_{i \in  I^{(b)}_{- \psi \ell, 1 } } \xi_i  = k_1,    \sum_{i \in  I^{(b)}_{- \psi \ell, 2 } } \xi_i  = k_2 \Bigg\}; \\
& \Phi^{(b)}(n) = \bigcup_{\substack{ k_1 + k_2 = n }} \Phi^{(b)} (k_1, k_2).
\end{flalign}

\end{defn}

The following two definitions provide certain operators on the space of functions on particle configurations and an auxiliary flow. Similar definitions appeared in \cite[Section 7.2]{bourgade2013eigenvector}.

\begin{defn}
	
	\label{k1k2operators}

Fix integers $k_1, k_2 \ge 0$. For integers $a,b  \ge 0$, we define operators $\flat^{(b)}_a = \operatorname{Flat}^{(b)}_{a; k_1, k_2}$ and $\av^{(b)} = \text{Av}^{(b)}_{k_1, k_2}$ on the space of functions $f\colon \mathbb{N}^N \rightarrow \mathbb{C}$ as follows. For each particle configuration $\bm \xi \in \N^N$ and function $f\colon \mathbb{N}^N \rightarrow \mathbb{C}$, set 
\begin{align} \label{d:flatdef}
 \big(\flat^{(b)}_a (f) \big)(\bm \xi) &=
\left\{
	\begin{array}{ll}
		f(\bm \xi),  & \mbox{if } \supp \bm \xi \subset I^{(b)}_a, \\
		A (k_1, k_2), & \text{otherwise};
	\end{array}
\right.\\
\av^{(b)}(f)&= \fa^{-1} \sum_{a = 1}^{\fa} \flat^{(b)}_a (f) \label{sumfa}.
\end{align}

\end{defn}

\begin{defn}
	
	\label{gsxi}

	Adopting the notation of \Cref{k1k2operators}, we define the flow $g_s (\bm \xi) = g_s^{(b)} (\bm \xi) = g^{(b)}_s (\bm \xi; k_1, k_2)$ for $s \ge t_0$ by 
	\beq \label{d:gdef}
	\partial_s g_s = \mathcal S(s) g_s,\quad \text{with initial data} \quad g_{t_0} (\bm \xi) = (\av^{(b)} h_{t_0} )(\bm \xi).
	\eeq
	
	\noindent For each $s \ge t_0$, let $\widetilde { \bm \xi} = \widetilde { \bm \xi}(s)  =  \widetilde { \bm \xi}(s; k_1, k_2) = \widetilde { \bm \xi}^{(b)} (s; k_1, k_2) \in \N^N$ denote a maximizing particle configuration for $g_s^{(b)}$:
	\beq\label{e:maximizer}
	g^{(b)}_s (\widetilde{\bm \xi}) = \max_{\bm \xi \in \Omega^{(b)} (k_1, k_2) }  g^{(b)}_s(\bm \xi; k_1, k_2) .
	\eeq
	When there are multiple maximizers, we pick one arbitrarily (in a way such that $\widetilde{\bm \xi} (s)$ remains piecewise constant in $s$).

\end{defn}

\subsection{Proof of \Cref{ftatqxi} } \label{s:dichoproof}

To establish \Cref{ftatqxi}, we begin with the following lemma providing bounds on $A (\textbf{q}, \bm \xi)$ (recall \eqref{atqk}).

\bel\label{l:aux3} 

For any integer $m \ge 0$, there exists a constant $C = C(m) > 0$ such that the following holds with overwhelming probability with respect to $\mathcal M$ for sufficiently small $\mathfrak{c} = \mathfrak{c} (m) > 0$. First, for any particle configuration $\bm \xi \in \mathbb{N}^N$ with $m = \mathcal N (\bm \xi)$ particles such that $\supp \bm \xi \subset \big[ (1/2 - c_1) N, (1/2 + c_1) N \big]$, we have that
\beq
\label{e:Abound}
 \big| A (\textbf{\emph{q}}, \bm \xi) \big| \le  C \psi^m.
\eeq
\noindent Second, for any integers $k_1, k_2, b \ge 0$ with $k_1 + k_2 \le m$ and $b \le m + 1$, we have that
\begin{flalign} 
\label{axiak1k2} 
 \displaystyle\max_{b \in [0, m + 1]}   \max_{ \bm \xi \in \Omega^{(b)}(k_1,k_2)} \big| A (\textbf{\emph{q}}, \bm \xi) - A (k_1, k_2) \big| < C N^{- 3 \fd}.
\end{flalign} 
\eel

\begin{proof}
	
	Recalling the definition \eqref{atqk} of $A$, the denominator  $\Im m_\alpha ( \gamma_{i_j} + \mathrm{i} \eta )$ \eer of each $A (\textbf{q}, i_j)$ is bounded below by a uniform constant  by \eqref{e:mabounds}.
	%with overwhelming probability by \eqref{rijestimate} (which applies since since $i_j \in \big[ (1/2 - c_1) N, (1/2 + c_1) N \big]$) and the fact that the same holds for $\Im m_{\alpha} \big( \gamma_{i_j} (t) + \mathrm{i} \eta \big)$. 
	Moreover, the numerator of each $A (\textbf{q}, i_j)$ is bounded above by $C \psi$ with overwhelming probability by the second part of \Cref{l:goodset}. Together these estimates yield \eqref{e:Abound}. 

	To establish \eqref{axiak1k2}, observe by \eqref{e:secondconclusion},  %[I don't understand the following parenthetical justification, since it seems that you only need $t \ge t_0$, which is true by definition]\eer (which applies since $t \ge c N^{(2 - \alpha) \nu}  \gg \eer N^{-\mathfrak{a}}$, by
	 \eqref{c1c2t0} and \eqref{e:parameters} that there exists a constant $c > 0$ such that, for any $\delta > 0$; $j \in \{ 1, 2 \}$; and $k \in I^{(b)}_{\widetilde{d} + \psi \ell,j}$, we have with overwhelming probability that
	\begin{flalign}
	\label{rqestimate1}
	\Big| \big\langle \textbf{q}, \textbf{R} ( t, \widehat \gamma_k + \mathrm{i} \eta) \textbf{q} \big\rangle - \big\langle \textbf{q}, \textbf{R} ( t, \widehat \gamma_{i_j}  + \mathrm{i} \eta) \textbf{q} \big\rangle \Big| \lesssim_{m, \delta} \displaystyle\frac{\psi^4}{\sqrt{N\eta}} + N^{\delta} \bigg( N^{-c} + \displaystyle\frac{\big| \widehat \gamma_k  - \widehat \gamma_{i_j}  \big|  + \eta \eer}{t - \tau} \bigg).
	\end{flalign}
	
	\noindent  To bound the right side of \eqref{rqestimate1}, we first note that by \eqref{e:newrigidity}, there exists a constant $c' > 0$ such that for indices $i,j$ with  $| i - N/2 | < c' N$ and $| j- N/2 | < c' N$,
	\beq\label{equalspace} | \widehat \gamma_i - \widehat \gamma_j | \le |\widehat \gamma_i - \gamma_i(t) | + |\widehat \gamma_j - \gamma_j(t) | + | \gamma_i(t) - \gamma_j(t) | \le 2\eta + CN^{-1} | i - j|.
	\eeq
	In the last inequality we used \eqref{e:newrigidity}, the fact that $N^{\delta - 1/2} \le \eta$ for $\delta \le \mathfrak{c}$, and the fact that the density $\dsc(t,x)$ satisfies $c \le \dsc(t,x) \le C$ for $|x| \le C^{-1}$. The latter fact is \cite[Lemma 3.2]{landon2017convergence}, whose hypotheses are satisfied in this case by the first inequality in \eqref{mn0z} and the first inequality in \eqref{e:mabounds}. We further observe using \eqref{equalspace} that for $k \in I^{(b)}_{\widetilde{d} + \psi \ell,j}$, 
	\begin{align}\label{e:imitate}
	\big| \widehat \gamma_{i_j} - \widehat \gamma_k  \big| & \dom   N^{-1} |i_j - k| + \eta 
	 \lesssim  2 ( 10 b\widetilde d + 3 \widetilde d) N^{-1} +  \eta
	 \le 30 b \widetilde{d} N^{-1} + \eta \le 3 1  b N^{3 \mathfrak{d}} \eta, 
	\end{align}
	
	\noindent where in the last inequality we used the fact that  $\widetilde{d} = \lfloor \ell \psi^{5m} N^{\mathfrak{d}} \rfloor$ (recall \eqref{e:ddef}), where $\ell = \psi^{4m + 1} N^{1 + \mathfrak{d}} \eta$ and $\mathfrak{d} = 50 (m + 1) \mathfrak{c}$ (recall \eqref{e:ell}). Further using the facts that $\eta \gg N^{-1 / 2}$ and $t - \tau = N^{7 \fd} \psi \eta$ (recall \eqref{e:taudefine}), it follows from \eqref{rqestimate1}, after taking $\mathfrak{c} = \mathfrak{c} (m) > 0$ sufficiently small, that with overwhelming probability 
	\begin{flalign}
	\label{rqestimate2} 
	\Big| \big\langle \textbf{q}, \textbf{R} ( t, \widehat \gamma_k  + \mathrm{i} \eta) \textbf{q} \big\rangle - \big\langle \textbf{q}, \textbf{R} ( t, \widehat \gamma_{i_j}  + \mathrm{i} \eta) \textbf{q} \big\rangle \Big| \lesssim_m N^{-4 \mathfrak{d}}.
	\end{flalign}

We now note that by \eqref{e:rhoacontinuity}, $\varrho_\alpha( x) > c$ for a constant $c>0$ and all $x$ in a neighborhood of zero that contains all $\gamma^{(\alpha)}_k$ such that $k \in I^{(b)}_{\widetilde{d} + \psi \ell,j}$. The definition \eqref{gammaalphai} then implies that for $j, k \in I^{(b)}_{\widetilde{d} + \psi \ell,j}$, $\big |\gamma^{(\alpha)}_j  - \gamma^{(\alpha)}_k | \le C N^{-1} |j - k|$. Using this fact along with \eqref{e:macontinuity} implies that for $k  \in I^{(b)}_{\widetilde{d} + \psi \ell,j}$,
\beq\label{e:denomconverge}
 \big| \Im m_\alpha ( \gamma_{i_j}  + \iu \eta ) -  \Im m_\alpha ( \gamma_{k }  + \iu \eta ) \big| \le C | i_j - k |/N + C \eta \le 3 1  b N^{3 \mathfrak{d}} \eta \ll 1,
\eeq
by the same calculation as in \eqref{e:imitate}.

	Thus, the bound \eqref{axiak1k2} follows from \eqref{rqestimate2}, \eqref{e:Abound}, and \eqref{e:denomconverge}.
\end{proof}

The following lemma, which we will establish in \Cref{s:proofofcrefaux1} below, essentially states that the difference $g_s (\widetilde{\bm \xi}) - A (k_1, k_2)$ is either nearly negative or its derivative is bounded by a negative multiple of itself. Here, we recall the intervals $\Phi^{(b)} (k_1, k_2)$ from \Cref{intervalsomegab}.

\bel

\label{l:aux1} 

Fix integers $b, n, k_1, k_2 \ge 0$ such that $n \le m$, $b \le m +1$, and $k_1 + k_2 = n$. If $\mathfrak{c} = \mathfrak{c} (m) > 0$ (recall \eqref{e:etadefine}) is chosen small enough, then there exist constants $C=C(b,n) > 0$ and $c=c(b,n) >0$ such following holds with overwhelming probability with respect to $\mathcal M$. Fix a realization of $\X_0$ and an associated eigenvalue trajectory $\bm\lambda = \big( \bm \lambda(s) \big)_{t_0 \le s \le t}$. There exists a countable subset $\mathcal{C} = \mathcal C ( \X_0, \bm \lambda) \subset [t_0, t]$ such that, for $s \in [t_0, t] \setminus \mathcal{C}$, the continuous function $g^{(b)}_s (\widetilde{\bm \xi} \big(s; k_1, k_2) ; k_1, k_2\big)$ is differentiable and satisfies either
\beq
\label{e:nothingtoprove1} 
g^{(b)}_s (\widetilde{\bm \xi})  - A (k_1, k_2) \le N^{-1},
\eeq or
\begin{flalign}
\begin{aligned}
\label{e:ultimate}
\partial_s \big (g^{(b)}_s (\widetilde{\bm \xi})  - A (k_1, k_2) \big ) & \le \frac{C}{\eta} \bigg( \psi \displaystyle\max_{\bm \xi}  \Big| h_s(\bm \xi) - A \big( \chi^{(b+1)}_1 (\bm \xi), \chi^{(b+1)}_2 (\bm \xi) \big) \Big| + N^{-\fd} \bigg) \\
& \qquad  - \frac{c}{\eta} \big(g^{(b)}_s(\widetilde {\bm \xi}) - A (k_1, k_2) \big),
\end{aligned}
\end{flalign}

\noindent where the maximum in \eqref{e:ultimate} is taken over all $\bm \xi \in \Phi^{(b+1)}(k_1 -1 , k_2) \cup \Phi^{(b+1)}(k_1, k_2 - 1)$.

 \eel
 	
 Given \Cref{l:aux1}, we can now establish \Cref{ftatqxi}. 
 
 \begin{proof}[Proof of \Cref{ftatqxi}.]

 	First observe that, for any $n\le m$ and $\bm \xi \in  \Phi^{(m-n+1)}(n)$, we have  
 	\begin{align}
 	\big| h_t(\bm \xi) - g_t^{(m-n+1)}(\bm \xi) \big| &\le \Big| \big(\mathcal  U_{\mathcal B}(t_0, t) h_{t_0}  - \mathcal U_{\mathcal S}(t_0,t) h_{t_0} \big) (\bm \xi) \Big| + \big| \mathcal U_{\mathcal S}(t_0, t)(h_{t_0} - \av^{(m-n+1)} h_{t_0} ) (\bm \xi) \big|
 	\\ &\dom_m \psi^{4m} N \ell^{-1} (t - t_0) + e^{-c \psi} \label{fgestimate}.
 	\end{align}
 	
 	\noindent Here, the first term from \eqref{fgestimate} follows from \Cref{l:shortrange}, where the containment $\supp \bm \xi \subset \big[ (1/2 - c) N, (1/2 + c) N \big]$ holds since $i_1, i_2 \in \big[ (1/2 - c_1) N, (1/2 + c_1) N \big]$. The second term in \eqref{fgestimate} follows from \eqref{e:finitespeed}, which applies due to the facts that $\big( h_{t_0} - \av^{(m-n+1)} h_{t_0} \big) (\bm \xi') = 0$ whenever $\supp \bm \xi' \subseteq I^{(m-n+1)}_0$ and that $\widetilde{d}_n (\bm \xi, \bm \xi') >  \psi \ell$ for any $\bm \xi \in \Phi^{(m-n+1)} (n)$ and $\supp \bm \xi' \not\subset I^{(m-n+1)}_0$.
 \begin{comment}
 	Since $\bm \zeta \in \Phi^{(1)} (m)$ and since our choice of parameters ensures that \eqref{fgestimate} is bounded by $N^{-c_2}$ (if $\mathfrak{c} = \mathfrak{c} (m) > 0$ and $c_2 = c_2 (m) > 0$ are chosen sufficiently small), it suffices to show with overwhelming probability that 
 	\begin{flalign}
 	\label{gtzetaestimate}
 	\big| g^{(1)}_t (\bm \zeta) - A (\textbf{q}, \bm \zeta) \big| < C N^{-c_2}.
 	\end{flalign} 
\end{comment}

 	We now define a discretization of the interval $[t_0, t]$ by 
 	\beq 
 	t_k = t_0 + k m^{-1} \psi \eta, \quad \text{for}\quad 0 \le k \le m,
 	\eeq 
 	
 	\noindent and we will show for each integer $n \in [0, m]$ that, with overwhelming probability,
\begin{equation}
\label{e:inductionstep}
\displaystyle\sup_{s \in [t_n, t]} \max_{\bm \xi \in   \Phi^{(m - n + 1)}(n) }  \Big| h_s(\bm \xi) - A \big(\chi^{(m - n + 1)}_1 (\bm \xi), \chi^{(m - n + 1)}_2 (\bm \xi) \big) \Big|  \dom_n  \frac{\psi^n }{ N^\fd}.
\end{equation}

\noindent  Given the $n = m$ case of \eqref{e:inductionstep} and using the facts that $\bm \zeta \in \Phi^{(1)} (m)$ and $t_m = t$, we obtain with overwhelming probability that
\begin{equation}
 \Big| h_t (\bm \zeta) - A \big(\chi^{(1)}_1 (\bm \zeta), \chi^{(1)}_2 (\bm \zeta) \big) \Big|  \dom_m  \frac{\psi^m }{ N^\fd}.
\end{equation}

\noindent Our conclusion, \eqref{mxiestimate}, then follows from the facts that $\chi^{(1)}_1 (\bm \zeta) = j_1 = \zeta_{i_1}$ and $\chi^{(1)}_2 (\bm \zeta) = j_2 = \zeta_{i_2}$; and our choices $\psi = N^{\fc}$ and $\fd = 50 (m + 1) \fc$.

So, it suffices to prove \eqref{e:inductionstep} and therefore the two estimates 
\begin{flalign}
\label{fsxiestimate}
& \displaystyle\sup_{s \in [t_n, t]} \max_{\bm \xi \in   \Phi^{(m - n + 1)}(n) }  \Big( h_s(\bm \xi) - A \big(\chi^{(m - n + 1)}_1 (\bm \xi), \chi^{(m - n + 1)}_2 (\bm \xi) \big) \Big)  \dom_n  \frac{\psi^n }{ N^\fd}; \\
& \displaystyle\sup_{s \in [t_n, t]} \max_{\bm \xi \in   \Phi^{(m - n + 1)}(n) }  \Big( A \big(\chi^{(m - n + 1)}_1 (\bm \xi), \chi^{(m - n + 1)}_2 (\bm \xi) \big) - h_s(\bm \xi) \Big)  \dom_n  \frac{\psi^n }{ N^\fd}. \label{fsxiestimate2}
\end{flalign}

\noindent To do this, we induct on $n \in [0, m]$. The base case $n = 0$ is trivial, since $\bm \xi \in \Phi^{(m+1)} (0)$ implies that $h_s (\bm \xi) = 1 = A \big(\chi^{(m+1)}_1 (\bm \xi), \chi^{(m+1)}_2 (\bm \xi) \big)$. For the induction step, we assume the induction hypothesis \eqref{e:inductionstep} holds for $n-1$ and prove \eqref{fsxiestimate} and \eqref{fsxiestimate2} for $n$. 

We will in fact only establish \eqref{fsxiestimate}, as the proof of \eqref{fsxiestimate2} is entirely analogous (by in what follows replacing the maximizer $\widetilde{\bm \xi}$ of $g^{(b)}$ with the minimizer). To that end, for any two fixed integers $k_1, k_2 \ge 0$ with $k_1 + k_2= n$, it suffices to show that
\beq
\label{gsestimate} 
\displaystyle\sup_{s \in [t_{n }, t]} \big( g^{(m - n + 1)}_s (\widetilde{\bm \xi}; k_1, k_2)  -  A (k_1, k_2)  \big)  \dom_n \frac{\psi^n}{N^\fd}
\eeq
 
 \noindent holds with overwhelming probability, where we have abbreviated $\widetilde{\bm \xi} = \widetilde{\bm \xi}^{(m - n + 1)} (s; k_1, k_2)$. Indeed, given \eqref{gsestimate}, \eqref{fsxiestimate} follows upon letting $(k_1, k_2)$ range over all pairs of integers summing to $n$; the fact that $\widetilde{\bm \xi}$ maximizes $g^{(m - n + 1)}$ over $\Omega^{(m -n + 1)} (k_1, k_2)$; the fact that $\Phi^{(m - n + 1)} (k_1, k_2) \subseteq \Omega^{(m - n + 1)} (k_1, k_2)$; and \eqref{fgestimate}.

To establish \eqref{gsestimate}, we first apply the $b = m - n + 1$ case of \Cref{l:aux1}. Since \eqref{e:nothingtoprove1} implies \eqref{gsestimate}, we may assume that \eqref{e:ultimate} holds. Then the induction hypothesis \eqref{e:inductionstep} (whose $n$ is equal to $n - 1$ here); the fact that $\bm \xi \in \Phi^{(m - n + 2)} (k_1 - 1, k_2) \cup \Phi^{(m - n + 2)} (k_1, k_2 - 1)$ implies $\bm \xi \in  \Omega^{(m - n + 1)} (n - 1)$ (since $10 \widetilde{d} > 2 \psi \ell$); and \eqref{e:ultimate} together yield that the bound 
\begin{equation}
\partial_s \big( g^{(m - n + 1)}_s(\widetilde{\bm \xi}; k_1, k_2) - A (k_1, k_2 ) \big) \le  C( m,n)\frac{\psi^{n} }{N^\fd \eta} - \frac{c(m,n)}{\eta} \big( g^{(m - n + 1)}_s(\widetilde{\bm \xi}; k_1, k_2) -  A (k_1, k_2 ) \big), 
\end{equation}

\noindent holds for all $s \in [t_{n - 1}, t] \setminus \mathcal{C}$ (for some countable subset $\mathcal{C}$) with overwhelming probability. 

In particular, if we define $F \colon[t_0 ,t ] \rightarrow \mathbb R$ by
\beq
\label{fs}
F(s) = F_{m, n; k_1, k_2} (s) = g_s^{(m - n + 1)} (\widetilde{\bm \xi}; k_1, k_2)  -  A (k_1, k_2),
\eeq

\noindent then there exist constants $c = c (m,n) > 0$ and $C = C(m,n) > 0$ such that
\beq\label{e:diffineq}
\partial_s \left( F(s) - C \frac{\psi^n}{N^\fd}  \right) \le - \frac{c}{\eta} \left( F(s) - C \frac{\psi^n}{N^\fd} \right), \qquad \text{for each $s \in [t_{n - 1}, t] \setminus \mathcal{C}$}. 
\eeq

\noindent Thus, integration and the fact that  $t_n - t_{n-1} = \frac{\psi \eta}{m}$ together yield for $s \in [t_n, t]$ that
\begin{flalign}
\label{fsestimate}
F(s) & \le \exp \bigg( - \displaystyle\frac{c}{\eta} (s - t_{n - 1}) \bigg) \left( F(t_{n-1}) - C \frac{\psi^n}{N^\fd} \right) + C \frac{\psi^n}{N^\fd} \le \exp \Big( - \displaystyle\frac{c \psi}{m} \Big)  F(t_{n-1})   + C \frac{\psi^n}{N^\fd}.
\end{flalign}  

\noindent To bound $\big| F (t_{n - 1}) \big|$, observe that 
\begin{flalign}
\label{ftn1}
\big| F(t_{n - 1}) \big| \le \big\| g_{t_{n - 1}}^{(m - n + 1)} \big\|_{\infty} + \big| A (k_1, k_2)\big| \le \big\| g_{t_0}^{(m - n + 1)} \big\|_\infty + C(m) \psi^m \le  C^m N^{m / 2} + C(m) \psi^m.
%\big| F(t_{n - 1}) \big| \le \big\| g_{t_{n - 1}}^{(m - n + 1)} \big\|_{\infty} + \big| A (k_1, k_2)\big| \le \big\| g_{t_0}^{(m - n + 1)} \big\| + C^m \eta^{-m} \le C^m \eta^{-m} + C^m N^{m / 2},
\end{flalign}

\noindent Here, to deduce the first inequality, we used the definition \eqref{fs} of $F$. To deduce the second, we used the fact that $\big\| g_s^{(b)} \big\|_{\infty} \le \big\| g_{s'}^{(b)} \big\|_{\infty}$ whenever $s' \le s$ (since $\mathcal{S}$ is the generator of a Markov process) and \eqref{e:Abound}. To deduce the third inequality, we used \eqref{d:gdef} and \eqref{e:observable2}.

Then \eqref{fs}, \eqref{fsestimate}, and \eqref{ftn1} together imply \eqref{gsestimate}, from which we deduce the proposition. 
\end{proof}

\subsection{Proof of \Cref{l:aux1}}

\label{s:proofofcrefaux1}

We first establish \Cref{l:aux1} assuming \eqref{e:thethirdtermbounded} below; the latter will be proven as \Cref{l:aux2}. Throughout this section, for $s \in [t_0, t]$ we occasionally abbreviate $\{ \lambda_j \}_{1 \le j \le N} = \big\{ \lambda_j(s) \big\}_{1 \le j \le N}$.

\begin{proof}[Proof of \Cref{l:aux1}]

The differentiability of $g_s(\widetilde{\bm \xi})= g^{(b)}_s (\widetilde{\bm \xi}; k_1,k_2)$ follows from the general fact that the maximum of finitely many differentiable functions on an interval $I$ is itself differentiable, away from a countable set $\mathcal{C}$. Thus, for any fixed $s \in [t_0, t] \setminus \mathcal{C}$, it remains to upper bound $g_s (\widetilde{\bm \xi})  - A(k_1, k_2)$ and its derivative. To that end, we may assume that  
\beq\label{e:nothingtoprove} 
g_s (\widetilde{\bm \xi})  - A(k_1, k_2) > N^{-1},
\eeq

\noindent for otherwise \eqref{e:nothingtoprove1} would hold. In this case, we set $\widetilde{\bm \xi} = (\widetilde{\xi}_1, \widetilde{\xi}_2, \ldots , \widetilde{\xi}_N)$ and use \eqref{sfs} to write
\beq\label{e:max1}
\partial_s \big (g_s (\widetilde{\bm \xi})  - A(k_1, k_2) \big ) = \mathcal S(s) g_s(\widetilde{\bm \xi})  =\sum_{ 0 < |j - k| \le \ell} c_{jk} (t) 2 \widetilde \xi_j ( 1 + 2 \widetilde \xi_k) \big( g_s(\widetilde{\bm \xi}^{jk}) - g_s (\widetilde{\bm \xi}) \big).
\eeq

Now let $\supp \widetilde{\bm \xi} = \{ j_1, j_2, \ldots , j_h \}$. We claim that 
\beq
\label{gsxigsxi}
g_s(\widetilde{\bm \xi}^{j_p k}) \le  g_s(\widetilde {\bm \xi}), \quad \text{for any integers $p \in [1, h]$ and $k \in [j_p - \ell, j_p + \ell]$}.
\eeq

\noindent To see this first observe that, since $\widetilde{\bm \xi}$ maximizes $g_s$ over $ \Omega(k_1, k_2) = \Omega^{(b)}(k_1, k_2)$, \eqref{gsxigsxi} holds if $\widetilde{\bm \xi}^{j_p k} \in \Omega(k_1, k_2)$. So, let us assume instead that $\widetilde{\bm \xi}^{j_p k} \notin \Omega(k_1, k_2)$, meaning that there exists some $v \in \{ 1,  2 \}$ such a particle originally at site $j_p \in I_{\widetilde{d} + \psi \ell, v}^{(b)}$ in $\widetilde{\bm \xi}$ jumped out of the interval $I_{\widetilde{d} + \psi \ell, v}^{(b)}$.  This implies that $k \notin I^{(b)}_{\widetilde d + \psi \ell}$, since the particle can jump at most $\ell$ sites by \eqref{gsxigsxi}, but the disjoint intervals $I_{\widetilde{d} + \psi \ell, 1}^{(b)}$ and  $I_{\widetilde{d} + \psi \ell, 2}^{(b)}$ are at least $\tilde d \gg \ell$ sites apart by hypothesis.

\begin{comment}
Let us assume for notational convenience that $v = 1$, in which case the particle originally at site $j_p$ jumped to some site $k \in [j_p - \ell, j_p + \ell]$ such that either $k \in I_{\widetilde{d} + \psi \ell, 2}^{(b)}$ or $k \notin I^{(b)}_{\widetilde d + \psi \ell}$. In both cases, we will show that \eqref{e:nothingtoprove} does not hold; this would establish $\widetilde{\bm \xi}^{j_p k} \in \Omega(k_1, k_2)$ and therefore verify \eqref{gsxigsxi}, by the above. 
\end{comment}
%Let us first assume that $k \notin I^{(b)}_{\widetilde d + \psi \ell}$. 
Then, \eqref{d:flatdef} and \eqref{sumfa} together yield $( \av h_{t_0} )  (\bm \xi) - A (k_1, k_2) = 0$ unless $\supp \bm \xi \subseteq I_{\widetilde{d}}^{(b)}$. Thus, any particle configuration $\bm \xi \in \mathbb{N}^N$ in the support of $\av h_{t_0} -  A(k_1, k_2)$ must satisfy $\widetilde{d}_n \big( \bm \xi, \widetilde{\bm \xi}^{j_p k} \big) \ge  \psi \ell$. Hence, the finite speed of propagation estimate \eqref{e:finitespeed} yields 
\beq 
\big| g_s(\widetilde{\bm \xi}^{j_p k}) - A(k_1, k_2) \big| = \Big| \mathcal U_{\mathcal S} (t_0, s)  \big( ( \av h_{t_0} )  -  A(k_1, k_2) \big) (\widetilde{\bm \xi}^{j_p k}) \Big| \dom_n \exp \Big( - \displaystyle\frac{c \psi}{2} \Big) < \displaystyle\frac{1}{N},
\eeq
 
\noindent which contradicts \eqref{e:nothingtoprove}. 

We now set $z_{j_p} = \lambda_{j_p} + \iu \eta$ and use \eqref{gsxigsxi}, the definition \eqref{cij} of the $c_{ij}$, and the fact that $\widetilde{\xi}_j (2 \widetilde{\xi}_k + 1) \ge 1$ when $\widetilde{\bm \xi}^{j_p k} \ne \widetilde{\bm \xi}$ to bound the sum in \eqref{e:max1} over $j$ by
\begin{flalign} 
\eqref{e:max1} \le \displaystyle\sum_{p = 1}^h \sum_{0 < | k - j_p| \le \ell}  \frac{g_s(\widetilde{\bm \xi}^{j_p k}) - g_s(\widetilde {\bm \xi})}{ N(\lambda_{j_p} - \lambda_k)^2 }   \le &  \frac{1}{N} \displaystyle\sum_{p = 1}^h \sum_{0 < | k - j_p| \le \ell} \frac{g_s(\widetilde{\bm \xi}^{j_p k}) - g_s(\widetilde {\bm \xi})}{(\lambda_{j_p} - \lambda_k)^2 + \eta^2}  \\
 =&  \frac{1}{N\eta} \displaystyle\sum_{p = 1}^h \sum_{0 < | k - j_p| \le \ell}  \bigg( \Im \frac{g_s(\widetilde {\bm \xi }^{j_p k}) }{z_{j_p} - \lambda_k} - \Im \frac{A(k_1, k_2) }{z_{j_p} - \lambda_k} \bigg) \label{sumkgsxi} \\
 &  - \frac{1}{N\eta} \big( g_s (\widetilde{\bm \xi}) - A(k_1, k_2) \big) \displaystyle\sum_{p = 1}^h \sum_{0 < | k - j_p| \le \ell}  \Im \frac{1}{z_{j_p} - \lambda_k} \label{e:359}.
\end{flalign}

\noindent Since the first bound in \eqref{e:xtdispersion} and the fact that $N^{-1} \ell \gg \eta$ yields
\beq
\displaystyle\sum_{p = 1}^h \sum_{0 < | k - j_p| \le \ell}\Im \frac{1}{z_{j_p} - \lambda_k} = \displaystyle\sum_{p = 1}^h \sum_{0 < | k - j_p| \le \ell} \frac{\eta}{(\lambda_{j_p} - \lambda_k)^2 + \eta^2} \ge \sum_{p=1}^h \sum_{k:|\lambda_k - \lambda_{j_p} | \le \eta} \frac{\eta}{2\eta^2} \ge c N,\eeq

\noindent we have that 
\beq\label{e:bound1}
\eqref{e:359} \le - \displaystyle\frac{c}{\eta} \big(g_s (\widetilde{\bm \xi}) - A(k_1, k_2) \big).
\eeq 

\noindent To bound \eqref{sumkgsxi}, we fix $p \in [1, h]$, recall $\mathcal{B}$ from \eqref{e:evmf}, and employ the decomposition 
\begin{flalign}
\label{e:theterm2}
\frac{1}{N} & \sum_{0  < |k - j_p| \le \ell}  \bigg( \Im \frac{g_s(\widetilde {\bm \xi }^{j_p k}) }{z_{j_p} - \lambda_k} - \Im \frac{A(k_1, k_2) }{z_{j_p} - \lambda_k} \bigg) \\
&\label{e:thefirstterm} \quad = \frac{1}{N} \Im \sum_{0  < |k - j_p| \le \ell}  \frac{ \big( {\mathcal U}_{\mathcal S} (t_0, s) \av^{(b)} h_{t_0} \big) ( \widetilde {\bm \xi}^{j_p k} )  -  \big(\av^{(b)} {\mathcal U}_{\mathcal S} (t_0, s) h_{t_0} \big) ( \widetilde{\bm \xi}^{j_p k} ) }{z_{j_p} - \lambda_k}  \\
&\label{e:thesecondterm} \qquad + \frac{1}{N} \Im \sum_{0  < |k - j_p| \le \ell} \frac{ \big( \av^{(b)} {\mathcal U}_{\mathcal S} (t_0, s) h_{t_0} \big) ( \widetilde{\bm \xi}^{j_p k} )  - \big(  \av^{(b)} \mathcal U_{\mathcal B} (t_0 , s) h_{t_0} \big) ( \widetilde{\bm \xi}^{j_p k} \big)  }{z_{j_p} - \lambda_k}  \\
&\label{e:thethirdterm} \qquad + \frac{1}{N} \Im \sum_{0  < |k - j_p| \le \ell} \bigg(  \frac{\big( \av^{(b)} \mathcal U_{\mathcal B} (t_0 , s) h_{t_0} \big) ( \widetilde{\bm \xi}^{j_p k} )  }{z_{j_p} - \lambda_k}  - \Im \frac{A(k_1, k_2) }{z_{j_p} - \lambda_k} \bigg).
 \end{flalign}
The terms \eqref{e:thefirstterm} and \eqref{e:thesecondterm} may be bounded as in the content following \cite[(3.64)]{bourgade2017eigenvector}:\footnote{When bounding \eqref{e:thefirstterm}, $I^{(b)}_a$ plays the role of the interval $[b_1 - a, b_2 +a]$ in \cite{bourgade2017eigenvector}.}% For \eqref{e:thesecondterm}, $I_{\widetilde d}$ plays the role of $[b_ 1 - d - \psi\ell, b_2 + d + \psi \ell]$.} We obtain
\beq\label{e:bound2}
\eqref{e:thefirstterm} \dom_n \frac{\psi^{n+1} \ell}{\widetilde d}  , \quad \eqref{e:thesecondterm} \dom_n \frac{\psi^{n} N (t - t_0)}{\ell}.
\eeq
For brevity we only prove here the second inequality in \eqref{e:bound2} and refer the reader to \cite{bourgade2017eigenvector} for details on the first. Using \Cref{l:shortrange} (which applies as $\supp \widetilde{\bm \xi}^{j_p k} \subseteq \big[ (1/2 - c) N, (1/2 + c) N \big]$, since $\supp \widetilde{\bm \xi} \subseteq \big[ (1/2 - c_1) N, (1/2 + c_1) N \big]$ and $|k - j_p| \le \ell$), we find
\begin{multline}
\left|  \big( \av^{(b)} {\mathcal U}_{\mathcal S} (t_0, s) h_{t_0} \big) ( \widetilde{\bm \xi}^{j_p k} )  - \big(  \av^{(b)} \mathcal U_{\mathcal B} (t_0 , s) h_{t_0} \big) ( \widetilde{\bm \xi}^{j_p k} \big)  \right| \\ 
 \le \Big|  \big( \mathcal U_{\mathcal S}(t_0,s)h_{t_0} - \mathcal U_{\mathcal B}(t_0, s)h_{t_0} \big) \big( \widetilde{\bm \xi}^{j_p k} \big)  \Big| \dom_n  \frac{\psi^{n} N (t -t_0)}{ \ell},
\end{multline}
which implies the second bound in \eqref{e:bound2}.

Next, as \Cref{l:aux2} below, we show that, for any fixed $p \in [1, h]$ and $s \in [t_0, t] \setminus \mathcal{C}$,
\beq\label{e:thethirdtermbounded}
\eqref{e:thethirdterm}  \dom_n  \psi \Big| h_s(\widetilde {\bm \xi} \setminus j_p )  - A \big( \chi^{(b+1)}_1 ( \widetilde {\bm \xi} \setminus j_p), \chi^{(b+1)}_2 ( \widetilde {\bm \xi} \setminus j_p) \big)  \Big| + N^{-\fd}.
\eeq
Combining these bounds and using the choices of $t_0$ from \eqref{e:taudefine}; $\ell$ from \eqref{e:ell}; and $\widetilde{d}$ from \eqref{e:ddef}, we obtain for fixed $p \in [1, h]$ and $s \in [t_0, t] \setminus \mathcal{C}$ that \eqref{e:theterm2} satisfies
\begin{flalign}
\frac{1}{N} & \sum_{ 0  < | k - j_p | \le \ell } \bigg( \Im \frac{g_s(\widetilde {\bm \xi }^{j_p k}) }{z_{j_p} - \lambda_k} - \Im \frac{A(k_1, k_2) }{z_{j_p} - \lambda_k} \bigg) \\
& \qquad \dom_n \max_{\bm \xi \in \Omega^{(b+1)}(k_1 -1 , k_2) \cup \Omega^{(b+1)}(k_1, k_2 - 1)}   \psi \Big| h_s(\bm \xi) - A \big( \chi^{(b+1)}_1 (\bm \xi), \chi^{(b+1)}_2 (\bm \xi) \big) \Big|  + N^{-\fd}.
\end{flalign}

\noindent Summing over $p \in [1, h]$, inserting this into \eqref{sumkgsxi}, and using \eqref{e:bound1} then completes the proof.
\end{proof}

We conclude this section by establishing \eqref{e:thethirdtermbounded}.

\bel\label{l:aux2}
 Retain the hypotheses and notation of the proof of  \Cref{t:dynamics}. Then equation \eqref{e:thethirdtermbounded} holds.
\eel
\begin{proof}
From complete delocalization \eqref{e:deloc}, we have with overwhelming probability that 
\beq
\label{fxiestimate}
 \displaystyle\max_{p \in [1, h]} h_s(\widetilde{\bm \xi}^{j_p k} ) \dom_n \psi^n.
\eeq

\noindent Now, for any particle configuration $\bm \xi \in \mathbb N^N$, define $a_{\bm \xi} = a_{\bm \xi; s} \in [0,1]$  through the equation
\beq 
\av(f)(\bm \xi) =  a_{\bm \xi} h_s (\bm \xi) + (1 - a_{\bm \xi} ) A (k_1, k_2) \qquad \text{if $h_s (\bm \xi) \ne A( k_1, k_2 )$},
\eeq

\noindent  and set $a_{\bm \xi} = 0$ if $h_s (\bm \xi)  = A( k_1, k_2 )$. Since $\mathcal U_{\mathcal{B}} (t_0, s) h_{t_0} = h_s$, the first term of \eqref{e:thethirdterm} equals  
\begin{flalign}
\begin{aligned} 
\label{e:91}
 \frac{1}{N} \Im \sum_{0 < |k - j_p| \le \ell} & \frac{a_{\widetilde {\bm \xi}^{j_p k}} h_s (\widetilde{\bm \xi}^{j_p k}) + (1 - a_{\widetilde {\bm \xi}^{j_p k}} ) A( k_1, k_2 )}{z_{j_p} - \lambda_k} \\
 & = \frac{1}{N}\Im \sum_{0 < |k - j_p| \le \ell} \frac{a_{\widetilde {\bm \xi}} h_s (\widetilde{\bm \xi}^{j_p k})  + (1 - a_{\widetilde {\bm \xi}} ) A( k_1, k_2 )}{z_{j_p} - \lambda_k}  \\ 
 & \qquad + \frac{1}{N}\Im \sum_{0 < |k - j_p| \le \ell} \displaystyle\frac{ (a_{\widetilde {\bm \xi}^{j_p k}}- a_{\widetilde {\bm \xi}}) \big( h_s (\widetilde{\bm \xi}^{j_p k}) - A(k_1, k_2) \big)}{z_{j_p} - \lambda_k} \\ 
 & = \frac{1}{N}\Im \sum_{0 < |k - j_p| \le \ell} \frac{a_{\widetilde {\bm \xi}} h_s (\widetilde{\bm \xi}^{j_p k})  + (1 - a_{\widetilde {\bm \xi}} ) A( k_1, k_2 )  }{z_{j_p} - \lambda_k} + O_n \left( \frac{ \ell \psi^{n } }{\widetilde d} \right).
 \end{aligned}
\end{flalign}

\noindent In the last line we used \eqref{e:Abound}; \eqref{fxiestimate}; the fact that $\Im \sum_{0 < |k - j_p| \le \ell} (z_{j_p} - \lambda_k)^{-1} \le \Im m_N(s,z_{j_p}) \le C$, which follows from the first part of \Cref{l:goodset}, \eqref{mszestimate}, and \eqref{e:mabounds}; and the bound 
\beq
\label{e:aetadiff}
\big| a_{\widetilde {\bm \xi}} - a_{\widetilde {\bm \xi}^{j_p k}} \big| \le \displaystyle\frac{d( \widetilde {\bm \xi}, \widetilde {\bm \xi}^{j_p k} )}{\widetilde{d}} \le \displaystyle\frac{\ell}{\widetilde d},
\eeq

\noindent which follows from the definition of $a_{\widetilde {\bm \xi}}$ and the definition of the $\av$ operator, since the sums defining $\av(f)(\widetilde {\bm \xi})$ and $\av(f)(\widetilde {\bm \xi}^{j_p k})$ can differ in at most $\ell$ terms. The equation \eqref{e:91} implies
\beq\label{e:intermediatestep}
\eqref{e:thethirdterm} = \frac{a_{\widetilde {\bm \xi}} }{N} \sum_{0 < |k - j_p| < \ell} \left( \frac{\eta h_s(\widetilde{\bm \xi}^{j_p k} ) }{(\lambda_{j_p} - \lambda_k)^2 + \eta^2} -  \frac{\eta A( k_1, k_2 ) }{(\lambda_{j_p} - \lambda_k)^2 + \eta^2} \right) + O_n \left( \frac{ \ell \psi^{n} }{\widetilde d} \right).
\eeq

\noindent Through \eqref{e:xtdispersion} and a dyadic decomposition analogous to the one used in the proof of \Cref{l:dyadic} (see also the proof of \cite[Lemma 3.5]{bourgade2017eigenvector} for more details), one has with overwhelming probability that
\beq
\label{sumlambdajp}
 \quad \frac{1}{N} \sum_{|k - j_p| > \ell} \frac{\eta}{(\lambda_{j_p} - \lambda_k)^2 +\eta^2} \le \frac{ C N \eta}{\ell},
\eeq

\noindent which by \eqref{fxiestimate} implies with overwhelming probability
\beq\label{e:intermediatestep2}
\left| \frac{1}{N}\sum_{|k - j_p| > \ell} \frac{\eta h_s (\widetilde { \bm \xi }^{j_p k})}{(\lambda_{j_p}  - \lambda_k)^2 + \eta^2} \right| \le \frac{C N \eta \psi^n}{\ell}.
\eeq
Also,
\beq\label{e:intermediatestep3}
\frac{1}{N}\sum_{k = 1}^N \frac{\eta}{(\lambda_{j_p} - \lambda_k)^2 + \eta^2}  = \Im \sum_{k = 1}^N \frac{1}{z_{j_p} - \lambda_k} = \Im m_N(s, z_{j_p}).
\eeq
We conclude from \eqref{e:intermediatestep}, \eqref{sumlambdajp}, \eqref{fxiestimate}, \eqref{e:intermediatestep2}, and \eqref{e:intermediatestep3} that with overwhelming probability
\begin{equation}\label{e:363}
\eqref{e:intermediatestep}\le \frac{a_{\widetilde {\bm \xi}}}{N} \sum_{k=1 }^N \frac{\eta h_s(\widetilde {\bm \xi}^{j_p k})}{(\lambda_{j_p} - \lambda_k)^2 + \eta^2} - a_{\widetilde {\bm \xi}} A( k_1, k_2 )    \Im m_N(s, z_{j_p}) + O_n \left( \frac{ \psi^n }{N\eta} + \frac{\psi^n N \eta}{\ell} + \frac{\ell \psi^{n}}{\widetilde d} \right),
\end{equation}
where we also used \eqref{e:Abound} and \eqref{fxiestimate} to restore the term with index $k = j_p$ in the sum, which accrues an error of size $O \left({ \psi^n }/{N\eta}\right)$.

By \eqref{rijestimate} and \eqref{mszestimate}, there exists a constant $c>0$ such that, with overwhelming probability,

\begin{flalign}
\big|  \Im m_N (s, z_{j_p}) - & \Im m_\alpha ( \gamma_{j_p}+ \iu \eta) \big|  \\
 \quad &\le    \big| \Im m_N\big(s, z_{j_p}\big) - \Im m_{\alpha} ( z_{j_p}) \big| + | \Im m_{\alpha} ( z_{j_p} ) - \Im m_{\alpha} ( \gamma_{j_p} + \iu \eta ) \big|    \\
&= \big| \Im m_{\alpha} ( z_{j_p}  ) - \Im m_{\alpha} ( \gamma_{j_p} + \iu \eta ) \big| + O\left( N^{-c} \right).
\label{e:input1}
\end{flalign}
\eer

\noindent  Moreover, by rigidity estimate from the first part of \Cref{l:goodset}, the combination of \eqref{e:gammacompare} and \eqref{e:newrigidity}, and \eqref{e:macontinuity}, %and the fact that $t - s \le t - t_0 = \psi \eta$, 
we have that \eer
\begin{flalign}\label{e:input2}
\Big| \Im m_{\alpha} \big( z_{j_p}  \big) - \Im m_{\alpha} \big( \gamma_{j_p} + \iu \eta \big) \Big| \le C \big| \lambda_{j_p} (s) - \gamma_{j_p}  \big|  + C \eta\eer \le C \left( \displaystyle\frac{\psi}{N} + N^{-c}   + \eta \eer \right) \le C N^{-c},
\end{flalign}

\noindent Thus, \eqref{e:input1} and \eqref{e:input2} yield
\begin{flalign}
\Big| \Im m_N\big( & s, z_{j_p}\big) - \Im m_\alpha ( \gamma_{j_p}+ \iu \eta) \Big| \le C N^{-c}, 
\end{flalign}

\noindent which, together with \eqref{e:Abound}, yields 
\beq 
\Big|  A( k_1, k_2 )    \Im m_N\big(s, z_{j_p}\big)  -   A( k_1, k_2 )  \Im m_\alpha ( \gamma_{j_p}+ \iu \eta)  \Big| \lesssim_n \psi^n N^{-c}
\eeq
with overwhelming probability. Therefore \eqref{e:363} implies
\begin{flalign}\label{e:364}
\eqref{e:thethirdterm}  & \le \frac{a_{\widetilde {\bm \xi}}}{N} \sum_{k=1 }^N \frac{\eta h_s(\widetilde {\bm \xi}^{j_p k})}{\big( \lambda_{j_p}(s) - \lambda_k \big)^2 + \eta^2} - a_{\widetilde {\bm \xi}} A( k_1, k_2 )    \Im m_\alpha ( \gamma_{j_p}+ \iu \eta) \\ & \qquad + O_n\left( \frac{ \psi^{n} }{N\eta} + \frac{\psi^{n} N \eta}{\ell} + \frac{\ell \psi^{n}}{\widetilde d}  + N^{-\fd} \right)
\end{flalign}
holds with overwhelming probability, where we used that $\psi^n N^{-c} = O( N^{-\fd})$ if $\fc$ is chosen sufficiently small (depending only on the value $m$ from \Cref{l:aux1}).  Using the definition of $h_s$ from \eqref{e:observable3} (and recalling from \eqref{e:Qdef} that $a(2j) = (2j - i)!!$), we \eer find that the first term on the right side of \eqref{e:364} is equal to
\begin{flalign}
a_{\widetilde {\bm \xi}} &  \displaystyle\sum_{k = 1}^N \E \left[ \Bigg(  \prod_{1 \le q \le h } \frac{\Big(N \big\langle \q , \u_{j_q} (s) \big\rangle^2 \Big)^{\widetilde{\xi}_q -  \one_{p=q}}}{a\big(2(\widetilde{\xi}_q - \one_{p=q})\big)  } \Bigg) \bigg( \frac{\eta \big\langle \q , \u_k (s) \big\rangle^2 }{(\lambda_{j_p} - \lambda_k)^2 + \eta^2 } \bigg) \displaystyle\frac{a \big( 2 ( \widetilde{\xi}_k ) \big)} {a \big( 2 ( \widetilde{\xi}_k + 1) \big)} \; \Bigg|\;  \bm \lambda, \mathcal F_{t_0} \right] \\
& \le a_{\widetilde {\bm \xi}} \E \left[ \Bigg(  \prod_{1 \le q \le h } \frac{\Big(N \big\langle \q , \u_{j_q} (s) \big\rangle^2 \Big)^{\widetilde{\xi}_q -  \one_{p=q}}}{a\big(2(\widetilde{\xi}_q - \one_{p=q})\big)  } \Bigg) \bigg( \sum_{k = 1}^N \frac{\eta \big\langle \q , \u_k (s) \big\rangle^2 }{(\lambda_{j_p} - \lambda_k)^2 + \eta^2 } \bigg) \; \Bigg|\;  \bm \lambda, \mathcal F_{t_0} \right].  \label{e:whycondition}
\end{flalign}

 We now  have with overwhelming probability that
\begin{align}
 \sum_{k= 1}^N \frac{\eta \big\langle \q , \u_k(s) \big\rangle^2 }{(\lambda_{j_p} - \lambda_k)^2 + \eta^2 } &=  \Im \Big\langle \q , \mathbf R\big(s , \lambda_{j_p}(s)  + \iu \eta\big) \q \Big\rangle  \\
 & \label{e:resolventconstant}= \Im \big\langle \q , \mathbf R(t_0 , \widehat \gamma_{j_p} (s)  + \iu \eta) \q \big\rangle + O_{n, \fc} \Bigg(  \frac{\psi^4}{\sqrt{N\eta}}  +  N^\fc  \bigg( \frac{t - t_0 }{t_0 - \tau} + N^{-c}  \bigg)  \\ & \qquad \qquad \qquad \qquad \qquad \qquad \qquad \qquad  + N^\fc \bigg( N^{-c}  + \frac{\psi N^{-1}  + 2\eta} {t_0 - \tau} \bigg)   \Bigg).
\end{align}
In the last equality, we used \eqref{rt1t2estimate}, \eqref{e:secondconclusion}, $N \eta \gg N^{1/2}$, $s \in [t_0, t]$, and the overwhelming probability estimate 
\beq
\big| \lambda_{j_p}(s) - \widehat \gamma_{j_p} (s) \big| \le \big| \lambda_{j_p}(s) - \gamma_{j_p}(s) \big| + \big| \gamma_{j_p}(s) -  \widehat \gamma_{j_p} (s)\big| \le \frac{\psi}{N} + \eta ,
\eeq
which follows from the rigidity estimate in the first item in \Cref{l:goodset} and \eqref{e:newrigidity} (with $\eta \gg N^{-1/2}$). By \eqref{e:deloc} and the fact that $\psi = N^{\mathfrak{c}}$, this yields
\begin{align}\label{e:whycondition2}
\eqref{e:whycondition} &\le a_{\widetilde {\bm \xi}} \E \left[ \Bigg(  \prod_{1 \le q \le h } \frac{\Big(N \big\langle \q , \u_{j_q} (s) \big\rangle^2 \Big)^{\widetilde{\xi}_q -  \one_{p=q}}}{a\big(2(\widetilde{\xi}_q - \one_{p=q})\big)  } \Bigg)  
\Im \Big\langle \q , \mathbf R \big( t_0 , \widehat \gamma_{j_p} (s)  + \iu \eta \big) \q \Big\rangle \; \Bigg|\;  \bm \lambda, \mathcal F_{t_0} \right] \\  
& \quad + O_{n, \fc} \Bigg(  \frac{\psi^{n + 4}}{\sqrt{N\eta}}  +  \psi^{n + 1}  \bigg( \frac{t - t_0 }{t_0 - \tau} + N^{-c} + \frac{\psi N^{-1}  + 2\eta} {t_0 - \tau} \bigg)   \Bigg).
\end{align}

We now recognize that the second factor inside the expectation on the right side of \eqref{e:whycondition2} is measurable with respect to $\mathcal F_{t_0}$. We may therefore factor it out of the expectation and rewrite the previous bound as % Using in addition the definition \eqref{atqk} of $A (\q , j_p)$, we may rewrite the above equation as 
%Using the definition \eqref{atqk} of $A (\q , j_p)$ and accounting for the conditioning on $\bm \lambda$ and $\mathcal F_{t_0}$,  as 
\begin{flalign}
\eqref{e:whycondition} &\le  a_{\widetilde {\bm \xi}} h_s( \widetilde{\bm \xi}\setminus j_p ) \Im \Big\langle \q , \mathbf R \big( t_0 , \widehat \gamma_{j_p} (s)  + \iu \eta \big) \q \Big\rangle\\
 & \quad + O_{n, \fc} \Bigg(  \frac{\psi^{n + 4}}{\sqrt{N\eta}}  +  \psi^{n + 1}  \bigg( \frac{t - t_0 }{t_0 - \tau} + N^{-c}  \bigg)  + \frac{\psi N^{-1}  + 2\eta} {t_0 - \tau} \bigg)   \Bigg).
\end{flalign}	
%where we used  \eqref{e:deloc} to bound the contribution from the first factor in \eqref{e:whycondition}. 
Using again the computation \eqref{e:resolventconstant} with $s=t$, and \eqref{fxiestimate} yields
\begin{flalign}
\eqref{e:whycondition} &\le a_{\widetilde {\bm \xi}} h_s( \widetilde{\bm \xi}\setminus j_p ) \Im \Big\langle \q , \mathbf R \big( t, \widehat \gamma_{j_p} (s) + \iu \eta \big) \q \Big\rangle \\
& \quad + O_{n, \fc} \Bigg(  \frac{\psi^{n+4}}{\sqrt{N\eta}}  +  \psi^{n+1}  \bigg( \frac{t - t_0 }{t_0 - \tau} +   N^{-c}  + \frac{\psi N^{-1}  + 2 \eta} {t_0 - \tau} \bigg)   \Bigg).
\end{flalign}

\noindent This, together with the definition \eqref{atqk} of $A (\q , j_p)$ and \eqref{e:364}, gives
\begin{flalign}
\label{e:bound3}
\eqref{e:thethirdterm}  & \le  a_{\widetilde {\bm \xi}} \Im m_\alpha ( \gamma_{j_p}+ \iu \eta)   \left( A(\q , j_p) h_s(\widetilde {\bm \xi} \setminus j_p )  - A(k_1, k_2)  \right) \\ 
& \quad  + O_{n, \fc} \Bigg(  \frac{\psi^{n+4}}{\sqrt{N\eta}}  +  \psi^{n+1}  \bigg( \frac{t - t_0 }{t_0 - \tau} +   N^{-c}  + \frac{\psi N^{-1}  + 2\eta} {t_0 - \tau} \bigg)   \Bigg) \\
& \quad +  O_{n, \fc} \left( \frac{ \psi^{4} }{N\eta} + \frac{\psi^{n} N \eta}{\ell} + \frac{\ell \psi^{n}}{\widetilde d} + N^{-\fd} \right).
\end{flalign}

Recalling that in \eqref{e:etadefine}, \eqref{e:ell}, \eqref{e:taudefine}, and \eqref{e:ddef}, we fixed small $\fd (m) >0$ such that $\fd = 50 ( 1 + m) \fc$ (recall $\psi  = N^\fc)$ and chose parameters so that $ N^{-1/2} \ll \eta \ll \tau \le  t_0 \le t$: 
\begin{equation}
\eta = N^{-\mathfrak a}\psi,  \quad \ell = \psi^{5m +1} N^{1+ \fd}\eta,  \quad t_0 =  t -   \psi \eta, \quad \tau = t -  N^{7 \fd} \psi \eta  , \quad \widetilde d = \ell \psi^{5 m} N^\fd.\end{equation}
%We also recall that in \eqref{e:ddef} we set
%\beq
%\widetilde d = \ell \psi^{5 m} N^\fd.
%\eeq
Then choosing  $\fc$ % = \fc(\alpha, \rho,\nu)$ 
sufficiently small, we deduce from \eqref{e:bound3} that
\begin{equation}\label{e:bound4}
\eqref{e:thethirdterm}  \le  a_{\widetilde {\bm \xi}} \Im m_\alpha ( \gamma_{j_p}+ \iu \eta)  \left( A(\q , j_p) h_s(\widetilde {\bm \xi} \setminus j_p )  - A(k_1, k_2)  \right) +  O_n\left(  N^{-\fd} \right).
\end{equation}

To complete the argument, it suffices to show that 
\begin{multline}\label{e:sufficestoshow}
a_{\widetilde {\bm \xi}} \Im m_\alpha ( \gamma_{j_p}+ \iu \eta)   \left( A_t(\q , j_p) h_s(\widetilde {\bm \xi} \setminus j_p )  - A_t(k_1, k_2)  \right)  \\\le  C \psi \left| h_s(\widetilde {\bm \xi} \setminus j_p )  - A_t\big(\chi^{(b+1)}_1 (\widetilde {\bm \xi} \setminus j_p), \chi^{(b+1)}_2 (\widetilde {\bm \xi} \setminus j_p)\big)  \right|  + O_n( N^{-\fd}).
\end{multline}
We recall that, for $j_p \in \big[ (1/2  -{c}) N, (1/2 + {c}) N \big]$,  there exists $C>0$ such that $\big| \Im m_\alpha ( \gamma_{j_p}+ \iu \eta)  \big| < C$, which holds by  \eqref{e:mabounds}\eer. This, together with \eqref{axiak1k2}, and the definition \eqref{atqk} of $A(\mathbf q, \mathbf{\bm \xi})$, we obtain
\begin{flalign}
&\Im m_\alpha ( \gamma_{j_p}+ \iu \eta)     \left( A(\q , j_p) h_s(\widetilde {\bm \xi} \setminus j_p )  - A(k_1, k_2)  \right)  
\\ &= \Im m_\alpha ( \gamma_{j_p}+ \iu \eta)   \left( A(\q , j_p) h_s(\widetilde {\bm \xi} \setminus j_p )  - A\big( \mathbf q, \widetilde{\bm  \xi}\big)  \right)  + O_n \left( N^{-\fd}  \right)\\ 
&=  \Im m_\alpha ( \gamma_{j_p}+ \iu \eta)  A(\q , j_p)   \left( h_s(\widetilde {\bm \xi} \setminus j_p )  - A( \mathbf q, \widetilde{\bm  \xi} \setminus j_p)  \right)  + O_n \left( N^{-\fd}  \right)
\\ &=  \Im m_\alpha ( \gamma_{j_p}+ \iu \eta)  A(\q , j_p)  \left( h_s(\widetilde {\bm \xi} \setminus j_p )  - A\big(\chi^{(b+1)}_1 (\widetilde {\bm \xi} \setminus j_p), \chi^{(b+1)}_2 (\widetilde {\bm \xi} \setminus j_p)\big) + O_n \left( N^{-2 \fd}  \right)  \right)  \\ & \qquad \qquad + O_n \left( N^{-\fd}  \right)
.
\end{flalign}
Combining the last line with \eqref{e:Abound} and using the bound $\big| \Im m_\alpha ( \gamma_{j_p}+ \iu \eta)  \big| < C$ again, we see
\begin{multline}
 \Im m_\alpha ( \gamma_{j_p}+ \iu \eta)  \left( A(\q , j_p) h_s(\widetilde {\bm \xi} \setminus j_p )  - A(k_1, k_2)  \right)  \\ \le  C \psi \left| h_s(\widetilde {\bm \xi} \setminus j_p )  - A\big(\chi^{(b+1)}_1 (\widetilde {\bm \xi} \setminus j_p), \chi^{(b+1)}_2 (\widetilde {\bm \xi} \setminus j_p)\big)  \right|  + O_n(N^{-\fd}).
\end{multline}
Then \eqref{e:sufficestoshow} follows because $|a_{\widetilde {\bm \xi}}| \le 1$. 
\end{proof}
\subsection{Outline of the proof of \Cref{t:dynamics} in the general case}
%\subsection{Outline of the proof of \Cref{t:dynamics} if $|\supp \bm \zeta | > 2$}
\label{OneInterval}
\label{ProofZeta2}

Previously, we assumed when defining the interval $I_a$ in \eqref{e:twosites} that $| \supp \bm \zeta |  = 2$. Consider now the general case where $| \supp \bm \zeta |  = n'$ for $n' \ge 1$. Set $\supp \bm \zeta = \{ i_1, i_2, \ldots , i_{n'} \}$, with $i_1 < i_2 < \cdots < i_{n'}$, recall $m = \mathcal{N} (\bm \zeta)$, and define
\begin{equation}
 I^{(b)}_a ( \bm \zeta) = \bigcup_{j=1}^{n'} I^{(b)}_{a,j}, \quad \text{where for all $1 \le j \le n'$}, \quad I^{(b)}_{a,j} = [ i_j - 10b \tilde d - a, i_j + 10b \tilde d + a]
\end{equation}

\noindent under the assumption that all the intervals $I^{(m)}_{2 \tilde d,j}$ are disjoint; we will describe the appropriate definition when this is not the case below. We further set, for each $j \in [1, n']$ and particle configuration $\bm \xi \in \mathbb{N}^N$, 
\begin{flalign}
\chi^{(b)}_j (\bm \xi) = \sum_{i \in  I^{(b)}_{\widetilde d + \psi \ell, j } } \xi_i.
\end{flalign}

\noindent For any integer $n' \ge 1$ and $n'$-tuple $\textbf{k} = (k_1, k_2, \dots, k_{n'})$ of nonnegative integers, we define 
\begin{flalign}
& \Omega^{(b)}(\textbf{k}) = \Big\{ \bm \xi \in \mathbb{N}^N: \supp \bm \xi \subset I^{(b)}_{\widetilde d + \psi \ell},   \chi^{(b)}_j (\bm \xi)  = k_j \text{ for } j \in [1, n']  \Big\}; \\
& \Omega^{(b)}(n) = \bigcup_{\substack{ |\textbf{k}| = n }} \Omega (\textbf{k}); \qquad A (\textbf{k}) =  \prod_{j=1}^{n'} A ( \q , i_j)^{k_j}, 
\end{flalign}

\noindent where $|\textbf{k}| = \sum_{j = 1}^{n'} k_j$. We also define the restricted intervals
\begin{flalign}
& \Phi^{(b)}(\textbf{k}) = \Big\{ \bm \xi \in \mathbb{N}^N: \supp \bm \xi \subset I^{(b)}_{ -  \psi \ell},   \sum_{i \in  I^{(b)}_{- \psi \ell, j } } \xi_i= k_j \text{ for } j \in [1, n'] \Big\}; \\
& \Phi^{(b)}(n) = \bigcup_{\substack{|\textbf{k}| = n }} \Phi^{(b)} (\textbf{k}).
\end{flalign}

 The operator $\operatorname{Flat}_{a; \textbf{k}}$ is then defined as in \eqref{d:flatdef}, except with $A(k_1, k_2)$ there replaced by $A (\textbf{k})$ here. Given this change, $\av$ is defined as in \eqref{sumfa}. Additionally define the flow $g_s (\bm \xi) = g_s^{(b)} (\bm \xi) = g_s^{(b)} (\bm \xi; \textbf{k})$ as in \eqref{d:gdef}, and also the maximizer $\widetilde{\bm \xi} = \widetilde{\bm \xi} (s) = \widetilde{\bm \xi} (s, \textbf{k}) = \widetilde{\bm \xi}^{(b)} (s, \textbf{k}) \in \N^N$ by
	\beq\label{e:maximizer2}
	g_s (\widetilde{\bm \xi})   = \max_{\bm \xi \in \Omega (\textbf{k}) }  g^{(b)}_s(\bm \xi; \textbf{k}) .
	\eeq
Now the argument then proceeds as before. Specifically, the dichotomy in \Cref{l:aux1} becomes that $g^{(b)}_s (\widetilde{\bm \xi} )$ satisfies either
\beq
\label{e:nothingtoprove12} 
g^{(b)}_s (\widetilde{\bm \xi})  - A (\textbf{k}) \le \frac{1}{N},
\eeq or
\begin{flalign}
\label{e:ultimate2}
\partial_s \big (g^{(b)}_s (\widetilde{\bm \xi})  - A (\textbf{k}) \big ) & \le \frac{C}{\eta} \bigg( \psi \displaystyle\max_{\bm \xi}  \Big| h_s(\bm \xi) - A \big( \chi^{(b+1)}_1 (\bm \xi), \dots, \chi^{(b+1)}_{n'} (\bm \xi) \big) \Big| + N^{-\fd} \bigg) \\
& \qquad  - \frac{c}{\eta} \big(g^{(b)}_s(\widetilde {\bm \xi}) - A (\textbf{k}) \big),
\end{flalign}
where the maximum in \eqref{e:ultimate2} is taken over $\bm \xi \in \bigcup_{1\le j \le n'} \Phi^{(b+1)}(k_1 , \dots, k_j -1, \dots, k_{n'})$. The proof of this claim is the same as the one in \Cref{s:proofofcrefaux1}, and the proof of the main result given this claim is the same as in \Cref{s:dichoproof}.

When the $I^{(m)}_{2 \tilde d,j}$ are not disjoint, we instead partition $\bigcup_{j=1}^{n'} I^{(m)}_{2 \tilde d  ,j}$ into a union of disjoint intervals $\widehat I ^{(m)}_l$ as follows. There exist an integer $v \in [1, n']$ and indices $1 \le j_1 <  j_2 < \ldots < j_v \le n' $ such that the intervals 
\beq
\widehat I^{(m)}_u = \bigcup_{j= j_u}^{j_{u+1} -1 } I^{(m)}_{2 \tilde d, j}
\eeq

\noindent are mutually disjoint over all $u \in [1, v]$ (where we set $j_{v + 1} = n' + 1$), but such that $I^{(m)}_{2 \tilde d, j} \cap I^{(m)}_{2 \tilde d, j + 1}$ is nonempty for each $j \in [j_u, j_{u + 1} - 2]$. We then can make the above definitions using instead the intervals
\begin{equation}
J^{(b)}_{a,l} = [ i_{j_l} - 10 b \tilde d  - a,  i_{j_{l+1} - 1 } + 10b \tilde d + a ],
\end{equation}
which are disjoint for all $a \in [ 0 , 2 \tilde d]$ (since the $\widehat{I}_u^{(m)}$ are).  For instance, we set $\chi^{(b)}_j (\bm \xi) = \sum_{i \in  J^{(b)}_{\widetilde d + \psi \ell, j } } \xi_i$, and
\begin{flalign}
& \Omega^{(b)}(\textbf{k}) = \Big\{ \bm \xi \in \mathbb{N}^N: \supp \bm \xi \subset J^{(b)}_{\widetilde d + \psi \ell},   \chi^{(b)}_j (\bm \xi)  = k_j \text{ for } j \in [1, v]  \Big\},
\end{flalign}

\noindent and similarly for the other intervals and quantities. 

Let us motivate this procedure by very briefly considering the case $|\supp \bm \zeta| = 2$, with $\supp \bm \zeta = \{ i_1, i_2 \}$, as in the material after \eqref{e:twosites}. However, we now suppose that $|i_2 - i_1| \le 20 m \widetilde{d} + 2 \widetilde{d}$, so that the intervals defined in \eqref{iba1iba2} are not disjoint. Then, according to the above, we instead work on the single connected interval 
$J_{a, 1}^{(b)} = [ i_1 - 10 b \tilde d  - a,  i_{2}  + 10b \tilde d + a ]$. 
Then $\av$ is defined as in \eqref{sumfa}, and we observe that \Cref{l:aux3} holds with \eqref{axiak1k2} replaced by the inequality
\begin{flalign} 
\label{axiak1k2one} 
 \displaystyle\max_{b \in [0, m + 1]}   \max_{ \bm \xi \in \Omega^{(b)}(k)} \big| A (\textbf{q}, \bm \xi) - A (k) \big| < C N^{- 3 \fd},
\end{flalign} 
so that $ A (\textbf{q}, \bm \xi)$, up to a small error, does not depend on $\bm \xi$ for  ${ \bm \xi \in \Omega^{(m+1)}(k)}$. Additionally, it is permissible to apply the finite speed of propagation estimate as in the material immediately following \eqref{gsxigsxi}, which would not be the case if we retained two disjoint but nearby or overlapping intervals and attempted the original argument (since then a particle could jump from one interval to the other). The same reasoning underlies the argument in the general case.
\eer

\section{Scaling limit} \label{s:scaling}

In \Cref{s:dynamics} we identified the moments of $\mathbf X_t$ through entries of the resolvent $\textbf{R} (t, z)$. Here, we determine the scaling limit of these entries, as $N$ tends to infinity. In \Cref{s:xtaulaw}, we recall some preliminary material from previous works. In \Cref{s:tightness} we compute the scaling limits of the moments $\E \big[ \Im \rstar (E + \iu \eta)^p \big]$ for $p\in \N$, as $\eta$ tends to $0$, and establish \Cref{p:tightness} as a consequence. In \Cref{s:aetascaling} we compute the scaling limits of the moments $\E \big[ \Im R_{ii} (E + \iu \eta)^p \big]$, as $N$ tends to $\infty$, and prove \Cref{t:main1}. Throughout this section, we recall $t$ from \eqref{t}.

	\subsection{Order parameter for $\X_t$} 
	
	\label{s:xtaulaw}
	
	We now discuss a certain order parameter, which is essentially given by the $\frac{\alpha}{2}$-th moment of (linear combinations of the imaginary and real parts of) $\rstar$. In what follows, for any $u, h \in \mathbb{C}$, we recall from  \cite[Section 5.1]{bordenave2013localization} the inner product 
	\beq 
	h.u = (\Re u)h + (\Im u) \bar h.
	\eeq
		
	\bed\label{d:gstar}  For any $z\in \mathbb{H}$ and $u \in \mathbb{C}$, we define $\gstar(u)\colon \mathbb C \rightarrow \mathbb C$ by
	\beq
	\gstar(u) = \Gamma\left( 1- \frac{\alpha}{2} \right) \E\left[  \big( - \iu \rstar(z).u \big)^{\alpha/2}  \right].%\quad \text{where}\quad h.u = (\Re u)h + (\Im u) \bar h.
	\eeq
	\eed
	
	%As explained in \cite[Section 3.2]{bordenave2017delocalization}, the function $\gstar (u)$ will satisfy a self-consistent equation in a certain functional metric space, defined as follows. 

	The following lemma establishes a lower bound on $\Re \gamma^\star_z$ and the existence of a limit for $\gamma^\star_z$ as $\Im z$ tends to $0$. It will be proved in \Cref{s:appendixA} below.
	
	\bel\label{l:gammalimitlemma} 
	
	There exists a constant $c>0$ such that the following two statements hold. First, we have the uniform lower bound
	\begin{equation}\label{e:starclaim2}
	\displaystyle\inf_{\substack{z \in \mathbb{H} \\ |z| \le c}} \inf_{u \in \mathbb S^1_+} \Re \gstar(u) > c.
	\end{equation}
	
	\noindent Second, for every real number $E \in [-c, c]$, there exists a function $\gamma^\star_E \colon \mathbb{S}^1_+ \rightarrow \mathbb{C}$ such that the following holds. Let $\{ E_j \}_{j \ge 1}$ and $\{ \eta_j \}_{j \ge 1}$ denote sequences of real numbers such that $\lim_{N \rightarrow \infty} E_N = E$ and $\lim_{N \rightarrow \infty} \eta_N = 0$. Then, denoting $z_N = E_N + \iu \eta_N$, we have 
	\beq
	\label{e:starclaim1}
	\lim_{N \rightarrow \infty} \sup_{u \in \mathbb S^1_+} \big| \gamma_{z_N}^\star (u) - \gamma^\star_E (u) \big| = 0.
	\eeq
	
	\eel

	We next recall from \cite[(7.37)]{aggarwal2018goe} notation for a particular analog of $\gstar$ for finite $N$ that will be useful for us. For any real number $s \ge 0$ and index set $\mathcal I \subset \N \cap [ 1, N]$, let $\R^{(\mathcal I)}(s, z) = \big( \textbf{X}_s - z \big)^{-1} = \big\{ R_{ij} (s, z) \big\}$ denote the resolvent of $\X_s^{(\mathcal I)}$, which we define as the matrix $\X_s$ but whose rows and columns with indices in $\mathcal I$ set to zero.
	
	\bed\label{d:orderparameter} 
	
	Fix a real number $s \ge 0$. For any index set $\mathcal I \subset \N \cap [ 1 , N ]$ and complex number $z\in \mathbb{H}$, the function $\gamma^{(\mathcal I)}_z(u)\colon \mathbb{K}^+ \rightarrow \mathbb C$ is defined by
	\begin{equation}
	\label{gammazi}
	\gamma^{\left(\mathcal I\right)}_z (u) = \gamma^{\left(\mathcal I\right)}_{z; s} (u) = \Gamma\left( 1 - \frac{\alpha}{2} \right) \frac{1}{N - \left| \mathcal I   \right|} \sum_{\substack{1 \le k \le N\\ k \notin \mathcal I}} \big( - \iu R^{ (\mathcal I)}_{kk}(s, z).u\big)^{\alpha/2} \frac{|g_k|^\alpha}{\E \big[ |g_k|^\alpha \big]},
	\end{equation}
	
	\noindent where $\textbf{g} = (g_1, g_2, \ldots , g_N)$ is a vector of i.i.d. standard Gaussian random variables\footnote{These Gaussian variables will be useful in \eqref{zjkgk} below, when applying a Hubbard--Stratonovich type transform.} independent from $\textbf{X}_s$. If $\mathcal I = \varnothing$, we abbreviate $\gamma_z = \gamma_z^{(\mathcal{I})}$. 
	
	\eed

	We next have the following local law stating that $\mathbb{E} \big[ \gamma_z (u) \big] \approx \gstar(u)$. It is a consequence of \cite[Theorem 7.6]{aggarwal2018goe}, where the $\Omega_z$ there is equal to $\gstar$ here by \cite[Lemma 4.4]{bordenave2017delocalization}.
	
	\bel[{\cite[Theorem 7.6]{aggarwal2018goe},\cite[Lemma 4.4]{bordenave2017delocalization}}]
	
	\label{l:hatlocal} 
	
	There exist constants $K > 0$ and $C = C (\delta) > 0$ such that the following holds. Fix a real number $\delta > 0$ with $\delta < \max \big\{ \frac{(b - 1/\alpha) ( 2- \alpha)}{20}, \frac{1}{2} \big\}$, and abbreviate $\widetilde{\mathcal{D}} = \widetilde{\mathcal{D}}_{K, \delta}$ (recall \eqref{d2}). Then, for any $s \in [0, t]$, we have
	\begin{flalign}
	\label{gammazuestimate} 
	\displaystyle\sup_{z \in \widetilde{\mathcal{D}}} \sup_{u\in \mathbb S^1_+ } \Big| \E \big[ \gamma_z( u) \big] - \gstar(u) \Big| \le C N^{-\alpha \delta/8},
	\end{flalign}
	
	\noindent  where expectation is taken with respect to both $\textbf{\emph{X}}_s$ and the Gaussian variables $g_k$.
	\eel

	Like \Cref{l:Xlocallaw}, \cite[Theorem 7.6]{aggarwal2018goe} was only stated in the case $s = 0$ in \cite{aggarwal2018goe}, but it is quickly verified that the same proof applies for arbitrary $s \in [0, t]$, especially since $\textbf{H} + s^{1 / 2} \textbf{W}$ satisfies the conditions in \Cref{momentassumption} for $s \in [0, t]$ if $\textbf{H}$ does. 
	
	We next have the following lemma, which can be viewed as an analog of \Cref{d:rde} for finite $N$. In what follows, we recall from \Cref{momentassumption} that there exist random variables $\{ Z_{ij} \}_{1 \le i, j \le }$ that are mutually independent (up to the symmetry condition $Z_{ij} = Z_{ji}$) and have the following properties. First, each $Z_{ij}$ has law $N^{-1 / \alpha} Z$, where $Z$ is $\alpha$-stable; second, each $N^{1 / \alpha} (H_{ij} - Z_{ij})$ is symmetric and has finite variance. 
	
	The first and second bounds in \eqref{estimatesjjrjj} below are consequences of \cite[Proposition 7.11]{aggarwal2018goe} and \cite[Proposition 7.9]{aggarwal2018goe}, respectively. 
	
	\begin{lem}[{\cite{aggarwal2018goe}}]
		
		\label{sjjrjjestimate}

		Define the $\{ Z_{ij} \}$ and, for each integer $j \in [1, N]$, set
		\begin{equation}
		\label{sjj}
		S_{jj} = - \Bigg( z +  \displaystyle\sum_{k \ne j} Z^2_{jk} R^{(j)}_{kk} \Bigg)^{-1}.
		\end{equation}
		
		\noindent Then, with overwhelming probability we have the bounds 
		\begin{flalign}
		\label{estimatesjjrjj}
		& \displaystyle\max_{1 \le j \le N} \E\big[ |S_{jj} - R_{jj}|  \big] \le \frac{ C (\log N)^C }{(N \eta^2)^{\alpha/8}}, \qquad \displaystyle\max_{1\le j \le N} \big\{ |R_{jj}|,  |S_{jj}| \big\} \le C (\log N)^C.
		\end{flalign}

	\end{lem}

	We conclude this section with the following concentration estimate, which is essentially \cite[Proposition 7.17]{aggarwal2018goe}. Although it was only stated in \cite{aggarwal2018goe} for the case when $\mathcal{I}$ is a single index, the fact that it can be extended to all $\mathcal{I}$ of uniformly bounded size is a quick consequence \cite[Lemma 5.6]{aggarwal2018goe}.
	
	\bel[{\cite[Proposition 7.17]{aggarwal2018goe}}]\label{l:1}
	
	There exists a constant $K > 0$ such that the following holds. Fix a real number $\delta > 0$, and abbreviate $\widetilde{\mathcal{D}} = \widetilde{\mathcal{D}}_{K, \delta}$ from \eqref{d2}. For every index set $\mathcal I \subset [ 1 , N ]$, there exists a constant $C = C \big( s, |\mathcal{I}| \big) > 0$ such that, with overwhelming probability, we have
	\beq
	\label{gammazigammaz}
	\displaystyle\sup_{z \in \widetilde{\mathcal{D}}} \displaystyle\sup_{u \in \mathbb{S}_+^1} \Big|  \gamma^{(\mathcal I)}_z(u)  - \E \big[\gamma_z (u) \big] \Big| \le \frac{C (\log N)^C }{N^{s/2}\eta^{\alpha/2}}. 
	\eeq
	
	\noindent Here, the expectation is taken with respect to both $\textbf{\emph{X}}_s$ and the Gaussian variables $g_k$.
	\eel

\subsection{Tightness}\label{s:tightness}

The following lemma computes the scaling limits of moments of $\Im \rstar (E + \iu \eta)$, as $\eta$ tends to $0$. We recall $\gamma^\star_E$ from \Cref{l:gammalimitlemma}.

\bel

\label{l:boundarylimit} 

There exists a constant $c>0$ such that the following holds. Fix a real number $E \in [-c,c]$, and let $\{ E_N \}_{N \ge 1}$ and $\{ \eta_N \}_{N \ge 1}$ be sequences of real numbers such that $\lim_{N\rightarrow \infty} E_N = E$ and $\lim_{N\rightarrow \infty} \eta_N = 0$. Then, for each $p\in \mathbb N$, we have that  
\begin{flalign}
\label{rlimitp}
\lim_{N \rightarrow \infty} \E \Big[  \big( \Im \rstar (E_N + \iu \eta_N) \big)^p  \Big] =  2^{-p} \left( \mathfrak{X} + \overline{\mathfrak{X}} + \displaystyle\sum_{a = 1}^{p - 1} \binom{p}{a} \mathfrak{Y} (a) \right),
\end{flalign}

\noindent where $\mathfrak{X} = \mathfrak{X}_p$ and $\mathfrak{Y} (a) = \mathfrak{Y}_p (a)$ are defined by 
\begin{flalign}
\label{x}
\mathfrak{X} = \displaystyle\frac{1}{\Gamma (p)} \displaystyle\int_{\mathbb{R}_+} t^{p - 1} \exp \big( \iu E t - t^{\alpha / 2} \gamma^\star_E (1) \big) \, dt,
\end{flalign}

\noindent and
\begin{flalign}
\begin{aligned}
\label{y}
\mathfrak{Y} (a) = \frac{1}{\Gamma(a) \Gamma(p - a)} \displaystyle\int_{\mathbb{R}_+^2} & t^{a-1} s^{p - a - 1} \exp \big( \iu E (t - s) \big) \\
& \times \exp \Bigg( - (t^2 + s^2)^{\alpha / 4} \gamma^\star_E \left( \frac{t + \iu s}{\sqrt{t^2 + s^2}}\right) \Bigg) \, dt \,ds. 
\end{aligned}
\end{flalign}

\noindent Thus, the left side of \eqref{rlimitp} exists, depends only on $E$ and $p$, and is uniformly continuous in $E$.
\eel

\begin{proof}
	
For brevity, we set $R_\star = R_\star(E_N + \iu \eta_N)$. We first express moments of $\Im \rstar$ in terms of $\gstar$ (recall \Cref{d:gstar}). To that end, we fix $p\in \mathbb N$ and use the identity $2 \mathrm{i} \Im \rstar = \rstar - \overline{\rstar}$ to write 
\begin{equation}
\label{e:expansion1}
(\Im \rstar )^p  =  (2 \mathrm{i})^{-p} (\rstar - \overline{\rstar})^p   = (2 \mathrm{i})^{-p}  \sum_{a=0}^p \binom{p}{a} R_\star^a (-\overline{R_\star})^{p-a}.
\end{equation}

\noindent So, to establish \eqref{rlimitp}, it suffices to show for each integer $a \in [1, p]$ that
\begin{flalign}
\label{xymoments}
\displaystyle\lim_{N \rightarrow \infty} \E \big[ (- \iu R_\star)^p \big] = \mathfrak{X}; \qquad \displaystyle\lim_{N \rightarrow \infty} \E \big[ (- \iu R_\star)^a ( \iu \overline{R_\star})^{p-a} \big] = \mathfrak{Y} (a).
\end{flalign}

\noindent We only establish the second equality in \eqref{xymoments}, as the proof of the former is entirely analogous.

To that end, let $R_1, R_2, \ldots $ denote  i.i.d.\ complex random variables whose laws are given by $\rstar$, and let $\{ \xi_k \}_{k \ge 1}$ denote a Poisson point process with intensity measure $\big( \frac{\alpha}{2} \big) x^{-\alpha/2 -1}\, dx$ (independent from the $\{ R_k \}$), as in \Cref{d:rde}. Then, \Cref{d:rde} implies 
\begin{flalign} 
\label{expectationabr} 
\mathbb{E} \big[ ( - \iu\rstar)^a (\iu \overline{\rstar})^b \big] = \mathbb{E} \Bigg[ \bigg( - \iu \displaystyle\sum_{k = 1}^{\infty} \xi_k R_k - \iu z \bigg)^{-a} \bigg(  \iu \displaystyle\sum_{k = 1}^{\infty} \xi_k \overline{R_k} + \iu \overline{z} \bigg)^{-b} \Bigg],
\end{flalign} 

\noindent for any $a, b \ge 0$. Next, recall the integral formula
 \begin{equation}
 \label{wbetaintegral}
w^{-\beta}  = \frac{1}{\Gamma(\beta)} \int_{\mathbb{R}_+} t^{\beta-1} \exp( - w t )\,d t, \qquad \text{for $\Re w > 0$ and $\beta>0$}. 
\end{equation}
 For brevity, set $A = \sum_{k=1}^\infty \xi_k R_k(z)$. Abbreviating $z = z_N = E_N + \iu \eta_N$, \eqref{wbetaintegral} implies for $a, b > 0$ that 
\begin{align}\label{e:sasb1}
( - \iu A - \mathrm{i} z)^{-a}	 (\iu \overline{A} + \mathrm{i} \overline{z} )^{-b} &\overset{d}{=} \frac{1}{\Gamma(a) \Gamma(b)} \displaystyle\int_{\mathbb{R}_+^2} t^{a-1} s^{b-1} \exp \big(\iu t (z + A)  -  \iu s(  \bar z + \overline{A}) \big)\, dt \,ds\\
&= \frac{1}{\Gamma(a) \Gamma(b)}\displaystyle\int_{\mathbb{R}_+^2} t^{a-1} s^{b-1} \exp (  \iu tA - \iu s\overline{A} ) \exp\left(  \iu t z -  \iu s \bar z  \right) \, dt \,ds.\label{e:sasb2}
\end{align}
We recall the L\'evy--Khintchine formula (see \cite[(4.5)]{bordenave2011spectrum}): for any i.i.d.\ complex random variables $\{ w_k \}_{k \ge 1}$ such that $\Re w_k \ge 0$ holds almost surely, we have 
\beq\label{e:lk}
\E \Bigg[ \exp \bigg( - \sum_{k = 1}^{\infty} \xi_k w_k \bigg) \Bigg]  = \exp \bigg( - \Gamma \Big( 1 - \frac{\alpha}{2} \Big) \E \big[ w_1^{\alpha / 2} \big] \bigg).
\eeq
Since $\Im R_\star \ge 0$, \eqref{expectationabr}, \eqref{e:sasb2}, and \eqref{e:lk} together imply
\begin{flalign} 
\E \big[ (- \iu \rstar)^a (\iu \overline{\rstar})^b \big] =  \frac{1}{\Gamma(a) \Gamma(b)} \displaystyle\int_{\mathbb{R}_+^2} & t^{a-1} s^{b-1} \exp \bigg( - \Gamma \Big( 1 - \displaystyle\frac{\alpha}{2} \Big) \mathbb{E} \big[ ( \iu s \overline{\rstar} -\iu t \rstar)^{\alpha / 2} \big] \bigg) \\
& \times \exp (  \iu t z -  \iu s \bar z) \, dt \,ds. 
\end{flalign} 

\noindent Recalling $\gstar$ from \eqref{d:gstar}, it follows that
\begin{flalign} 
\label{e:sasb3}
\begin{aligned}
\E \big[ (- \iu \rstar)^a (\iu \overline{\rstar})^b \big] =  \frac{1}{\Gamma(a) \Gamma(b)} \displaystyle\int_{\mathbb{R}_+^2} & t^{a-1} s^{b-1} \exp \bigg( - (t^2 + s^2)^{\alpha / 4} \gstar \Big( \frac{t + \iu s}{\sqrt{t^2 + s^2}} \Big) \bigg) \\
& \times \exp (  \iu t z -  \iu s \bar z) \, dt \,ds.
\end{aligned}
\end{flalign}

Next, observe by \eqref{e:starclaim2}and \eqref{e:starclaim1}, there exists a constant $c>0$ such that
\begin{equation} 
\label{imaginarygammaestimate}
\displaystyle\sup_{\substack{z \in \mathbb{H} \\ |z| < c}} \inf_{u \in \mathbb S^1_+} \Re \gstar(u) > c; \qquad \lim_{N \rightarrow  \infty} \displaystyle\sup_{u \in \mathbb{S}_+^1} \big| \gamma^\star_{z_N} (u) - \gamma^\star_E (u) \big| = 0.
\end{equation}
Therefore, \eqref{e:sasb3}; the dominated convergence theorem; and the fact that 
\begin{flalign} 
\label{stintegral}
\displaystyle\int_{\mathbb{R}_+^2} s^{a - 1} t^{b - 1} \exp \big( - c (s^2 + t^2)^{\alpha / 4} \big) \, ds \, dt < \infty, 
\end{flalign} 

\noindent together imply for $a \in [1, p - 1]$ that $\lim_{N \rightarrow \infty} \E \big[ (- \iu \rstar)^a (\iu \overline{\rstar})^{p - a} \big] = \mathfrak{Y} (a)$; this establishes the second statement in \eqref{xymoments}. The proof of the first is entirely analogous and is therefore omitted. Now \eqref{rlimitp} follows from \eqref{e:expansion1} and \eqref{xymoments}. 

That the left side of \eqref{rlimitp} depends only on $E$ and $p$ holds since the same is true for $\mathfrak{X}$ and $\mathfrak{Y}$. Similarly, to verify the uniform continuity of the left side of \eqref{rlimitp} in $E$, it suffices to do the same for $\mathfrak{X}$ and $\mathfrak{Y}$. The latter follows from the continuity in $E$ for the integrands on the right sides of \eqref{x} and \eqref{y}, the first bound in \eqref{imaginarygammaestimate}, \eqref{stintegral}, and the dominated convergence theorem. 
\end{proof}

\begin{rmk}

\label{rxry}

The proof of \Cref{l:boundarylimit} implies that $\lim_{N \rightarrow \infty} \E \big[ (- \iu \rstar)^a (\iu \overline{\rstar})^{p - a} \big]$ is equal to $\overline{\mathfrak{X}}$ if $a = 0$, to $\mathfrak{Y} (a)$ if $a \in [1, p - 1]$, and to $\mathfrak{X}$ if $a = p$. 

\end{rmk}

Now we can quickly establish \Cref{p:tightness}. 

\begin{proof}[Proof of \Cref{p:tightness}.] 
Since \Cref{l:boundarylimit} implies that $\mathbb{E} \big[ \big( \Im \rstar (E + \iu \eta) \big)^2 \big]$ is uniformly bounded in $\eta > 0$, the sequence $\big\{ \Im \rstar (E + \iu \eta) \big\}_{\eta > 0}$ of random variables is tight. This establishes the first claim of the proposition. The second is a direct consequence of \Cref{l:boundarylimit}.\end{proof}

\subsection{Scaling limit of $A (\textbf{q}, \bm \xi)$}

\label{s:aetascaling}

We begin with the limit of the numerator of $A (\textbf{q}, \bm \xi)$. To compute the scaling limits of the moments of $A(\textbf{q}, \bm \xi)$, we first show that the off-diagonal resolvent entries in the numerator of $A (\textbf{q}, \bm \xi)$ are negligible. Here, we recall the $\widehat{\gamma}_i = \widehat{\gamma}_i (t)$ from \eqref{hatgammaalphai}.

\bel

\label{l:offdiagonal} 

For all real numbers $\delta > 0$; integers $m, n > 0$; and unit vectors $\q = (q_1, q_2, \ldots , q_N) \in \mathbb{R}^N$ with $|\supp \textbf{\emph{q}}| = m$, there exist constants $c > 0$ (independent of $\delta$, $m$, and $n$) and $C = C (\delta, m, n) > 0$ such that the following holds. Let $\{ k_1, k_2, \ldots , k_n \} \subset [1, N]$ denote an index sequence such that $\max_{1 \le j \le n} | k_j - N/2 | < cN$; let $\supp \q = \{ j_1, j_2, \ldots , j_m \}$; and let $t$ be as in \eqref{t}. Then, for $\eta \ge N^{\delta - 1 / 2}$,	
\beq
\label{e:ugh}
\Bigg| \E \bigg[\prod_{i =1}^n \Im \Big\langle \q , \R \big(t , \widehat{\gamma}_{k_i} + \iu \eta \big) \q \Big\rangle \bigg] - \E \bigg[\prod_{i = 1}^n \Im \sum_{h = 1}^m q_{j_h}^2 R_{j_h j_h} \big( t , \widehat{\gamma}_{k_i} + \iu \eta \big) \bigg] \Bigg| \le \frac{C N^\delta}{\sqrt{N\eta}}.
\eeq
\eel

\begin{proof}

First observe that 
\begin{flalign} 
\E \bigg[\prod_{i =1}^n \Im \Big\langle \q , \R \big(t , \widehat{\gamma}_{k_i}  + \iu \eta \big) \q \Big\rangle \bigg] & = \E \bigg[\prod_{i =1}^n \displaystyle\sum_{a = 1}^m \displaystyle\sum_{b = 1}^m q_{j_a} q_{j_b} \Im R_{j_a j_b} \big(t , \widehat{\gamma}_{k_i}  + \iu \eta \big) \bigg] \\
& = \displaystyle\sum_{\textbf{a}, \textbf{b}} \E \bigg[\prod_{i =1}^n q_{j_{a(i)}} q_{j_{b(i)}} \Im R_{j_{a(i)} j_{b(i)}} \big(t , \widehat{\gamma}_{k_i}  + \iu \eta \big) \bigg], \label{productrmn}
\end{flalign}

\noindent where in the right side of \eqref{productrmn}, $\textbf{a} = \big( a(1), a(2), \ldots , a(n) \big)$ and $\textbf{b} = \big( b(1), b(2), \ldots , b (n) \big)$ are summed over all sequences of $\{1, 2, \ldots , m \}^n$. 

It suffices to bound by $C N^{\delta} (N \eta)^{-1 / 2}$ any summand on the right side of \eqref{productrmn} for which there exists some $i' \in [1, n]$ such that $a(i') \ne b(i')$. To that end, observe that the second bound in \eqref{mn0z} (to bound $\big| R_{j_{a(i)} j_{b(i)}} \big| \le N^{\delta / 2n}$ with overwhelming probability for $i \ne i'$); \eqref{kij}; the fact that $q_j \le 1$ for each $j \in [1, N]$; and the exchangeability of the matrix entries of $\textbf{X}_t$ together imply that any such term is bounded by 
\beq
\label{estimaterjarjb}
N^{\delta / 2} \mathbb{E} \bigg[ \Big|R_{12} \big(t, \widehat{\gamma}_{k_{i'}}  + \iu \eta \big) \Big| \bigg].
\eeq

\noindent To estimate this quantity, abbreviate $R_{ij} = R_{ij} (t, \widehat{\gamma}_{k_{i'}} + \mathrm{i} \eta)$, and observe that the Ward identity \eqref{kijsum} and the exchangeability of $\textbf{X}_t$ together imply that
\begin{flalign}
\E \big[ | R_{12}|  \big] \le \Big( \mathbb{E} \big[ | R_{12}|^2 \big] \Big)^{1 / 2} \le \Bigg( \mathbb{E} \bigg[ \displaystyle\frac{1}{N - 1} \displaystyle\sum_{j = 1}^N | R_{1j}|^2 \bigg] \Bigg)^{1 / 2} \le  \displaystyle\frac{2 \mathbb{E} [ \Im R_{11} ]^{1 / 2}}{\sqrt{N\eta}}.
\end{flalign}

\noindent Using the second bound in \eqref{mn0z}, and the deterministic bound \eqref{kij} on the exceptional set where the former estimate does not apply, yields $\mathbb{E} [ \Im R_{11} ] \le C (\log N)^C$, and so 
\beq
 \E \big[ |R_{12} | \big] \le \frac{C N^{\delta / 2}}{\sqrt{N\eta}}.
\eeq

\noindent Together with \eqref{productrmn} and \eqref{estimaterjarjb}, this implies \eqref{e:ugh}. 
\end{proof}

In \cite[Theorem 2.8]{bordenave2011spectrum} it was shown for fixed $z \in \mathbb{H}$ that the diagonal resolvent elements $G_{ii}(z)$ of the matrix $\bH$ are asymptotically independent. The next lemma is a version of this result (for the perturbed model $\X_{t}$) when $\eta = \Im z$ is simultaneously tending to $0$. \Cref{t:main1} is then deduced quickly as a consequence. Below, we recall $\mathcal R_\star(E)$ from \Cref{d:arbitrary}.

\begin{prop}

\label{l:diagElimit}

There exists a constant $c>0$ such that the following holds. Fix integers $m, n > 0$ and a unit vector $\q = (q_1, q_2, \ldots , q_N) \in \mathbb{R}^N$ with $|\supp \textbf{\emph{q}}| = m$. Let $\supp \q = \{ j_1, j_2, \ldots , j_m \}$, and let $t \in \mathbb{R}_{> 0}$ be as in \eqref{t}. Fix a real number $E \in [-c, c]$, and let $\{ \eta_N \}_{N \ge 1}$ and $\big\{ E_N^{(i)} \big\}_{N \ge 1}$ for each integer $i \in [1, n]$ be sequences of real numbers such that 
\begin{flalign} \label{ehypothesis}
\lim_{N\rightarrow \infty} \eta_N = 0; \qquad \eta_N \gg N^{-1 / 2}; \qquad \lim_{N\rightarrow \infty} E_N^{(1)} = E; \qquad \displaystyle\max_{1 \le i \le n} \big| E_N^{(i)} - E_N^{(1)} \big| \ll \eta_N.
\end{flalign} 

\noindent Then, letting $\big\{ \mathcal{R}_{j_i} (E) \big\}_{i \in [1, m]}$ be i.i.d.\ random variables each with law $\mathcal{R}_{\star} (E)$, we have that 
\beq
\label{e:diagElimit}
\lim_{N\rightarrow\infty} \E\left[\prod_{i = 1}^n \Im \sum_{k = 1}^m q_{j_k}^2 R_{j_k j_k} \big( t , E_N^{(i)} + \iu\eta_N \big) \right] =    \E \Bigg[ \bigg( \sum_{k = 1}^m q_{j_k}^2 {\mathcal R}_{j_k} (E) \bigg)^n \Bigg].
\eeq

\end{prop} 

\begin{proof}
	
	It suffices to show that, for any sequences of nonnegative integers $\textbf{n}^{(i)} = \big( n_1^{(i)}, n_2^{(i)}, \ldots , n_m^{(i)} \big)$ for $1 \le i \le n$ with $N_k = \sum_{i = 1}^n n_k^{(i)}$, we have 
	\beq
	\label{limitproductrjrj}
	\lim_{N\rightarrow\infty} \E \Bigg[\prod_{i = 1}^n \displaystyle\prod_{k = 1}^m \Big( \Im R_{j_k j_k} \big( t , E_N^{(i)} + \iu\eta_N \big) \Big)^{n_k^{(i)}} \Bigg] =   \displaystyle\prod_{k = 1}^m \E \big[ \mathcal{R}_{\star} (E)^{N_k} \big].
	\eeq
	
	To ease notation, we detail the proof of \eqref{limitproductrjrj} when $(m, n) = (1, 1)$ and outline it when $(m, n) = (1, 2)$ and $(m, n) = (2, 2)$, which are largely analogous. We omit the proofs in the remaining cases, since they are very similar to those of the $(m, n) \in \big\{ (1, 2), (2, 2) \big\}$ cases.

	To that end, first assume $(m, n) = (1, 1)$; abbreviate $j_1 = j$, $z = E_N^{(1)} + \iu \eta_N$, and $R_{ik} = R_{ik} (t, z)$; and set $p = n_1^{(1)}$. We compute $\lim_{N \rightarrow \infty} \E \big[ (\Im R_{jj})^p  \big]$. As in the proof of \Cref{l:boundarylimit}, we use the identity $2 \mathrm{i} \Im R_{jj} = R_{jj} - \overline{R}_{jj}$ to write 
	\begin{equation}
	\label{e:expansion}
	\E \big[ (\Im R_{jj})^p  \big] =  (2 \mathrm{i})^{-p} \E \big[ (R_{jj} - \overline{R}_{jj})^p  \big] = (2 \mathrm{i})^{-p} \sum_{a=0}^p (-1)^{p - a} \binom{p}{a} \E \big[  R_{jj}^a \overline{R}_{jj}^{p-a} \big].	
	\end{equation}
	
	Now, recall from \Cref{momentassumption} that each entry of $\textbf{H}$ has law $N^{-1 / \alpha} (Z + J)$, where $Z$ is $\alpha$-stable and $J$ has finite variance. For each $1 \le i \le j \le N$, let $\{ Z_{ij} \}$ denote mutually independent random variables with law $N^{-1 / \alpha} Z$ such that $N^{1 / \alpha} |H_{ij} - Z_{ij}|$ has uniformly bounded variance. Following \eqref{sjj}, for any subset $\mathcal{I} \subset [1, N]$ and index $j \in [1, N] \setminus \mathcal{I}$, set $\mathcal{J} = \mathcal{I} \cup \{ j \}$ and define 
	\begin{flalign}
	\label{sjj2}
	S_{jj}^{(\mathcal{I})} = S_{jj}^{(\mathcal{I})} (z) = - \Bigg( z + \displaystyle\sum_{k \notin \mathcal{J}} Z_{kj}^2 R_{kk}^{(\mathcal{J})} \Bigg)^{-1}.
	\end{flalign} 
	
	\noindent If $\mathcal{I}$ is empty, then we abbreviate $S_{jj} = S_{jj}^{(\mathcal{I})}$. 

Then \eqref{estimatesjjrjj} and the deterministic bounds given by \eqref{kij} and $|S_{jj}| \le \eta^{-1}$ together imply that there exists a constant $C = C(p) > 0$ such that
	\beq\label{e:l54}
	\E \big[ |R_{jj}^a  - S_{jj}^a|  \big] + \E \big[ | \overline{R}_{jj}^a  - \overline{S}_{jj}^a|  \big]  \le \frac{ C (\log N)^C }{(N \eta^2)^{\alpha/8}},
	\eeq
	
	\noindent for any $a \in [0, p]$. Then, \eqref{e:l54} and \eqref{estimatesjjrjj} together imply for any $a, b \in [0, p]$ that 
	\begin{align}
	\label{rjjsjjbestimate}
	\mathbb{E} \big[ |R_{jj}^a \overline{R}_{jj} ^b  - S_{jj}^a \overline{S }_{jj}^b| \big] & \le \mathbb{E} \big[ |R_{jj}|^a |\overline{R}_{jj} ^b -  \overline{S }_{jj}^b| \big] + \mathbb{E} \big[ |S_{jj}|^b |R_{jj} ^a - S_{jj}^a| \big] \le \displaystyle\frac{C (\log N)^C}{(N \eta^2)^{\alpha / 8}}.
	\end{align}
	
	\noindent So, by \eqref{e:expansion}, the definition \eqref{limitresolvent} of $\rstar$, and \Cref{p:tightness}, it suffices to show for each integer $a \in [0, p]$ that 
	\begin{flalign}
	\label{sjjpaxy}
	\displaystyle\lim_{N \rightarrow \infty} \mathbb{E} \big[ S_{jj}^a \overline{S}_{jj}^{p - a} \big] = \displaystyle\lim_{\eta \rightarrow 0} \mathbb{E} \big[ \rstar (E + \iu \eta)^{a} \overline{\rstar} (E + \iu \eta)^{p - a}\big].
	\end{flalign} 
	
	\noindent Recalling \eqref{x}, \eqref{y}, and \Cref{rxry}, the right side of is equal to $\iu^{-p} \overline{\mathfrak{X}}$ if $a = 0$, to $\iu^{-p} (-1)^a \mathfrak{Y} (a)$ if $p \in [1, a - 1]$, and to $\iu^p \mathfrak{X}$ if $a = p$. Let us only show \eqref{sjjpaxy} in the case $a \in [1, p - 1]$, as the cases $a \in \{ 0, p \}$ are entirely analogous. 
	
	To that end, we proceed similarly to as in the proof of \Cref{l:boundarylimit}. More specifically, by \eqref{sjj2} and \eqref{wbetaintegral}, we deduce that 
	\begin{align}
	\label{e:sasb}
	( - \iu & S_{jj})^a (\iu \overline{S}_{jj})^b = \frac{1}{\Gamma(a) \Gamma(b)} \displaystyle\int_{\mathbb{R}_+^2} t^{a-1} s^{b-1} \exp \Bigg( \iu t z -  \iu s \bar z + \iu \sum_{k\neq i }  Z^2_{ik} \big( t R_{kk}^{(j)}  -  s\overline{R}_{kk}^{(j)} \big)  \Bigg) \, dt \,ds.
\end{align}

	\noindent To analyze the right side of \eqref{e:sasb}, observe by \cite[Corollary B.2]{bordenave2013localization} (whose proof proceeds by first applying a type of Hubbard--Stratonovich transform to linearize the exponential in the $\{ Z_{jk} \}$, and then using \eqref{e:lk} to evaluate the expectation) 
\begin{equation}
\label{zjkgk}
\E \Bigg[ \exp \bigg(  \iu \sum_{k \neq j}  Z_{jk}^2 \big( tR^{(j)}_{kk}  -  s\overline{R}^{(j)}_{kk}  \big)  \bigg) \Bigg] = \E \Bigg[\exp \bigg( - \frac{( -2\iu)^{\alpha/2} \sigma^\alpha}{N} \sum_{k\neq j} \big( tR^{(j)}_{kk}  -  s\overline{R}_{kk}^{(j)} \big)^{\alpha/2}   |g_k|^{\alpha} \bigg) \Bigg],
\end{equation}

\noindent where $\sigma >0$ is as in \eqref{stable}, the $g_k$ are i.i.d.\ standard Gaussian random variables, and the expectation is taken with respect to the $Z_{jk}$ on the left side and the $g_k$ on the right. 

Therefore, by the definition \eqref{gammazi} of $\gamma^{(j)}_z$, we find 
\begin{equation}
\label{e:514}
\mathbb{E} \Bigg[ \exp  \bigg( - \frac{( -2\iu)^{\alpha/2} \sigma^\alpha}{N} \sum_{k\neq j} \big( tR^{(j)}_{kk}  -  s\overline{R}^{(j)}_{kk}  \big)^{\alpha/2}   |g_k|^{\alpha} \bigg) \Bigg] = \E  \left[\exp \left( - \frac{N-1}{N} \gamma^{(j)}_z (t + \iu s) \right) \right],
\end{equation}

\noindent where we have used the fact (see \cite[(7.39)]{aggarwal2018goe}) that $\E \big[ |g_k|^{\alpha}\big] = 2^{-\alpha / 2} \sigma^{-\alpha} \Gamma \big( 1 - \frac{\alpha}{2} \big)$. Thus,
\begin{equation}\label{e:516}
\eqref{e:514} =  \E \Bigg[ \exp \bigg( -  \frac{N-1}{N} (t^2 + s^2)^{\alpha/4} \gamma^{(j)}_z \Big( \displaystyle\frac{t + \iu s}{\sqrt{t^2 + s^2}} \Big) \bigg) \Bigg].
\end{equation}
Using \eqref{gammazigammaz} and \eqref{gammazuestimate}, the fact that $\Im z \gg N^{-1 / 2}$, and the deterministic estimate $\Re \gamma^{(j)}_z(u) \ge 0$ on the exceptional event where \eqref{gammazigammaz} and \eqref{gammazuestimate} do not hold, we obtain 
\begin{equation}
\label{stgamma} 
\eqref{e:516} =  \exp \bigg( -  \frac{N-1}{N} (t^2 + s^2)^{\alpha/4} \gstar \Big( \displaystyle\frac{t + \iu s}{\sqrt{t^2 + s^2}} \Big) + O (N^{-c}) \bigg) + O (N^{-10}).
\end{equation}

\noindent Combining \eqref{e:sasb}, \eqref{zjkgk}, \eqref{e:514}, \eqref{e:516}, and \eqref{stgamma} yields 
\begin{align}
\E & \left[ ( - \iu S_{jj})^a (\iu \overline{S}_{jj})^b\right] \\
 & \quad = \frac{1}{\Gamma(a) \Gamma(b)} \displaystyle\int_{\mathbb{R}_+^2} t^{a-1} s^{b-1} \exp ( \iu t z -  \iu s \bar z) \\
& \qquad \times \Bigg ( \exp \bigg( -  \frac{N-1}{N} (t^2 + s^2)^{\alpha/4} \gstar \Big( \displaystyle\frac{t + \iu s}{\sqrt{t^2 + s^2}} \Big) + O (N^{-c}) \bigg) + O (N^{-10}) \Bigg) \, dt \,ds \label{sjjasjjb}.
\end{align}

By \eqref{imaginarygammaestimate} (with the $z_N$ there equal to $z$ here), \eqref{stintegral}, and the fact that $\Im z \gg N^{-1 / 2}$, we deduce from \eqref{sjjasjjb} and the dominated convergence theorem that 
\begin{align}
\lim_{N\rightarrow \infty} \mathbb{E} & \left[( - \iu S_{jj})^a (\iu \overline{S}_{jj})^b \right] \\
& =  \frac{1}{\Gamma(a) \Gamma(b)} \displaystyle\int_{\mathbb{R}_+^2} t^{a-1} s^{b-1} \exp \bigg( \iu E (t - s) - (t^2 + s^2)^{\alpha/4} \gamma^\star_E \Big( \displaystyle\frac{t + \iu s}{\sqrt{t^2 + s^2}} \Big) \bigg) \, dt \,ds.
\end{align}

\noindent By \Cref{rxry} and \eqref{y}, this yields \eqref{sjjpaxy} when $a \in [1, p - 1]$. The cases when $a \in \{ 0, p \}$ are handled analogously and therefore omitted. This therefore establishes \eqref{limitproductrjrj} in the case $m = 1 = n$. 

Next let us outline how to establish \eqref{limitproductrjrj} in the case $(m, n) = (1, 2)$. We abbreviate $z_1 = E_N^{(1)} + \iu \eta_N$, $z_2 = E_N^{(2)} + \iu \eta_N$, and $j = j_1$. Then following \eqref{e:expansion}, %, \eqref{e:l54}, \eqref{rjjsjjbestimate}, \eqref{sjjpaxy}, 
it suffices to show for any integers $a, b, c, d \ge 0$ that
\begin{flalign}
\label{s12rabcd2}
\lim_{N \rightarrow \infty} \E & \big[ R_{jj} (z_1)^a \overline{R}_{jj} (z_1)^b R_{jj} (z_2)^c \overline{R}_{jj} (z_2)^d \big] = \displaystyle\lim_{\eta \rightarrow 0} \mathbb{E} \big[ \rstar (E + \iu \eta)^{a + c} \overline{\rstar} (E + \iu \eta)^{b + d} \big].
\end{flalign}

We  write
\begin{align}
 \E \big[ R_{jj} (z_1)^a \overline{R}_{jj} (z_1)^b R_{jj} (z_2)^c \overline{R}_{jj} (z_2)^d \big]  & = \big[ R_{jj} (z_1)^{a + c} \overline{R}_{jj} (z_1)^{b + d} \big] \\
 & \quad + \E  \Big[ R_{jj} (z_1)^a \overline{R}_{jj} (z_1)^{b + d} \big(R_{jj} (z_2)^c  - R_{jj} (z_1)^c \big) \Big] \label{rjjabdc} \\
 & \quad + \label{e:errort1} \E \Big[ R_{jj} (z_1)^a \overline{R}_{jj} (z_1)^b R_{jj} (z_2)^c \big(\overline{R}_{jj} (z_2)^d  - \overline{R}_{jj} (z_1)^d  \big) \Big].
 \end{align}
 The first term is the main one, and it was shown in the preceding case, as \eqref{rjjsjjbestimate} and \eqref{sjjpaxy}, that
 \begin{flalign}
 \label{rjjabcd} 
\lim_{N \rightarrow \infty} \E & \big[ R_{jj} (z_1)^{a+c} \overline{R}_{jj} (z_1)^{b+d}  \big] = \displaystyle\lim_{\eta \rightarrow 0} \mathbb{E} \big[ \rstar (E + \iu \eta)^{a + c} \overline{\rstar} (E + \iu \eta)^{b + d} \big].
\end{flalign}
The latter two terms are error terms, and they tend to zero asymptotically. Let us show this for \eqref{e:errort1}, as the other term is similar. 

Since $ R_{jj}(z) - R_{jj}(w) = (w-z)\sum_{a=1}^N R_{ja}(z) R_{aj}(w)$ by \eqref{ab}, we have that 
\begin{equation}
\big| \partial_z R_{jj}(z) \big| \le \left|  \sum_{a=1}^N R_{ja}(z) R_{aj}(z) \right| \le  \sum_{a=1}^N \big| R_{ja}(z) \big|^2 = \frac{\Im R_{jj}(z)}{\eta} \le (\log N)^C \eta^{-1}
\end{equation}
with overwhelming probability, where in the equality we used \eqref{kijsum}, and in the last bound we used \eqref{mn0z}. Integrating  $\partial_z R_{jj}(z)^d = d R_{jj}^{d-1} \partial_z R_{jj}$ from $z_1$ to $z_2$ yields 
\begin{equation} \label{e:opr}
\big|R_{jj}(z_1)^d - R_{jj}(z_2)^d\big| \le d |z_1 - z_2|  (\log N)^{dC} \eta^{-1} \end{equation}
with overwhelming probability. % here we used the hypothesis that $| E^{(1)} - E^{(2)}| \ll \eta $ to bound $|z_1 - z_2|$ and \eqref{mn0z} to bound $\big| R_{ii}(z) \big|$. %We also used $\eta \gg N^{-1/2}$ and took $\delta$ sufficiently small. 
Therefore, with overwhelming probability, we have
\begin{flalign}
& \left| R_{jj} (z_1)^a \overline{R}_{jj} (z_1)^b R_{jj} (z_2)^c \big(\overline{R}_{jj} (z_2)^d  - \overline{R}_{jj} (z_1)^d  \big)\right| \\ & \qquad \ll C | z_1 - z_2| (\log N)^{dC}  \eta^{-1} \left| R_{jj} (z_1)^a \overline{R}_{jj} (z_1)^b {R}_{jj} (z_2)^c \right|\\
& \qquad \ll |z_1 - z_2 | \eta^{-1} (\log N)^{(a+b+c+d)C}   \ll 1,
\end{flalign}
where we used \eqref{mn0z} in the last estimate, as well as the hypothesis \eqref{ehypothesis} that $| E_N^{(1)} - E_N^{(2)}| \ll \eta $ to bound $|z_1 - z_2|$. On the complementary event, we use the trivial bound \eqref{kij}. Together, these show that 
\beq
\lim_{N\rightarrow \infty} \E \Big[ R_{jj} (z_1)^a \overline{R}_{jj} (z_1)^b R_{jj} (z_2)^c \big(\overline{R}_{jj} (z_2)^d  - \overline{R}_{jj} (z_1)^d  \big) \Big] = 0,
\eeq
as desired. We have therefore established that \eqref{e:errort1} is small; since the same holds for \eqref{rjjabdc}, \eqref{rjjabcd} implies \eqref{s12rabcd2}.

Now let us outline how to establish \eqref{limitproductrjrj} in the case $(m, n) = (2, 2)$. We abbreviate $z_1 = E_N^{(1)} + \iu \eta_N$ and $z_2 = E_N^{(2)} + \iu \eta_N$, and we assume for notational convenience that $j_1 = 1$ and $j_2 = 2$. As in \eqref{e:expansion}, it suffices to show for any integers $a,b,c,d \ge 0$ that
\begin{flalign}
\lim_{N \rightarrow \infty} \E & \left[ R_{11} (z_1)^a \overline{R}_{11} (z_1)^b R_{22} (z_2)^c \overline{R}_{22} (z_2)^d \right] \\
&  = \displaystyle\lim_{\eta \rightarrow 0} \mathbb{E} \big[ \rstar (E + \iu \eta)^a \overline{\rstar} (E + \iu \eta)^b \big] \mathbb{E} \big[ \rstar (E + \iu \eta)^c \overline{\rstar} (E + \iu \eta)^d \big].
\end{flalign} 

\noindent In what follows, we assume that $a, b, c, d > 0$ for notational simplicity. Note that \cite[Lemma 5.5]{bordenave2017delocalization} implies for each $i \neq j$ that
\beq\label{e:interlacing}
\E \Big[ \big| R_{jj} - R_{jj}^{(i)}  \big| \Big] \le \frac{C}{N \eta}.
\eeq

\noindent It quickly follows from %the first statement of \eqref{estimatesjjrjj} and 
\eqref{e:interlacing}, \eqref{estimatesjjrjj}, the deterministic bound \eqref{kij} that 
\beq
\lim_{N \rightarrow \infty} \E \big[ R_{11} (z_1)^a \overline{R}_{11} (z_1)^b R_{22} (z_2)^c \overline{R}_{22} (z_2)^d \big] = \lim_{N \rightarrow \infty} \E \Big[ R_{11}^{(2)} (z_1)^a \overline{R}_{11}^{(2)} (z_1)^b R_{22}^{(1)} (z_2)^c \overline{R}_{22}^{(1)} (z_2)^d \Big],
\eeq
as in \eqref{rjjsjjbestimate}. As before, \eqref{estimatesjjrjj} and \eqref{kij} together imply that 
\beq
\lim_{N \rightarrow \infty} \E \Big[ R_{11}^{(2)} (z_1)^a \overline{R}_{11}^{(2)} (z_1)^b R_{22}^{(1)} (z_2)^c \overline{R}_{22}^{(1)} (z_2)^d \Big] = \lim_{N \rightarrow \infty} \E \Big[ S_{11}^{(2)} (z_1)^a \overline{S}_{11}^{(2)} (z_1)^b S_{22}^{(1)} (z_2)^c \overline{S}_{22}^{(1)} (z_2)^d \Big], 
\eeq

\noindent and so it suffices to show that 
\begin{flalign}
\label{s12rabcd}
\lim_{N \rightarrow \infty} \E & \Big[ S_{11}^{(2)} (z_1)^a \overline{S}_{11}^{(2)} (z_1)^b S_{22}^{(1)} (z_2)^c \overline{S}_{22}^{(1)} (z_2)^d \Big] \\
& = \displaystyle\lim_{\eta \rightarrow 0} \mathbb{E} \big[ \rstar (E + \iu \eta)^a \overline{\rstar} (E + \iu \eta)^b \big] \mathbb{E} \big[ \rstar (E + \iu \eta)^c \overline{\rstar} (E + \iu \eta)^d \big].
\end{flalign}

Once again using \eqref{sjj2} and \eqref{wbetaintegral}, we find
\begin{flalign}
\E & \Big[ \big( - \iu S_{11}^{(2)} (z_1) \big)^a \big( \iu \overline{S}_{11}^{(2)} (z_1) \big)^b \big( - \iu S_{22}^{(1)} (z_2) \big)^c \big( \iu \overline{S}_{22}^{(1)} (z_2) \big)^d \Big] \\
&= \frac{1}{\Gamma(a) \Gamma(b)\Gamma(c) \Gamma(d) } \displaystyle\int_{\mathbb{R}_+^4} t^{a-1} s^{b-1} x^{c-1} y^{d-1} \exp\left(  \iu t z_1 + \iu x z_2 -  \iu s \overline{z_1} - \iu y \overline{z_2} \right)\\
& \quad \times \E \Bigg[ \exp \bigg(  \iu \sum_{k\notin \{ 1,2 \} }  Z^2_{1k} \big(tR^{(12)}_{kk}  -  s\overline{R}^{(12)}_{kk} \big)  \bigg) \exp \bigg(  \iu \sum_{k \notin \{ 1, 2 \} }  Z^2_{2k} \big( x R^{(12)}_{kk}  -  y\overline{R}^{(12)}_{kk}  \big)  \bigg) \Bigg] \, dt \,ds \, dx \, dy.
\end{flalign}

We now condition on $\{ h_{ij} \}_{i, j \notin \{ 1, 2 \}}$, which makes the two exponential terms in the previous line conditionally independent. Then by following \eqref{zjkgk}, \eqref{e:514}, \eqref{e:516}, \eqref{stgamma}, and \eqref{sjjasjjb}, we obtain
\begin{flalign}
\E & \Big[ \big( - \iu S_{11}^{(2)} (z_1) \big)^a \big( \iu \overline{S}_{11}^{(2)} (z_1) \big)^b \big( - \iu S_{22}^{(1)} (z_2) \big)^c \big( \iu \overline{S}_{22}^{(1)} (z_2) \big)^d \Big] \\
&= \frac{1}{\Gamma(a) \Gamma(b)\Gamma(c) \Gamma(d) } \displaystyle\int_{\mathbb{R}_+^4} t^{a-1} s^{b-1} x^{c-1} y^{d-1} \exp\left(  \iu t z_1 + \iu x z_2 -  \iu s \overline{z_1} - \iu y \overline{z_2} \right)\\
& \quad \times \Bigg( \exp \Bigg( -  \frac{N-2}{N} \bigg( (t^2 + s^2)^{\alpha/4}  \gamma_{z_1}^\star\eer \Big( \displaystyle\frac{t + \iu s}{\sqrt{t^2 + s^2}} \Big) + (x^2 + y^2)^{\alpha/4}  \gamma_{z_2}^\star\eer  \bigg( \displaystyle\frac{x + \iu y}{\sqrt{x^2 + y^2}} \bigg) + O (N^{-c}) \bigg)  \Bigg) \\
& \qquad + O (N^{-10}) \Bigg) \, dt \,ds \, dx \, dy.
\end{flalign}

\noindent Thus \eqref{imaginarygammaestimate} (with the $z_N$ there equal to $z_1$ and $z_2$ here), \eqref{stintegral}, the dominated convergence theorem, \eqref{y}, and \Cref{rxry} together give 
\begin{flalign}
& \displaystyle\lim_{N \rightarrow \infty} \E \Big[ \big( - \iu S_{11}^{(2)} (z_1) \big)^a \big( \iu \overline{S}_{11}^{(2)} (z_1) \big)^b \big( - \iu S_{22}^{(1)} (z_2) \big)^c \big( \iu \overline{S}_{22}^{(1)} (z_2) \big)^d \Big] \\
& \quad = \frac{1}{\Gamma(a) \Gamma(b)\Gamma(c) \Gamma(d) } \displaystyle\int_{\mathbb{R}_+^4} t^{a-1} s^{b-1} x^{c-1} y^{d-1} \exp \big(  \iu E (t - s + x - y) \big)\\
& \qquad \times \exp \bigg( - (t^2 + s^2)^{\alpha/4} \gamma^\star_E \Big( \displaystyle\frac{t + \iu s}{\sqrt{t^2 + s^2}} \Big) - (x^2 + y^2)^{\alpha/4} \gamma^\star_E \Big( \displaystyle\frac{x + \iu y}{\sqrt{x^2 + y^2}} \Big) \bigg) \, dt \,ds \, dx \, dy \\
& \quad = \mathfrak{Y}_{a + b} (a) \mathfrak{Y}_{c + d} (c) = \displaystyle\lim_{\eta \rightarrow 0} \mathbb{E} \big[ \rstar (E + \iu \eta)^a \overline{\rstar} (E + \iu \eta)^b \big] \mathbb{E} \big[ \rstar (E + \iu \eta)^c \overline{\rstar} (E + \iu \eta)^d \big],
\end{flalign}

\noindent from which we deduce \eqref{s12rabcd}. 
\end{proof}

\begin{proof}[Proof of \Cref{t:main1}.]
	
	We will apply \Cref{l:diagElimit}, with the $\eta_N$ there equal to the $\eta = N^{\mathfrak{c} - \mathfrak{a}}$ here (recall \eqref{e:parameters}) and the $E^{(j)}_N$ there equal to the $\widehat \gamma_{k_j}$ here. To that end, we must verify the assumptions \eqref{ehypothesis} of that proposition. The first and second statements there follow from the fact that  $\eta = N^{\mathfrak{c} - \mathfrak{a}}$, that $\mathfrak{c}$ is sufficiently small, and the fact that $\mathfrak{a} < \frac{1}{2}$ (by \eqref{e:parameters}). The third follows from the fact that $\lim_{N \rightarrow \infty} \gamma_{k_1} = E$ and \eqref{e:gammacompare}.
	
	To verify the fourth, we must show that $| \widehat \gamma_{k_1} - \widehat \gamma_{k_j} | \ll \eta$. To that end, observe since $| k_1 - k_j| \le N^{1/2}$, we have for any fixed $\delta > 0$ and $j \in [1, n]$ that
\begin{flalign}
| \widehat \gamma_{k_1} - \widehat \gamma_{k_j} | & \le \big|\widehat \gamma_{k_1} - \gamma_{k_j} (t) \big| + \big|\widehat \gamma_{k_j} - \gamma_{k_j} (t) \big| + \big| \gamma_{k_1} (t) - \gamma_{k_j} (t) \big| \\
& \dom N^{\delta  - 1/2} + N^{1 + 4 \mathfrak{c}} | k_1 - k_j| \dom N^{\delta + 4 \mathfrak{c} - 1/2}
\end{flalign} 

\noindent for sufficiently large $N$, where we used \eqref{e:newrigidity} and \eqref{e:xtdispersion}. Then, since $\eta \gg N^{-1/2}$, we may choose $\delta$ and $\mathfrak{c}$ small enough that the last bound in \eqref{ehypothesis} of \Cref{l:diagElimit} is also satisfied.
 
 Now the theorem follows from \Cref{l:offdiagonal}; \Cref{l:diagElimit}; the facts that $\lim_{N \rightarrow \infty} \Im m_{\alpha} (\gamma_k + \mathrm{i} \eta) = \pi \rho_{\alpha} (E)$ (as $\lim_{N \rightarrow \infty} \gamma_k = E$) and $\mathcal{U}_{\star} (E) = \big( \pi \varrho_{\alpha} (E) \big)^{-1} \mathcal{R}_{\star} (E)$ (see \Cref{d:arbitrary}); and \eqref{e:mabounds}. 
\end{proof} 

\appendix

\section{Proofs of results from \Cref{s:preliminary} and \Cref{s:scaling}} \label{s:appendixA}

In this section, we prove results from \Cref{s:preliminary} and \Cref{s:scaling} which are used in the rest of the paper. We begin with the proof of  \Cref{l:gammalimitlemma}, since facts derived in the course of that proof will be useful for proving the statements from  \Cref{s:preliminary}.
\eer

	%We now proceed to the proof of \Cref{l:gammalimitlemma}. 
	For any $w \in \mathbb{C}$, we let $\mathcal{H}_w$ denote the space of $\mathcal{C}^1$ functions $g\colon \mathbb{K}^+ \rightarrow \mathbb{C}$ such that $g (\lambda u) = \lambda^w g(u)$ for each $\lambda \in \mathbb{R}_{+}$. Following \cite[(10)]{bordenave2017delocalization}, we define for any $r \in [0, 1)$ a norm on $\mathcal{H}_w$ by 
	\begin{flalign}
	\| g \|_r = \| g \|_{\infty} + \sup_{u \in \mathbb{S}_+^1} \sqrt{\big| (\iu . u)^r \partial_1 g (u) \big|^2 + \big| ( \iu . u )^r \partial_2 g (u) \big|^2 },
	\end{flalign}
	
	\noindent where  $\partial_1 g (x + \iu y) = \partial_x g (x + \iu y)$ and $\partial_2 g (x + \iu y) = \partial_y g (x + \iu y)$, and we recall $\| g \|_{\infty} =  \sup_{u \in \mathbb{S}_+^1} \big| g(u) \big|$. We let $\mathcal H_{w, r}$ denote the closure of $\mathcal H_w$ in $\| \cdot \|_{r}$, which is a Banach space.
	
	Following \cite[(11), (12), (13)]{bordenave2017delocalization}, we define for any complex numbers $u \in \mathbb{S}_+^1$ and $h \in \overline{\mathbb{K}}$, and any function $g \in \mathcal{H}_{\alpha / 2}$, the function 	
	\begin{flalign}
	F_{h}(g) (u) = \displaystyle\int_0^{\pi / 2} \Bigg( \displaystyle\int_{\mathbb{R}_+^2} & \bigg(\Big( e^{ -r^{\alpha / 2} g (e^{\iu \theta}) - (rh . e^{\iu \theta})} - e^{ -r^{\alpha / 2} g (e^{\iu \theta} + uy) - (yrh . u) - (rh . e^{\iu \theta})}  \Big)  \\
	& \times r^{\alpha / 2 - 1} y^{- \alpha / 2 - 1} dr  dy   \bigg) \Bigg) (\sin 2 \theta)^{\alpha / 2 - 1} d \theta.
	\end{flalign}
	
	\noindent Further, for any $z \in \cplus$, the map $G_z(f)\colon \mathbb S^1_+ \rightarrow \mathbb C$ is given by
	\beq\label{e:Gdef}
	G_z(f)(u)  = \frac{\alpha}{2^{\alpha/2} \Gamma(\alpha/2)^2}F_{ - \iu z} (f) ( \iu \bar u).
	\eeq
	
	The following lemma from \cite{bordenave2017delocalization} indicates that the function $\gstar$ from \Cref{d:gstar} is a fixed point of $G_z$. 
	
	\bel[{\cite[Lemma 4.4]{bordenave2017delocalization}}]\label{l:gstarfixedpoint} For any $z \in \mathbb H$ and $u \in \mathbb{S}_+^1$, we have $\gstar(u) = G_z (\gstar)(u)$.
	\eel

\begin{proof}[Proof of \Cref{l:gammalimitlemma}]

For the first statement, we use \cite[Proposition 3.3]{bordenave2017delocalization}, which shows that there exists $c >0$ such that, uniformly in $|z| <c$, 
\beq\label{e:eipi4}
\frac{\gamma^\star_z ( e^{\iu \pi /4})}{\Gamma\left(1 - \frac{\alpha}{2} \right)}  = 2^{\alpha/4}  \E \left[ \big(\Im R_\star(z) \big)^{\alpha/2} \right] > c.
\eeq
We now compute, for any $u \in \mathbb S^1_+$,
\begin{flalign}
\frac{\Re \gamma^\star_z(u)}{\Gamma\left(1 - \frac{\alpha}{2} \right)} &= \E \left[  \Re \big( - \iu \rstar(z).u \big)^{\alpha/2} \right] \ge \E \left[  \Big( \Re  \big(- \iu \rstar(z).u  \big) \Big)^{\alpha/2} \right] \ge \E \left[ \big(\Im R_\star(z) \big)^{\alpha/2} \right] \ge c.
\end{flalign}
In the first inequality, we used the fact that $\Re a^r \ge (\Re a)^r$ for any $a \in \mathbb K$ and $r \in (0,1)$ (see \cite[Lemma 5.10]{bordenave2017delocalization}). The second inequality follows from $\Re (a. u) \ge \Re a$ for any $u \in \mathbb S^+_1$ and $a\in \mathbb K^+$. The final inequality follows from \eqref{e:eipi4}. This completes the proof of the first claim.

Set $z = E + \iu \eta$.  We now establish convergence of the order parameter $\gstar$ as $\eta\rightarrow 0$. 
We note that for any $\tau >0$, there exists $c = c(\tau ) >0$ such that 
\beq\label{e:input}
\| \gamma_z^\star - \gamma_w^\star \|_r \le  \| \gamma_z^\star - \gamma_0^\star \|_r +  \| \gamma_w^\star - \gamma_0^\star \|_r \le \tau
\eeq
if $z,w \in \mathbb H$ satisfy $|z| < c$ and $|w| <c$. The final inequality follows from the first displayed equation in the proof of \cite[Proposition 3.3]{bordenave2017delocalization}.

Then \eqref{e:input} and \cite[Proposition 3.4]{bordenave2017delocalization} together imply there exist constants $C, c>0$ such that
\beq\label{e:lipschitzcombining}
\| \gamma_w^\star  - \gstar \|_r \le C \| \gamma_w^\star - G_z (\gamma_w^\star) \|_r =  C \big \| G_w(\gamma_w^\star)  - G_z(\gamma_w^\star)\big \|_r
\eeq
for $z,w \in \mathbb H$ such that $|z| < c$ and $|w| <c$. In the equality, we used that $\gamma_w^\star$ is a fixed point for $G_w$, as stated in \Cref{l:gstarfixedpoint}.

We now claim 
\beq\label{e:gwgstar}
 \big\| G_w(\gamma_w^\star)  - G_z(\gamma_w^\star)\big \|_r \le C | w - z|.
\eeq
In the proof of \cite[Lemma 4.2]{bordenave2017delocalization} (in the final lines), it was shown that the partial (Fr\'echet) derivative of $F_h(g)$ in either the real or imaginary part of $h$ has finite $\| \cdot \|_r$ norm, and the exact derivative was calculated. Further, the derivative may be bounded in the $\| \cdot \|_r$ norm by a constant $C$ using \cite[(20)]{bordenave2017delocalization} when $g = \gamma^\star_w$, which is uniform in $z,w$ with $z,w \in \mathbb {H}$ and $|z|, |w| \le c$.\footnote{We remark that although the constant $C$ in the bound \cite[(20)]{bordenave2017delocalization} depends on $\Re g$ and degenerates as $\Re g$ goes to zero, here $g = \gamma_w^\star$ and $\inf_{u \in \mathbb S^+} \Re g(u) > c$ for some $c >0$, from \eqref{e:starclaim2}. So we obtain the claimed bound.} Also, it is continuous by the computation following \cite[(21)]{bordenave2017delocalization}. By definition \eqref{e:Gdef}, the same is true for $G_h(g)$, and we obtain \eqref{e:gwgstar} by integration using the fundamental theorem of calculus. Combining this with \eqref{e:lipschitzcombining} we obtain the Lipschitz estimate 
\beq\label{e:gE}
\| \gamma^{\star}_w  - \gstar \|_r \le C | w -z |
\eeq
for any $r \in [0, 1)$. This estimate implies that $\lim_{\eta \rightarrow 0} \gamma_{E + \iu \eta}^\star$ exists as a function in $\mathcal H_{\alpha/2, r}$, which we denote by $\gamma^\star_E$. 
\end{proof}

\begin{proof}[Proof of \Cref{l:mabounds}]

We begin with the first estimate of \eqref{e:rhoacontinuity}. By \cite[(3.4)]{bordenave2011spectrum} and \cite[Theorem 4.1]{bordenave2011spectrum}, we have $\Im m_\alpha(z) = \E \big[\Im R_\star(z)\big]$ for ${z \in \mathbb H}$. Recalling  $\lim_{\eta \rightarrow 0} \Im m_\alpha(E + \iu \eta ) =  \pi \varrho_\alpha(E)$, it then suffices to show that ${\lim_{\eta \rightarrow 0}  \E \big[\Im R_\star(E + \iu \eta )\big]}$ is Lipschitz in $E$. For $c$ small enough, by \eqref{rlimitp}, this limit is given by $(\mathfrak{X} + \overline{\mathfrak{X}})/2$, where
\beq
\mathfrak{X}(E)=  \displaystyle\int_{\mathbb{R}_+}  \exp \big( \iu E t - t^{\alpha / 2} \gamma^\star_E (1) \big) \, dt.
\eeq
Further, if $c$ is small enough, by \Cref{l:gammalimitlemma} we have the uniform lower bound 
\beq \inf_{|E| \le c} \inf_{u \in \mathbb S^1_+} \Re \gamma^\star_E(u) > c'.\eeq Define

\beq
F(x,y,w)= \displaystyle\int_{\mathbb{R}_+}  \exp \big( \iu x t - t^{\alpha / 2} (y + \iu w) \big) \, dt.
\eeq
By \eqref{e:gE}, $\big| \gamma^\star_{E_1}(1) - \gamma^\star_{E_2}(1) \big| \le C | E_1 - E_2|$ for some constant $C$ and $E_1, E_2 \in [ -c, c]$, if $c$ is small enough.  Using this inequality, to show $\mathfrak{X}(E)$ is Lipschitz in $E$, it suffices by the fundamental theorem of calculus to show the partial derivatives $\partial_x F(x,y,w)$, $\partial_y F(x,y,w)$, and $\partial_w F(x,y,w)$ are uniformly bounded by a constant when $|x| \le c$ and $ y > c'$. This follows straightforwardly after differentiating under the integral sign (which is permissible as the integrand is dominated in absolute value by $\exp( - t^{\alpha/2} y ) \le \exp( - t^{\alpha/2} c' )$).

We have shown that the density $\varrho_\alpha(x)$ is continuous in a neighborhood of zero. By \cite[Theorem 1.6(ii)]{bordenave2011spectrum}, $\varrho_\alpha(0)$ is positive and bounded. Therefore the second claim in \eqref{e:rhoacontinuity} follows from the first after possibly decreasing $c$.

For \eqref{e:macontinuity}, we first suppose $|E_1 - E_2 | \le c/10$ and $|E_1| \le c/2$, where $c$ is the constant from the previous part of this proof. We use the definition of the Stieltjes transform to write

\begin{align}
\Im m_\alpha\big( E_1 + \iu \eta \big) =   \Im \int_{\mathbb R} \frac{\varrho_\alpha(x) \, dx}{ x -  E_1 - \iu \eta }, \quad \Im m_\alpha\big( E_2  + \iu \eta \big) =   \Im \int_{\mathbb R} \frac{\varrho_\alpha\big(x +  E_2 - E_1 \big) \, dx}{ x - E_1 - \iu \eta },
\end{align}
where in the second equality we used the change of variables $x \mapsto x + E_2 - E_1$. Computing the imaginary part of the integrand directly, we have
\begin{align}
\left| \Im m_\alpha\big( E_1 + \iu \eta \big) - \Im m_\alpha\big( E_2  + \iu \eta \big) \right| &\le    \eta  \int_{\mathbb R} \frac{\big| \varrho_\alpha(x)  - \varrho_\alpha (x +  E_2  - E_1) \big| }{ (x -   E_1)^2  + \eta^2 }\, dx\\
&\label{e:term1} =  \eta  \int_{[-c/2,c/2 ]} \frac{\big| \varrho_\alpha(x)  - \varrho_\alpha (x +  E_2  - E_1) \big| }{ (x -   E_1 )^2  + \eta^2 }\, dx\\
&\label{e:term2} +  \eta  \int_{[-c/2,c/2 ]^c} \frac{\big| \varrho_\alpha(x)  - \varrho_\alpha \big(x +  E_2  - E_1) \big| }{ (x -   E_1)^2  + \eta^2 }\, dx.
\end{align}
Using \eqref{e:rhoacontinuity}, $|x| \le c/2$, and $|E_1 - E_2| \le c/10$, we find that 
\beq
\eqref{e:term1} \le C | E_1 - E_2| \int_{\mathbb R} \frac{\eta}{(x -   E_1)^2  + \eta^2}\, dx \le C | E_1 - E_2|.
\eeq
By \cite[Theorem 1.6]{bordenave2011spectrum}, the density $\varrho_\alpha(x)$ is uniformly bounded. Therefore 
\begin{align}
\eqref{e:term2} \le  C \int_{[-c/2,c/2 ]^c} \frac{\eta }{ (x -   E_1 )^2  + \eta^2 }\, dx &= C \eta  \int_{[-c/2,c/2 ]^c} \frac{1 }{ (x -   E_1)^2  }\, dx \\
& \le C c^{-1} \eta.
\end{align}
In the last line, we used the hypothesis that $|E| \le c/2$ to show the integral is uniformly bounded and adjusted the value of $C$.
This completes the proof of \eqref{e:macontinuity} after decreasing $c$ if necessary.

We now address the first claim of \eqref{e:mabounds}. Again using the uniform boundedness of $\varrho_\alpha(x)$ over all $x \in \mathbb R$, we deduce 
\beq 
\Im m_\alpha(E + \iu \eta ) = \int_{\mathbb R} \frac{\eta \varrho_\alpha(x)\, dx }{\eta^2 + (E-x)^2 } \le C.
\eeq
We also note, again using \eqref{e:rhoacontinuity}, that 
\beq
\Im m_\alpha(E + \iu \eta ) \ge  \int_{[ - c, c ] } \frac{\eta \varrho_\alpha(x)\, dx }{\eta^2 + (E-x)^2 } \ge c \int_{[ -c, c ] } \frac{\eta\, dx }{\eta^2 + (E-x)^2 },
\eeq
and the latter quantity is uniformly bounded below $\eta$ tends to zero because $E \in (-c,c)$. Thus, after decreasing $c$ if necessary, we have $\Im m_\alpha(E + \iu \eta ) > c$ when $\eta, |E| < c$. 

The second claim of \eqref{e:mabounds} follows after noting (using the bounds on the density $\varrho_\alpha$ given in \eqref{e:rhoacontinuity}) that for any $c$, there exists $c' >0$ such that $\big|\gamma_i^{(\alpha)}\big| < c$ for all $i\in [ ( 1/2 - c') N , (1/2 + c') N]$. 
\end{proof}

\begin{proof}[Proof of \Cref{l:hgammarigid}]
We prove only the first claim in detail. The proof of the second is analogous, and the third follows from the second by \eqref{e:xtrigidity}. 

By a standard stochastic continuity argument, it suffices to prove the desired bound holds with overwhelming probability at fixed $\gamma$; all bounds below will be independent of $\gamma$. Let $C>0$ be a parameter. A straightforward calculation shows that for any $1 \le i \le N$, the law of the sum of the absolute value of the entries of the $i$-th row and column of the matrix $\bH^\gamma$ has a power law tail with parameter $\alpha$. This implies, by Hoeffding's inequality, that with overwhelming probability there are at most $2C^{-\alpha}N$ such $i$ whose corresponding sum is greater than $C$. When this holds, after removing at most $2C^{-\alpha}N$ rows and columns, the largest eigenvalue is at most $C$ in absolute value, since the largest absolute value of a row of a matrix bounds the magnitude of its largest eigenvalue. Then eigenvalue interlacing \cite[Lemma 7.4]{erdos2017dynamical} implies that there are at most $4C^{-\alpha}N$ eigenvalues of $\bH^\gamma$ outside of the interval $[-C,C]$.

By \cite[Theorem 1.1]{arous2008spectrum} (or \cite[Theorem 1.2]{bordenave2011spectrum}),
for any fixed compact interval $I \subset \mathbb R$, 
\beq\label{e:intervallimit}
 \E \left[{\mu_N(I)} \right] \rightarrow {\mu_\alpha(I)} 
\eeq
 as $N$ tends to $\infty$. Here $\mu_N = \mu_N^{(\gamma)}$ denotes the empirical spectral distribution of $\bH^\gamma$.

For any $C > 0$, let $I_C =  [c_1/2, C]$. Then \eqref{e:intervallimit} and the concentration estimate \cite[Lemma C.1]{bordenave2013localization} imply that for any choice of $C$ there exists $N(C)$ such that
\beq\label{e:icestimate}
\mu_N(I_C) \le C^{-1}  +  \mu_\alpha(I_C)
\eeq 

\noindent holds for any $N > N(C)$ with overwhelming probability. By the symmetry of $\mu_{\alpha}$ and the second estimate in \eqref{e:rhoacontinuity}, we have $\mu_\alpha(I_C) \le (1/2 - \delta)$ for some $\delta = \delta(c_1) > 0$ such that $\lim_{c_1 \rightarrow 0} \delta (c_1) = 0$. Combining \eqref{e:icestimate} with the estimate for eigenvalues lying outside $[-C,C]$, we find 
\beq
\mu_N\big([ c_1, \infty)\big) \le C^{-1}  +  (1/2 - \delta) + 4 C^{-1/\alpha} < 1/2 - \delta/2
\eeq
for large enough $C$, with overwhelming probability. Then the $(1/2 - \delta/2)N$-th eigenvalue is less than $c_1$ with overwhelming probability. A similar argument shows that the $(1/2 + \delta/2)N$-th eigenvalue is greater than $c_1$ with overwhelming probability. This completes the proof. \end{proof}

Before proceeding to the proof of \Cref{l:newrigidity}, we require the following preliminary lemma. We recall $m_N(s,z)$ and its expectation $ \widehat m_N (s,z)= \E\big[\widehat m_N(s,z)\big]$ from \eqref{rsz}, and $\mu_s$ from \eqref{xsempiricalmeasure}. Let $\widehat \mu_s = \mathbb{E} [\mu_s]$, which is symmetric about the origin. We further define the counting functions
\beq
n_s(E) = \frac{1}{N} \Big| \big\{i :  \lambda_i(s) \le E  \big\}  \Big| = \mu_s \big (  (-\infty, E]  \big), \quad \widehat n _s (E)  = \widehat \mu_s \big (  (-\infty, E]  \big).
\eeq

\bel Retain the notation of \Cref{l:newrigidity}. There exists a constant $c_0 > 0$ such that, for any $E \in [ - c_0, c_0]$, we have with overwhelming probability that
\beq\label{concentrationcounting}
\sup_{s \in [ N^{-1/2 + \delta}, N^{-\delta} ] }  \big| n_s(E ) - \widehat n_s(E )  \big| \le  N^{\delta- 1/2}. 
\eeq 
\eel
\begin{proof}

This proof will largely follow the calculations of \cite[Section 7.3]{landon2017convergence}, with some modifications to account for the fact that the spectral distributions we consider are not compactly supported.

To that end, we begin with a tail bound on the smallest eigenvalue, $\lambda_1 = \lambda_1 (s)$, of $\mathbf X_s$. A straightforward calculation shows that, for any $1 \le i \le N$, the law of the sum of the absolute value of the entries of the $i$-th row and column of the matrix $\textbf{X}_s$ has a power law tail with parameter $\alpha$. Thus, since the largest such sum bounds $|\lambda_1|$ above, we deduce for any $t > 1$ that
\beq\label{e:smallestev}
\P \big(   \lambda_1 < -t  \big) \le CN t^{- \alpha}.
\eeq

Now, we must show that \eqref{concentrationcounting} holds on an event of probability at least $1 - N^{-D}$ for any fixed $D >0$ and sufficiently large $N$. Throughout the remainder of this proof, set $B = \alpha^{-1} (D + 3)$ so, by \eqref{e:smallestev}, $\mathbb{P} ( \lambda_1 > -N^B) \ge 1  - N^{- D - 2}$. Thus, we may work on the event on which $\lambda_1 > -N^B$. 

By the Helffer--Sj\"{o}strand formula (see, for example, \cite[Chapter 11]{erdos2017dynamical}), for any smooth and compactly supported function $f\colon \mathbb R \rightarrow \mathbb R$, 
\beq\label{e:hs}
f(u ) = \frac{1}{\pi} \int_{\mathbb R^2}\frac{ \iu y f''(x) g(y) + \iu \big( f(x) + \iu y f'(x)  \big) g'(y) }{u - x - \iu y }\, dx\,dy,
\eeq
where $g$ is any smooth, compactly supported function that is $1$ in a neighborhood of $0$. Set $E_1 = - N^{4B}$ and fix some $E_2 \in [ -(2K)^{-1}, (2K)^{-1}]$, where $K$ is the constant from \Cref{l:thatlaw}. Let $\eta = N^{-1/2 + \delta}$, and let $f$ be a smooth function satisfying $f(E) = 0$ for $E \notin [ E_1 - 1,  E_2 + \eta]$ and $f(E)=1$ for $E \in [E_1 , E_2]$. 
We can select $f$ such that $\big| f(x) \big|\le 1$ for all $x \in \mathbb{R}$; $\big| f'(x) \big|\le C$ and  $\big| f''(x) \big| \le C$ for $x \in [ E_1 - 1, E_1]$; and $\big| f'(x) \big|\le C\eta^{-1}$ and $\big| f''(x) \big| \le C\eta^{-2}$ for $x \in [E_2, E_2 + \eta]$. We also let $g(y)$ be a smooth function satisfying $g(y) = 1$ for $|y| \le N^{10 B}$; $g(y) = 0$ for $|y| > N^{10 B} + 1$; we may select $g$ such that $0 \le  g(y) < 1$ and $\big| g' (y) \big| < C$ for all $y \in \mathbb R $.

Write $\mu_\Delta = \mu_s - \widehat \mu_s$ and let $m_\Delta(z) =   m_N(s,z) - \widehat m_N (s,z)$ be the Stieltjes transform of $\mu_\Delta$. Our first goal is to prove that 
\beq\label{e:Deltagoal}
\left| \int_{\mathbb R} f(E)\,d\mu_\Delta(E) \right| \le C N^{\delta/2 - 1/2}
\eeq
with probability at least $1 - N^{-D - 1}$, for large enough $N$. Using\eqref{e:hs}, we find 
\begin{align}
\left| \int_{\mathbb R} f(E)\,d\mu_\Delta(E) \right| 	& \le C  \left| \int_{\mathbb R^2} y f''(x) g(y)    \Im m_\Delta(x + \iu y) \, dx \,dy \right| \label{aterm1}  \\ 
&  \label{aterm2} \qquad +  C   \int_{\mathbb R^2} \big| f(x) g'(y) \big|   \big| \Im m_\Delta(x + \iu y)\big|  \, dx \,dy  \\
&\label{aterm3} \qquad +  C   \int_{\mathbb R^2} \big| y f'(x) g'(y) \big|   \big| \Re m_\Delta(x + \iu y)\big|  \, dx \,dy .
\end{align}

\noindent Now, since \cite[(5.13)]{aggarwal2018goe} states 
\beq\label{e:mconcentrate}
\P \left[ \big| m_N(s,z)  - \widehat m_N (s,z)   \big|  > \frac{4\log N }{N^{1/2} \Im z} \right] \le 2 \exp \big( - ( \log N)^2 \big),  
\eeq

\noindent we have (by a standard stochastic continuity argument) with overwhelming probability that 
\beq\label{deltaconcentrate}
\displaystyle\sup_{|y| \le N^{20B}} \big| y m_\Delta(x + \iu y ) \big| \le 5 N^{-1/2} \log N.
 \eeq
 
 \noindent Now let us bound the quantities \eqref{aterm1}, \eqref{aterm2}, and \eqref{aterm3}. We begin with the latter. To that end, observe that since $\supp f' \subseteq [E_1 - 1, E_1] \cup [E_2, E_2 + \eta]$; since $\supp g' \subseteq [-N^{10B} - 1, -N^{10B}] \cup [N^{10B}, N^{10B} + 1]$; since $\big| f' (x) \big|\le C$ for $x \in [E_1 - 1, E_1]$; since $\big| f'(x) \big| \le C \eta^{-1}$ for $x \in [E_2, E_2 + \eta]$; and since $\big| g'(y) \big| \le C$, we have by \eqref{deltaconcentrate} that $\eqref{aterm3} \le C (\log N) N^{-1/2}$ with overwhelming probability. 
 
 Similarly, to bound \eqref{aterm2}, observe that $\big| f(x) \big|\le 1$; that $\big| g'(x) \big| \le C$; that $\supp f$ is contained in the interval $[-N^{4B}, N^{4B}]$ of length at most $2 N^{4B}$; and that $\supp g' \subset [- N^{10B} - 1, - N^{10B}] \cup [N^{10B}, N^{10B} + 1]$, on which we have $\big| m_{\Delta} (z) \big| \le N^{-10B}$ (due to the deterministic bound $\big|m_{\Delta} (z)\big| < |\Im z|^{-1}$). Together, these yield the deterministic estimate  $\eqref{aterm2} \le N^{-1}$. 
 
 It therefore suffices to bound \eqref{aterm1}, to which end we write
 \begin{align}
 \label{b1}  \eqref{aterm1} &  \le   \left| \int_{|x-E_1| \le 2, |y| \le 10 } y f''(x) g(y)    \Im m_\Delta(x + \iu y) \, dx \,dy \right| \\ 
 \label{b2} & \qquad +  \left| \int_{|x-E_1| \le 2, |y| > 10} y f''(x) g(y)    \Im m_\Delta(x + \iu y) \, dx \,dy \right| \\ 
 \label{b3} & \qquad+  \left| \int_{|x-E_2| \le 2 \eta, |y| \le \eta} y f''(x) g(y)   \Im m_\Delta(x + \iu y) \, dx \,dy \right| \\ 
\label{b4} &  \qquad+  \left| \int_{|x-E_2| \le 2 \eta, |y| > \eta} y f''(x) g(y)  \Im m_\Delta(x + \iu y) \, dx \,dy \right|.
 \end{align}
 
 \noindent We must bound the terms \eqref{b1}, \eqref{b2}, \eqref{b3}, and \eqref{b4}; we begin with the former. Since we restricted to the event of probability $1 - N^{-D - 2}$ on which $\lambda_1 > -N^B$, and since $E_1 = N^{4B}$, the definition of $m_N(s,z)$ shows that $\big|m_N(s,x + \iu y) \big| \le C N^{- 4B}$ for $x \in [E_1 - 2, E_1 + 2]$. Using the trivial bound $\eqref{kij}$ on the complementary event and then taking expectation shows that $|\widehat m_N (s,x + \iu y)| \le C N^{-4B} + N^{-D - 2} y^{-1}$ for $x \in [E_1 - 1, E_1 + 1]$, if $D$ is sufficiently large. Hence, with probability $1 - N^{-D - 2}$, we have $\big| m_\Delta(x + \iu y) \big| \le N^{-4B} + N^{-D} y^{-1}$ for $x \in [E_1 - 2, E_1 + 2]$. Combining this with the fact that $\big| f'' (x) \big|$ for $x \in [E_1 - 1, E_1 + 1]$, we obtain $\eqref{b1} \le N^{-1}$.
 
 To estimate \eqref{b2}, we first integrate by parts in $x$, using the identity $\partial_x \Im m_\Delta = - \partial_y \Re m_\Delta$ and the fact that $\supp f' \subseteq [E_1 - 1, E_1 + 1]$ to deduce
 \begin{flalign} 
 \eqref{b2} = \left|  \int_{|x - E_1 | \le 2, |y| > 10} y f'(x) g(y) \partial_y\big( \Re m_\Delta(x + \mathrm{i} y) \big) \, dx \, dy\right|.
\end{flalign}
 
 \noindent Integrating by parts in $y$ and using the fact that $\partial_y \big( y g(y) \big) = g(y) + y g' (y)$ then gives 
\begin{align}
\label{c1}  \eqref{b2} & \le \left|  \int_{|x - E_1 | \le 2, |y| > 10 }  f'(x) \big(g(y) + y g'(y) \big) \Re m_\Delta(x + \mathrm{i} y) \, dx \, dy\right| \\
& \label{c2} \qquad + \left|  \int _{|x-E_1| \le 2 }  f'(x) 10 g(10) \Re m_\Delta(x  + 10 \iu) \,dx \right| .
\end{align} 

\noindent To bound \eqref{c2}, we use \eqref{deltaconcentrate} and the fact that $\big| f' (x) \big| \le C$ for $x \in [E_1 - 2, E_1 + 2]$ to deduce that $\eqref{c2} \le C (\log N )N^{-1/2}$ with overwhelming probability. To estimate \eqref{c1}, we again use the facts that $\big| f' (x) \big| \le C$ for $x \in [E_1 - 2, E_1 + 2]$; that $\supp g \subseteq [-N^{10B} - 1, N^{10B} + 1]$; that $0 \le g(y) \le 1$; that $\supp g' \subseteq [N^{10B}, N^{10B} + 1]$; $g' (y) \le C$; and \eqref{deltaconcentrate} to deduce 
\begin{align}
\eqref{c1} & \le C N^{-1/2} \log N \left|  \int_{|y| > 10 } |y|^{-1} \big(g(y) + y g'(y) \big) \, dy\right| \\
& \le C N^{-1/2} \log N \Bigg( \bigg| \int_{10}^{N^{11B}} |y|^{-1} dy \bigg| + \bigg| \int_{|y| > 10} g'(y) \, dy \bigg| \Bigg) \le C N^{-1/2} (\log N)^3 \label{estimatesumintegral},
\end{align} 

\noindent with overwhelming probability and for sufficiently large $N$. Hence, $\eqref{b2} \le C N^{-1/2} (\log N)^3$.

To bound \eqref{b3}, first recall that the function $y \Im m ( x + \iu y)$ is increasing in $y$, for any Stieltjes transform $m$ of a positive measure. Therefore, \eqref{deltaconcentrate} implies with overwhelming probability that 
\beq
\displaystyle\sup_{y \le \eta} y \Im m_\Delta ( x + \iu y ) \le \eta  \Im m_\Delta(x + \iu \eta) \le C N^{-1/2} \log N
\eeq

\noindent Putting this estimate into \eqref{b3} and using that $\big| f''(x) \big|$ vanishes except on $[E_2, E_2 + \eta]$, where it is at most $C\eta^{-2}$, and that $\big| g(y) \big|$ is $1$ for $|y| \le \eta$, we deduce that $\eqref{b3} \le  C (  \log N  ) N^{-1/2}$.  

For the term \eqref{b4}, we integrate by parts as we did for \eqref{b2} to obtain 
\begin{align}
\label{d11} \eqref{b4} & \le \left|  \int_{|x - E_2 | \le 2 \eta, |y| > \eta}  f'(x) \partial_y\big( g(y) + y g' (y) \big) \Re m_\Delta(x + iy) \, dx \, dy\right| \\
 \label{d21} & \qquad + \left|  \int _{|x-E_2| \le 2 \eta}  f'(x) \eta g(\eta) \Re m_\Delta(x  +  \iu \eta ) \,dx \right| .
\end{align} 

\noindent We use \eqref{deltaconcentrate} to estimate $\eqref{d11} \le C (\log N)^3 N^{-1/2}$ and $\eqref{d21} \le C (\log N) N^{-1/2}$, in the same way we bounded \eqref{c1} (in \eqref{estimatesumintegral}) and \eqref{c2}, except now we note that $f'(x)$ vanishes off of $x \in [E_2 , E_2 + \eta]$, where it is at most $C\eta^{-1}$. This shows $\eqref{b4} \le C (\log N)^3 N^{-1/2}$, and so $\eqref{aterm1} \le C (\log N)^3 N^{-1/2}$, with overwhelming probability. Combining our estimates on \eqref{aterm1}, \eqref{aterm2}, and \eqref{aterm3}, we deduce \eqref{e:Deltagoal} holds with probability at least $1 - N^{-D - 1}$. 

We now use \eqref{e:Deltagoal} to estimate the difference between the eigenvalue counting functions for the measures $\mu_s$ and $\widehat \mu_s$. 
We recall that we are working on the set of probability at least $1 - N^{-D - 2}$ where $\lambda_1  >  - N^{B}$. Therefore, recalling the definition of $f(x)$, we see $n_s(E_2) \le \int_{-\infty}^{\infty} f(x) \, d\mu_s(x) \le n_s (E_2 + \eta)$ for any $E_2 \in [ -(2K)^{-1}, (2K)^{-1}]$ on this event. The case of $\widehat n _s (E)$ is slightly more delicate, since we must estimate the contribution to the mass of $\widehat \mu_s $ from eigenvalues in the interval $( - \infty, - N^{B}]$. With probability at least $1 - N^{-D - 2}$, there are no eigenvalues in this interval. On the complementary event,  trivially have $\int_{-\infty}^{ - N^{B} } d\mu_s(x)  \le 1$, since $\mu_s$ is a probability measure. Therefore, $\int_{-\infty}^{ - N^{B} }d  \widehat \mu_s(x) \le N^{- 2}$ by taking expectation and using the previous two observations.

Using \eqref{e:Deltagoal} and the overwhelming probability estimate $\big | n_s(E_2 + \eta)  - n_s(E_2 )  \big| \le C \eta$ (which follows from the second inequality of \eqref{e:disperse}), we deduce that, with probability at least $1 - C N^{-D - 1}$,
\begin{align}
\widehat n_s(E_2) - n_s (E_2) \le \widehat{n}_s (E_2) - n_s (E_2 + \eta) + C \eta & \le \displaystyle\int_{\mathbb{R}} f(x) d \widehat{\mu}_s (x) + C N^{-2} - \displaystyle\int_{\mathbb{R}} f(x) d \mu_s (x) + C \eta \\
& \le C N^{\delta / 2 - 1/2} + C \eta \le C N^{\delta / 2 - 1/2}.
\end{align} 

\noindent Similarly, now using the bound $\big| \widehat n_s(E_2 + \eta)  - \widehat n_s(E_2 )  \big| \le C \eta$ (which follows from the overwhelming probability estimate $\big | n_s(E_2 + \eta)  - n_s(E_2 )  \big| \le C \eta$ after taking expectation and using the trivial bound  $\big | n_s(E_2 + \eta)  - n_s(E_2 )  \big| \le 1$ on the set where this estimate does not hold), we obtain
\begin{align}
n_s (E_2) - \widehat n_s(E_2) &  \le n_s (E_2 + \eta) - \widehat{n}_s (E_2 + \eta) + C \eta \\
 & \le \displaystyle\int_{\mathbb{R}} f(x) d \mu_s (x) + C N^{-2} - \displaystyle\int_{\mathbb{R}} f(x) d \widehat{\mu}_s (x) + C \eta \\
& \le C N^{\delta / 2 - 1/2} + C \eta \le C N^{\delta / 2 - 1/2},
\end{align} 

\noindent with probability at least $1 - C N^{-D - 1}$. Hence, with probability at least $1 - N^{-D}$, we conclude $\big| n_s (E) - \widehat n_s (E)  \big| \le N^{\delta - 1/2}$, for $|E| \le (2K)^{-1}$. 
\end{proof}

We are  now ready for the proof of \Cref{l:newrigidity}. 

\begin{proof}[Proof of \Cref{l:newrigidity}]
We start with the first claim.
%% Proof of weak rigidity with respect to \gamma_i^{(\alpha)} 
By \eqref{rijestimate} and \eqref{mszestimate}, we have with overwhelming probability that 
\begin{equation}\label{fact1}
\sup_{s\in [ N^{-1/2+\delta} , N^{-\delta} ] } \sup_{z \in \widetilde{ \mathcal{D}} } \big| m_\alpha(z)  - m_N(s,z) \big |  < N^{-\alpha\delta/16}.
\end{equation}
By using \eqref{kij} on the set where this does not hold and taking expectation, we deduce the deterministic estimate
\begin{equation}\label{fact2}
\sup_{s\in [ N^{-1/2+\delta} , N^{-\delta} ] } \sup_{z \in \widetilde{ \mathcal{D}} } \big| m_\alpha(z)  - \widehat m_N (s,z) \big |  < CN^{-\alpha\delta/16}. 
\end{equation}

\noindent Next, we define 
\beq
n_\alpha(E) = \int_{-\infty}^E \varrho_\alpha (x) \, dx.
\eeq

The distribution functions $n_{\alpha}(E)  - n_{\alpha}(0)= \int_0^E \varrho_\alpha(x)\, dx$ and  $\widehat n_s(E)  - \widehat n_s(0) = \int_0^E d\widehat \mu_s(x)$ may be compared using \cite[(11.3)]{erdos2017dynamical} (see also the proof of \cite[Lemma 11.3]{erdos2017dynamical}), which by \eqref{fact2} gives 
\beq\label{rigidgoal1}
\big|  n_\alpha(E) - \widehat n_s(E)  - n_\alpha(0) + \widehat n_s(0) \big|  \le C N^{-c}
\eeq

\noindent for some $c = c(\delta) > 0$ and $E \in [ -c_0, c_0]$, where $c_0 > 0$ is sufficiently small.

There are two cases: $N$ is even and $N$ is odd. When $N$ is even, by the symmetry of the measures $\widehat \mu_s$ and $\mu_\alpha$, we see that $\widehat \gamma_{N/2}(s) = \gamma^{(\alpha)}_{N/2} = 0$. We consider this case first.
 
 We will show that for any $c_1 >0$ sufficiently small, there exists $c_2 >0$ such that, if $| i - N/2 | \le c_2 N$, then $\big| \widehat \gamma_i(s) \big | < c_1$. To that end, observe that \eqref{e:xtdispersion} implies, with overwhelming probability, 
 \beq\label{takeexpectation}
 \int_0^v  d \mu_s  \ge  cv,
 \eeq

\noindent for any $v \in [N^{\delta - 1}, c_0]$ (after decreasing $c_0$, if necessary). Taking expectation and using that $\mu_s$ is a nonnegative measure, we find
 \beq\label{geq}
 \int_0^v d \widehat \mu_s \ge  \frac{c v}{2}, \qquad \text{for any $v \in [N^{\delta - 1}, c_0]$.}
 \eeq

\noindent We may suppose by symmetry that $i \ge N/2$, so that $\widehat \gamma_i(s) \ge 0$ by the symmetry of $\widehat\mu_s(x)$.  By definition, 
 \beq\label{shownabove}
 \int_0^{\widehat \gamma_i(s) }  d\widehat \mu_s = \frac{i - N/2}{N} \le c_2.
 \eeq
If we choose $c_2 < c_1c/4$, then $\widehat \gamma_i(s) > c_1$ produces a contradiction with \eqref{geq}. By \eqref{e:rhoacontinuity}, we also have  that for any $c_1 >0$ sufficiently small, there exists $c_2 >0$ such that, if $| i - N/2 | \le c_2 N$, then $\big| \gamma_i \big | < c_1$. We take $c_2$ small enough so that $| i - N/2 | \le c_2 N$ implies $\gamma_i, \widehat \gamma_i(s) \in [-c_0, c_0]$, and consider just this set of indices in what follows. 
 
 We observe that, for $s > 0$, $\widehat \mu_s$ is absolutely continuous with respect to Lebesgue measure, since each entry of $\mathbf X_s$ is. Then, by the definition of $\gamma_i$ and $\widehat \gamma_i(s)$, 
\beq
\label{evencasedef}
 \int_0^{\widehat \gamma_i(s) } \varrho_\alpha(x)\, dx +  \int_{\widehat \gamma_i(s) } ^{\gamma_i} \varrho_\alpha(x)\, dx = \int_0^{\gamma_i } \varrho_{\alpha} (x) dx  =  \frac{i - N/2}{N} = \int_0^{\widehat \gamma_i(s) } d \widehat \mu_s.
\eeq

Using \eqref{rigidgoal1} and the fact that $\varrho_\alpha(x)$ is bounded below on $[-K, K]$ by some constant $c'$ by \eqref{e:rhoacontinuity}, we see 
\beq
\label{gammaigammaisestimate}
N^{-c } \ge  \left| \int_{\gamma_i(s) } ^{\gamma_i} \varrho_\alpha(x)\, dx \right| \ge c' \big| \gamma_i  - \widehat \gamma_i(s) \big|
\eeq
deterministically, which completes the proof when $N$ is even.

When $N$ is odd, we have $\widehat \gamma_{\lfloor N/2 \rfloor}(s) =-  \widehat \gamma_{\lceil N/2 \rceil}(s)$ by symmetry. If $\widehat \gamma_{\lceil N/2 \rceil}(s) > N^{-1 + \delta}$, then setting $v = \big| \widehat \gamma_{\lceil N/2 \rceil}(s) \big |$ in \eqref{geq} yields and also using the fact that 
\beq\label{ge2}
 N^{-1} = \int_{\widehat \gamma_{\lfloor N/2 \rfloor}(s)}^{\widehat \gamma_{\lceil N/2 \rceil}(s)} d\widehat\mu_s = 2 \int_0^{\widehat \gamma_{\lceil N/2 \rceil}(s)} d\widehat\mu_s \ge \displaystyle\frac{c N^{\delta - 1}}{2},
\eeq

\noindent which is a contradiction. Thus, $\big| \widehat \gamma_{\lfloor N/2 \rfloor}(s) \big | = \big|  \widehat \gamma_{\lceil N/2 \rceil}(s) \big| \le C N^{-1 + \delta}$. We write
 \beq
 \int_{\widehat \gamma_{\lceil N/2 \rceil}(s)}^{\widehat \gamma_i(s)} d \widehat \mu_s =  \int_{ \gamma_{\lceil N/2 \rceil}(s)}^{\widehat \gamma_i(s) } \varrho_\alpha(x)\, dx +  \int_{\widehat \gamma_i(s) } ^{\gamma_i} \varrho_\alpha(x)\, dx.
 \eeq
 Since $\big|  \widehat \gamma_{\lceil N/2 \rceil}(s) \big| \le C
  N^{-1 + \delta}$, and by \eqref{e:rhoacontinuity}, $ \big|\gamma_{\lceil N/2 \rceil}(s) \big| \le C
  N^{-1 + \delta}$, we have
  \begin{equation}
  \big|  \widehat \gamma_{\lceil N/2 \rceil}(s)  - \gamma_{\lceil N/2 \rceil}(s) \big| \le C
  N^{-1 + \delta}.
  \end{equation}
In conjunction with \eqref{e:rhoacontinuity}, this shows 
  \beq
 \int_{\widehat \gamma_{\lceil N/2 \rceil}(s)}^{\widehat \gamma_i(s)} d \widehat \mu_s =  \int_{\widehat \gamma_{\lceil N/2 \rceil}(s)}^{\widehat \gamma_i(s) } \varrho_\alpha(x)\, dx +  \int_{\widehat \gamma_i(s) } ^{\gamma_i} \varrho_\alpha(x)\, dx + O( N^{-1 + \delta}) .
 \eeq
 
\noindent We may then proceed using \eqref{rigidgoal1} as before in \eqref{gammaigammaisestimate} to complete the proof.

The proof of the second claim, \eqref{e:newrigidity}, uses \eqref{concentrationcounting} and proceeds similarly, except there is no need to treat the cases of even and odd $N$ separately. By \eqref{l:hgammarigid} and the discussion following \eqref{shownabove}, there exists $c_2>0$ such that $| i - N/2 | \le c_2 N$ implies $\gamma_i, \lambda_i(s) \in [-c_0, c_0]$ with overwhelming probability. We then write 
 \beq
 \int_{-\infty}^{\widehat \gamma_i(s)} d \widehat\mu_s =  \int_{-\infty}^{ \widehat \gamma_i(s) } d  \mu_s +  \int_{\widehat \gamma_i(s) } ^{ \lambda_i(s)} d  \mu_s.
 \eeq
 Then using \eqref{concentrationcounting}, with overwhelming probability we have 
\beq
N^{\delta / 2 - 1/2} \ge  \left| \int_{\widehat \gamma_i(s) } ^{ \lambda_i(s)} d \mu_s \right|.% \ge c' \big| \gamma_i(s)  - \widehat \gamma_i(s) \big|.
\eeq

\noindent Assuming to the contrary that $\big| \lambda_i(s)  - \widehat \gamma_i(s) \big| \ge N^{\delta - 1}$, we use \eqref{e:xtdispersion} again to show that, with overwhelming probability,
\begin{equation}
C N^{\delta / 2 - 1 / 2} \ge \left| \int_{\widehat \gamma_i(s) } ^{ \lambda_i(s)} d \mu_s \right| \ge c' \big| \lambda_i(s)  - \widehat \gamma_i(s) \big|.
\end{equation}

\noindent Thus, $\big| \lambda_i(s)  - \widehat \gamma_i(s) \big| \le C N^{\delta / 2 -1/2}$, which is a contradiction and so $\big| \lambda_i(s)  - \widehat \gamma_i(s) \big| \ge N^{\delta - 1}$. Finally, we estimate 
\begin{equation}
\big| \gamma_i(s)  - \widehat \gamma_i(s) \big| \le \big| \gamma_i(s)  - \lambda_i(s) \big| +  \big| \lambda_i(s)  - \widehat \gamma_i(s) \big|  \le N^{-1+\delta} + C N^{-1/2 + \delta}  \dom N^{-1/2 + \delta}
\end{equation}
using \eqref{e:xtrigidity} to bound $\big| \gamma_i(s)  - \lambda_i(s) \big|$, which proves  \eqref{e:newrigidity}
\end{proof}

\eer

\section{Convergence in distribution} \label{s:appendixdist}

\label{s:cvgdist}

\bep\label{p:dconvergence}

For $\alpha \in (2/3,2) \setminus \mathcal A $, there is a unique limit point $\vstar (E)$ of the sequence of random variables $\big\{\Im R_\star (E + \iu \eta) \big\}_{\eta > 0}$ in the weak topology. For $\alpha \in (1, 2) \setminus \mathcal A$, the conclusions of \Cref{t:thedistribution1}, \Cref{t:main2}, and \Cref{c:median} hold in the sense of convergence in distribution.
\eep

\begin{proof}
We begin by showing that there exist constants $C > 1 > c>0$ such that, for ${z \in \mathbb H}$ with $|z|<c$, the random variable $\rstar(z)$ satisfies the tail bound
\beq\label{e:startail}
\P\big( \Im \rstar( z) > s \big)  \le   \exp\left( - \frac{ s^{\alpha/ (2 - \alpha)} }{C}  \right).
\eeq

\noindent To that end, let $R_1 (z), R_2 (z), \ldots $ denote mutually independent random variables each with law $\rstar (z)$. By the L\'evy--Khintchine formula \eqref{e:lk},
\beq\label{markov2}
\E \Bigg[ \exp \bigg( - t   \Im \sum_{k=1}^\infty \xi_k R_k(z) \bigg) \Bigg]= \exp\Bigg(  - t^{\alpha/2} \Gamma\left( 1  - \frac{\alpha}{2} \right) \E \left[ \big(\Im \rstar (z) \big)^{\alpha/2}\right] \Bigg)\le \exp\left( - \frac{2 t^{\alpha/2}}{C}  \right),\eeq
for some $C >0$, where we used that $ \E \big[ \big( \Im \rstar(z) \big)^{\alpha/2} \big]  > c'$ for $z$ in a neighborhood of $0$ (see \eqref{e:eipi4}). 

We now compute, using \eqref{e:rde}, 
\begin{align}\label{markov1}
\P\big( \Im \rstar(z) > C t^{ 1 - \alpha/2}  \big)  &\le \P\left(  \Im \sum_{k=1}^\infty \xi_k R_k(z) < \frac{t^{\alpha/2 -1  } }{C} \right)  \\&= \P\left(  \exp \bigg( - t   \Im \sum_{k=1}^\infty  \xi_k R_k(z) \bigg)  > \exp\left({- \frac{t^{\alpha/2}}{C}}  \right)\right).
\end{align}
In the first equality, we used $\Im R_k(z) >0$. %This fact was shown in the proof of \cite[Proposition 3.3]{bordenave2017delocalization}. 
Now applying Markov's inequality to \eqref{markov2} and \eqref{markov1}, we obtain
\begin{align}\label{markov3}
\P\big( \Im \rstar(z) > C t^{ 1 - \alpha/2}  \big)  \le \exp\left( - \frac{t^{\alpha/2}}{C} \right).
\end{align}
Setting $s = C t^{ 1 - \alpha/2}$,  we obtain \eqref{e:startail} for a new value of $C$.

%For the convergence in distribution, we may argue as follows. %We discuss just the modification of \Cref{c:median}; the other claims are similar. %By \Cref{l:boundarylimit}, $\lim_{\eta \rightarrow 0} \E [\Im \rstar]^p$ exists for all $p\in \mathbb N$. 
%We recall the sequence of random variables ${\{ \Im \rstar(E + \iu \eta) \}_{\eta > 0}}$ is tight, by \Cref{p:tightness}. By Prokohorov's theorem, there exists a positive sequence $\{\eta_j\}_{j=1}^\infty$ with $\eta_j \rightarrow 0$ such that the subsequence $\{ \Im \rstar(E + \iu \eta_j) \}_{j=1}^\infty$ converges in distribution. Let $\mathcal R (E)$ be the subsequential limit. 
Let $\mathcal R_\star (E)$ be the limit point in the weak topology of ${\{ \Im \rstar(E + \iu \eta) \}_{\eta > 0}}$ from \Cref{d:arbitrary}.
\begin{comment}
Let $\{\eta_j\}_{j=1}^\infty$ with $\lim_{j \rightarrow \infty} \eta_j = 0$ be such that the subsequence $\{ \Im \rstar(E + \iu \eta_j) \}_{j=1}^\infty$ converges in distribution to $\mathcal R(E)$. Observe that \eqref{e:startail} implies the sequence $\Big\{ \big(\Im \rstar(E + \iu \eta)\big)^p \Big\}_{\eta >0}$ is uniformly integrable for any $p >0$. Then for any $p\in \mathbb N$, these facts
%\Cref{p:tightness}, and \Cref{l:boundarylimit}
imply
\beq\label{e:scriptRmoments}
\lim_{j \rightarrow \infty} \E \Big[ \big(\Im \rstar ( E + \iu \eta_j)\big)^p \Big] = \E \big[ \mathcal R (E)^p \big].
\eeq
\end{comment}
Note that since the tail bound \eqref{e:startail} holds for $\Im \rstar(z)$ uniformly in $z$, it also holds for
${\mathcal R_\star}(E)$. 

%The remainder of the argument is a standard application of Carleman's condition, which we now sketch. 
Using the bound \eqref{e:startail} for ${\mathcal R_\star}(E)$, we see that for all $k$, the $k$-th moment of ${\mathcal R_\star}(E)$ is bounded by $(Ck)^{k (2-\alpha) /\alpha}$ for some $C>0$. Therefore, the series $\sum_{k\ge 1 } \E \big[ {\mathcal R_\star}(E)^k  \big]^{-1/2k}$ diverges when $\alpha \in (2/3,2)$. By Carleman's condition for positive random variables (the Stieltjes moment problem) \cite[p. 21]{shohat1943problem}, this implies that ${\mathcal R_\star}(E)$ is determined by its moments when $\alpha \in (2/3 ,2) $. By \Cref{p:tightness}, these moments are the same for any subsequential limit of $\{ \Im \rstar(z) \}_{\eta >0}$, so ${\mathcal R_\star}(E)$ is only possible subsequential limit. %of this sequence. 
%Therefore $\mathcal R(E)$ is the same as the arbitrary limit point ${\mathcal R_\star}(E)$ chosen in \Cref{d:arbitrary}, and 
Therefore the sequence converges in distribution to ${\mathcal R_\star}(E)$.

 Similar reasoning may be applied to the quantities $\mathcal N^2 \cdot \mathcal R_\star(E)$ appearing in \Cref{t:thedistribution1}, \Cref{t:main2}, and \Cref{c:median}. By Stirling's formula, the $k$-th moment of $\mathcal N^2$ is bounded by  $(Ck)^k$ for some $C >0$, so the moments of $\mathcal N^2 \cdot \mathcal R_\star(E)$ are bounded by  $(Ck)^{ k(1 + (2-\alpha) /\alpha ) }$. For $\alpha \in (1,2)$, $1 + (2-\alpha) /\alpha < 2$ and Carleman's condition applies. This completes the proof.
\end{proof}

\section{Quantum unique ergodicity of eigenvectors} 

\label{matrixeigenvectors}

For any $a_N\colon [1, N] \cap\mathbb  N  \rightarrow [-1,1]$ we denote by $|a_N| = \big|  1 \le i \le N : a_N(i) \neq 0 \big|$ the cardinality of the integer support of $a_N$. We define $\langle \mathbf{u}_k, a_N \mathbf u_k \rangle = \sum_{i = 1}^N |\mathbf u_k(i)|^2 a_N(i)$.
\bec
For all $\alpha\in (0,2) \setminus \mathcal A$, there exists $c = c(\alpha)>0$ such that the following holds. Fix any index sequence $k=k(N)$ such that $\lim_{N\rightarrow\infty} \gamma_k = E$ for some $E \in \mathbb R$ satisfying $|E| <c$. Then for every $\delta >0$, for any  $a_N\colon [1, N] \cap \mathbb N  \rightarrow [-1,1]$ such that $\sum_{i=1}^\infty a_N(i) = 0$ and $|a_N| \rightarrow \infty$,
\beq
\P\Bigg( \left|  \frac{N}{|a_N|} \langle \mathbf u_k, a_N \mathbf u_k \rangle > \delta \right|  \Bigg) \rightarrow 0.
\eeq
\eec
\begin{proof}
	Letting $m_2 = \E \big[ \mathcal{U}_{\star} (E) \big]$, we compute
	\begin{multline}\label{e:secondmomentque}
	\E \bigg[  \Big( \frac{N}{|a_N|} \langle \mathbf u_k, a_N \mathbf u_k \rangle \Big)^2 \bigg] = \frac{1}{|a_N|^2}  \E \left[  \left(  \sum_{i = 1}^N a_N(i) \big( N |\mathbf u_k(i)|^2 - m_2 \big) \right)^2  \right]\\ \le \max_{\substack{i_1, i_2 \in [1, N] \\ i_1 \neq i_2}} \E \Big[ \big( N | \mathbf u_k(i_1)|^2 - m_2 \big) \big( N |\mathbf u_k(i_2)|^2 - m_2   \big) \Big]  + \frac{1}{|a_N|} \max_{i \in [1, N]} \E \left[  \left( N|{\mathbf u}_k(i)|^2 - m_2\right)^2 \right].
	\end{multline}
	The conclusion applies after applying Markov's inequality to the second moment computed in \eqref{e:secondmomentque} and applying \Cref{t:main2}. The hypothesis that  $|a_N| \rightarrow \infty$ ensures the second term in the second moment computation tends to zero.
\end{proof}

\bibliography{levy}
\bibliographystyle{abbrv}

\end{document}